\documentclass{amsart}

\usepackage{hyperref}
\usepackage{xcolor}

\usepackage[utf8]{inputenc}
\usepackage{csquotes}

\usepackage[style=alphabetic,maxalphanames=5,url=false,isbn=false, giveninits=true, maxnames=5]{biblatex}
\usepackage{graphicx}
\graphicspath{{./images/}}
\usepackage{float}

\usepackage{prob-style}
\usepackage[margin=1.1in]{geometry}
\usepackage[numbering=section]{my-theorems}
\usepackage{comment}
\usepackage{bm}

\usepackage{tikz}
\usetikzlibrary{arrows}
\usetikzlibrary{decorations.markings}

\newcommand{\im}{\mathrm{i} }
\newcommand{\bo}{\mathbf{1}}
\DeclareMathOperator{\Ai}{Ai}
\newcommand{\bv}{\mathbf{v}}
\DeclareMathOperator{\TW}{TW}

\newcommand{\dint}{\int}
\newcommand{\func}{\operatorname}

\newcommand{\gamhat}{\hat{\gamma}}
\newcommand{\psc}{p_{\rm SC}}
\DeclareMathOperator{\sgn}{sgn}

\newcommand{\mR}{\mathbb{R}}
\newcommand{\mP}{\mathbb{P}}
\newcommand{\tr}{\operatorname{tr}}
\newcommand{\bxi}{\bm{\xi}}

\newcommand{\fix}[1]{\textcolor{black}{#1}}

\def\mt{\tilde{m}}
\def\deltat{\tilde{\delta}}
\def\xit{\tilde{\xi}}

\numberwithin{equation}{section}

\title[Spin glass to paramagnetic transition]{Spin glass to
  paramagnetic transition \fix{and triple point in spherical SK model}}
\author[I.M. Johnstone]{Iain M. Johnstone}
\address{Department of Statistics, Stanford University}
\email{imj@stanford.edu}
\author[Y. Klochkov]{Yegor Klochkov}
\address{Faculty of Economics, University of Cambridge}
\curraddr{ByteDance}
\email{eklochov@gmail.com}
\author[A. Onatski]{Alexei Onatski}
\address{Faculty of Economics, University of Cambridge}
\email{ao319@cam.ac.uk}
\author[D. Pavlyshyn]{Damian Pavlyshyn}
\address{Department of Statistics, Stanford University}
\curraddr{Burnet Institute}
\email{damian.pavlyshyn@burnet.edu.au}

\date{\today}

\begin{document}

\maketitle

\tikzset
  {decoration=
     {markings,
=at position 0.5 with {\arrow{stealth}}
     },
   plain/.style={line width=0.8pt},
   arrow/.style={plain,postaction=decorate}
  }

%\tableofcontents

\begin{abstract}
This paper studies spin glass to paramagnetic transition in the
Spherical Sherrington-Kirkpatrick model with ferromagnetic Curie-Weiss
interaction with coupling constant
\textcolor{black}{$J$} 
%$J \in [0,1)$ 
and inverse
temperature $\beta$. The disorder of the system is represented by a
general Wigner matrix. We confirm a conjecture of \fix{Baik and Lee},
\cite{Baik2016} and
\cite{Baik2017}, that the critical window of temperatures for this
transition is \(\beta = 1 + bN^{-1/3} \sqrt{\log N}\) with
$b\in\mathbb{R}$.
\fix{The limiting distribution of the scaled free energy is Gaussian
  for negative $b$ and a weighted linear combination of independent
  Gaussian and Tracy-Widom components for positive $b$.
  In the special case where the Wigner matrix is from the Gaussian
  Orthogonal or Unitary Ensemble, we describe the triple point
  transition between spin glass, paramagnetic, and ferromagnetic
  regimes in a critical window for $(\beta, J)$ \textcolor{black}{around the triple point $(1,1)$}: the Tracy-Widom
  component is replaced by the one parameter family of deformations
  described by Bloemendal and Virag, \cite{BloVirI}.
}
% For $b\leq 0$, we derive a Gaussian limiting distribution of the free energy. As $b$ increases from $0$ to $\infty$, we describe the transition of the limiting distribution from Gaussian to Tracy-Widom. \textcolor{black}{Furthermore, in the special case where the Wigner matrix is from the Gaussian Orthogonal or Unitary Ensembles, we describe the triple point transition between spin glass, paramagnetic, and ferromagnetic regimes in the critical window \(\beta = 1 + bN^{-1/3} \sqrt{\log N}\) and $J=1-\omega N^{-1/3}$, $(b,\omega)\in\mathbb{R}^2$. For $b\leq 0$, the limiting distribution of the free energy is still Gaussian, whereas for $b>0$ it is a convolution of Gaussian with scaled distribution $\mathrm{BV}_{2/\alpha}(\omega)$, a member of one-parameter family converging to Tracy-Widom $\mathrm{TW}_{2/\alpha}$ when $\omega\rightarrow \infty$, and (after scaling) to Gaussian when $\omega\rightarrow -\infty$. The scale of $\mathrm{BV}_{2/\alpha}(\omega)$ in the convolution is proportional to $b$, so that as $b$ increases from $0$ to $\infty$, the Gaussian part becomes negligible.}
\end{abstract}

\section{Introduction}

\subsection{Set-up}

We study the large-\(N\) behavior of \textcolor{black}{partition functions, represented by} the spherical integrals
\begin{align}
  \label{eq:integral}
  \textcolor{black}{\mathcal{I}_{\alpha, J, N}}
  = \int_{\mathcal{S}^{N-1}_{\alpha}} \exp \bigl\{ \frac{N\beta}{\alpha} \cdot u^\star \textcolor{black}{W_{J,N}} u \bigr\} (\diff u),
\end{align}
with
\begin{align}
  \label{eq:spiked-model}
  \textcolor{black}{W_{J,N}}  = J \cdot w w^\star + W_N,
\end{align}
\textcolor{black}{where $w$ is an arbitrary $N$-dimensional unit-length vector and $W_N$ is an $N\times N$ random Wigner matrix. They are complex-valued if $\alpha=1$ and real-valued if $\alpha=2$.} In \eqref{eq:integral}, $\mathcal{S}_\alpha^{N-1}$ denotes \textcolor{black}{the corresponding} unit sphere, $(\diff u)$ denotes the normalized uniform measure over $\mathcal{S}_\alpha^{N-1}$, and symbol $^\star$ denotes combined transposition and complex conjugation.
We investigate the limiting distributions of the quantities
\begin{align}
  F_{\alpha, N}
  = \frac{\alpha}{2N} \log \textcolor{black}{\mathcal{I}_{\alpha,J, N}}
  \label{eq:free energy}
\end{align}
for \(\beta\) in the  ``critical regime'' of \(\beta = 1 + O(N^{-1/3}\sqrt{\log N})\), \textcolor{black}{and either a constant $J\in[0,1)$ or $J=1+O(N^{-1/3})$.} 

Our original motivation stems from the fact that integrals
\eqref{eq:integral} appear in the likelihood ratio in statistical
tests of
spiked models in multivariate statistics.
In such models, $J$ and
$\beta$ play the roles of the size of the spike under the null and
under alternative hypotheses, respectively. We discuss this
\fix{briefly} 
% in more detail 
at the end of this introduction.

% As we studied the statistical problem, we learned that
\fix{Expression} \eqref{eq:free energy} has an important physical interpretation. For $\alpha=2$ \textcolor{black}{and all entries of $w$ equal $N^{-1/2}$}, it is the free energy in the \fix{s}pherical Sherrington-Kirkpatrick (SSK) model with inverse temperature $\beta$ and ferromagnetic Curie-Weiss interaction with coupling constant \(J\). \textcolor{black}{The model is characterized by Hamiltonian
\begin{align} \label{eq:hamiltonian}
  H_N(\sigma)
  = \frac{1}{2}\Bigl(
   \sum_{i,j=1}^N W_{ij} \sigma_i \sigma_j + \frac{J}{N} \sum_{i,j=1}^N \sigma_i \sigma_j\Bigr)
  ,
\end{align}
where \(\sigma \in \sqrt{N}\mathcal{S}^{N-1}_2\) corresponds to a scaled version of $u$ in \eqref{eq:integral},  and \(W\) is a real symmetric \(N
\times N\) matrix with zeroes on the diagonal and independent upper
triangular entries \(W_{ij}\) with mean zero and variance~$1/N$.} It was
introduced \fix{by Kosterlitz et al.} in \cite{Kosterlitz1976} as a
tractable variant
of the original Sherrington-Kirkpatrick model that has discrete spins $\sigma \in \{\pm 1\}^{N}$.

\cite{Kosterlitz1976} show that \(F_{2,N}\) exhibits three distinct asymptotic regimes illustrated in \cref{fig:pase-diagram}. These regimes are defined by the value of $\max\{1,\beta^{-1},J\}$: 1 for the spin glass, $\beta^{-1}$ for the paramagnetic, and $J$ for the ferromagnetic.   
They correspond to distinct behavior of complex magnetic media, and
were extensively studied, 
see the recent works of \fix{Baik, Lee and Wu
\cite{Baik2016,Baik2017,BaikLeeWu}
(which are especially relevant to this work) and the references therein.}

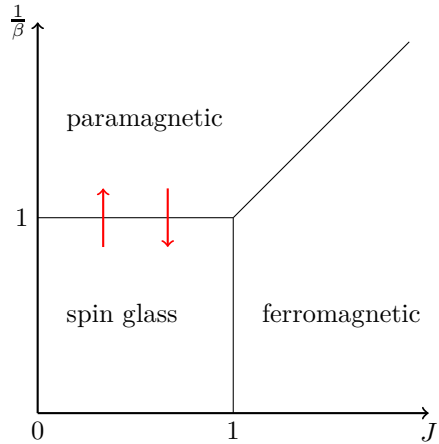
\begin{figure}
    \centering
    \begin{tikzpicture}[scale=1.3]
    \draw[thick, ->] (0,0) -- (4,0) node[anchor=north] {$J$};
    
    \draw	(0,0) node[anchor=north] {0}
    		(2,0) node[anchor=north] {1};
    
    \draw[thick, ->] (0,0) -- (0,4) node[anchor=east] {$\frac{1}{\beta}$};
    
    \draw   (0,2) node[anchor=east] {1};
    \draw[-] (0,2) -- (2,2);
    \draw[-] (2, 0) -- (2,2);
    \draw[-] (2, 2) -- (3.8,3.8);
    
    \draw   (0.2,1) node [anchor=west] {spin glass}
            (2.2, 1) node [anchor=west] {ferromagnetic}
            (0.2, 3) node [anchor=west] {paramagnetic};
    
    \draw[thick, red, <->] (1, 1.63) -- (1, 2.37);
    %\draw[thick, red, <-] (1.2, 1.7) -- (1.2, 2.3);
    %\draw[thick, red, <->] (1.5, 1.7) -- (1.5, 2.3);
    \draw[thick, blue, ->] (2,2) -- (1.72, 1.72);
    \draw[thick, blue, ->] (2,2) -- (1.85, 2.37);
    \draw[thick, blue, ->] (2,2) -- (2.37, 1.85);
\end{tikzpicture}
    \caption{Phase diagram showing the Spin glass, Paramagnetic, and Ferromagnetic regimes. Red arrows indicate the transition between Spin glass and Paramagnetic regimes, which is the focus of this paper. \textcolor{black}{Blue arrows represent the triple point transition} \textcolor{black}{between all the regimes, which we study in G(O/U)E cases.}}
    \label{fig:pase-diagram}
\end{figure}

  \textit{Transitions} between
the spin glass and ferromagetic regimes,
and between 
para- and ferro-magnetic regimes
have been established, see  \cite{Baik2017} and \cite{BaikLeeWu}
respectively. 
However
%at the time of writing,
%little is known about
the transition between the spin glass and
paramagnetic regimes 
has not been fully described, \textcolor{black}{whereas the triple point transition has not been described at all, to our knowledge.}
As \textcolor{black}{these transitions are} also of statistical interest, their
description is the goal of this paper.
% This paper describes such a transition. 

In the rest of this introduction, we  first provide brief
background on the three regimes in the SSK model. Then, we 
describe the main result of this paper. Finally, we return to our
original statistical motivation and discuss connections to
this paper. 

\subsection{\textcolor{black}{The three regimes}}
\fix{Baik, Lee and Wu},
\cite{Baik2016,Baik2017,Baik2017cor,BaikLeeWu}, make a thorough study
of the fluctuations of the free energy \(F_{2,N}\) in the spin glass,
para- and ferro-magnetic regimes.\footnote{In these papers, the  
parameterization is slightly different from ours, so that
%authors lack the \(1/2\)-scaling for the Hamiltonian.
%The main difference of this scaling is that 
in their case, the critical threshold is \(\beta = 1/2\) %where in our parametrization, the threshold is 
instead of \(\beta = 1\).}
The fluctuations of the free energy in the three regimes are shown to be
\begin{enumerate}
\item
  (Spin glass) If \(\beta > 1\) and \(J < 1\), then
  \begin{align*}
    \frac{2N^{2/3}}{\beta - 1}(F_{2,N} - F(\beta))
    \drightarrow \TW_1.
  \end{align*}
  
\item
  (Paramagnetic) If \(\beta < 1\) and \(\beta < 1/J\), then
  \begin{align*}
    N(F_{2,N} - F(\beta))
    \drightarrow \Normal(f_1, a_1),
  \end{align*}
  where \(a_1\) depends on \(\beta\) but not on \(J\), while \(f_1\)
  depends on both \(\beta\) and \( J\). \fix{Specifically for GOE,
  \begin{equation} \label{eq:a1f1}
    f_1 = \tfrac{1}{4}\log(1-\beta^2) - \tfrac{1}{2} \log(1-\beta J),
    \qquad
    a_1 = - \tfrac{1}{2} \log(1- \beta^2).
  \end{equation}}
  % and the moments of
  % \(A_{ij}\), % and the moments of
  % \(A_{ij}\). 

\item
  (Ferromagnetic) If \(J > 1\) and \(\beta > 1/J\), then
  \begin{align*}
    \frac{N^{1/2}}{\textcolor{black}{\beta-1/J}}(F_{2,N} - F(\beta))
    \drightarrow \Normal(\textcolor{black}{f_2}, a_2),
  \end{align*}
  where \textcolor{black}{$f_2$ and } \(a_2\) depend
  \textcolor{black}{only on $J$.}\footnote{\fix{$f_1, a_1$, and $f_2$ also
    depend on subsets for the first four moments of the atom
    distributions of $W_N$}}
  % on \(\beta\).% and the moments of \(A_{ij}\).
\end{enumerate}

\ The leading order term $F(\beta)$ differs across the regimes:
\begin{equation}
\label{eq:leading term}
    F(\beta)=\left\{
    \begin{aligned}
    &\beta-\frac{1}{2}\log\beta-\frac{3}{4}& & \text{for spin glass}\\
    &\frac{1}{4}\beta^2 & &\text{for paramagnetic}\\
    &\frac{\beta}{2}(J+1/J)-\frac{1}{2}\log(\beta J)-\frac{1}{4}J^{-2}-\frac{1}{2} & &\text{for ferromagnetic.}
    \end{aligned}
    \right.
\end{equation}

These results characterize the fluctuations of \(F_{2,N}\) in models lying strictly within the \textcolor{black}{three} regimes.
The results for the transitions studied in 
\cite{Baik2017} and \cite{BaikLeeWu} \textcolor{black}{can be summarized as follows.
\begin{enumerate}
\item[(1 $\leftrightarrow$ 3)] (Spin glass $\leftrightarrow$ Ferromagnetic) For $\beta>1$ and $J=1-\omega N^{-1/3}$ with $\omega\in\mathbb{R}$
\[
N^{2/3}(F_{2,N}-F(\beta))\drightarrow\frac{\beta-1}{2}\textcolor{black}{BV_1(\omega)},
\]
where \textcolor{black}{$F(\beta)$ is as defined in \eqref{eq:leading
    term} for the spin glass regime, and $BV_1(\omega)$} denotes a
one-parameter family of distributions described in theorems 1.5 and
1.7 of \cite{BloVirI}, \fix{see Sec. \ref{sec:convergence-at-edge}.}
\item[(3 $\leftrightarrow$ 2)] (Ferromagnetic $\leftrightarrow$ Paramagnetic) For $J>1$ and $\beta=1/J+BN^{-1/2}$ with $B\in\mathbb{R}$
\[
N\left(F_{2,N}-F(\beta)\textcolor{black}{-\frac{\log N}{4 N}}\right)\drightarrow G_1+Q_B(G_2),
\]
where \textcolor{black}{$F(\beta)$ is as defined in \eqref{eq:leading term} for the ferromagnetic regime, } \((G_1, G_2)\) has a bivariate Gaussian distribution that depends on $J$ but not on $B$, and \(Q_B\) is a non-linear function that depends both on $J$ and \(B\).
\end{enumerate}}

Concerning the remaining
transition between the spin glass and the paramagnetic regimes,
\cite{Baik2016} and \cite{Baik2017} conjecture that the critical window of temperatures for this transition is \(\beta = 1 + O(N^{-1/3} \sqrt{\log N})\) for any $J<1$. They arrive at this conjecture by matching the orders of the variance of \(F_{2,N}\) as \(\beta \rightarrow 1\) from above and below. In this paper we confirm that the conjecture is correct, and describe the asymptotic behavior of \(F_{\alpha,N}\) in the critical temperature window. 

\subsection{\fix{Results}}

\textcolor{black}{Our first main result shows that if}
\begin{align*}
  \beta = 1 + b N^{-1/3} \sqrt{\log N},
  \qquad
  0 \leq J < 1,\quad \textcolor{black}{b\in\mathbb{R}},
\end{align*}
then \(F_{\alpha, N}\) has fluctuations of order 
\(\sqrt{\log N}/N\).
Moreover, as \(b\) increases from \(-\infty\) to \(\infty\), we describe the transition of the limiting distribution of \(F_{\alpha, N}\) from Gaussian to the Tracy-Widom.

\begin{theorem} \label{thm:main}
  Consider $F_{\alpha,N}$ with $\alpha=1$ or $\alpha=2$ as defined in
  \eqref{eq:integral} -- \eqref{eq:free energy}. \textcolor{black}{Let $W_N$ from \eqref{eq:spiked-model} be a Wigner matrix whose off-diagonal moments match scaled GOE ($\alpha=2$) or GUE ($\alpha=1$) up to third order.}
  Further, let \(\beta = 1 + b N^{-1/3} \sqrt{\log N}\) with a constant \(b \in \R\) and let \(0 \leq J < 1\).
  Finally let \(b_+ = \max\{0, b\}\)
  be the positive part of $b$.
  Then  
  \begin{align}
  \label{eq:main thm}
    \frac{N}{\sqrt{\frac{\alpha}{12} \log N}}\left(F_{\alpha,N} - F(\beta)+\frac{\log N}{12 N}\right)
    \drightarrow \Normal(0, 1) + \sqrt{\frac{3}{\alpha}}b_+ \mathrm{TW}_{2/\alpha},
  \end{align}
  where \(\mathrm{TW}_2\) and \(\mathrm{TW}_1\) are the complex and real Tracy-Widom distributions, respectively, independent from the $\mathcal{N}(0,1)$, and where $F(\beta)$ is as in \eqref{eq:leading term}, that is
  \begin{align*}
    F(\beta)
    = \left\{
    \begin{aligned}
    &\beta-\frac{1}{2}\log\beta-\frac{3}{4} & &\text{for } b\geq 0 \\
    &\frac{1}{4}\beta^2 & &\text{for } b\leq 0.
    \end{aligned}
    \right.
  \end{align*}
\end{theorem}

For definitions of Wigner matrices, scaled G(O/U)E, and the matching moment condition see \cref{sec: definitions}. \textcolor{black}{To avoid confusion, note that we use the $\alpha$-parameterization for G(O/U)E: $\alpha=2$ for GOE and $\alpha=1$ for GUE. It comes from the literature linking these ensembles to Jack polynomials (e.g. \cite{dumitriu07}). We do not use the more familiar parameterization $2/\alpha=\beta$ because $\beta$ has been already used as the inverse temperature parameter.}

As can be seen from the theorem, the fluctuations of $F_{\alpha,N}$
remain Gaussian for negative $b$. However, as $b$ changes sign
%to positive 
and enters the spin glass domain from the paramagnetic
domain, the fluctuations acquire an independent Tracy-Widom component,
which  eventually dominates the Gaussian as $b$ diverges to $+\infty$.

\fix{The shift $-\log N/12N$ in $F_{\alpha,N}$ and the scale
  $\sqrt{\frac{1}{6} \log N}$ (for $\alpha = 2$) can be `predicted'
   by formal substitution of $(\beta, J)$
  in the paramagnetic regime formulas (\ref{eq:a1f1}).
}

To prove the independence of TW$_{2/\alpha}$ and $\mathcal{N}(0,1)$ components of the limit in \eqref{eq:main thm} we establish the asymptotic independence of the largest eigenvalue of $W_N$ and the log determinant $\log|\det(W_N-2)|$. When obtaining such a result initially for G(O/U)E, we use techniques similar to those developed in \cite{johnstone2020logarithmic}.
%We prove the independence of TW$_{2/\alpha}$ and $\mathcal{N}(0,1)$ components of the limit in \eqref{eq:main thm} by establishing the asymptotic independence of the largest eigenvalue of \textcolor{black}{$W_N$ and the log determinant $\log|\det(W_N-2)|$} 
%the GUE or GOE matrix $Z_\alpha$ and the log determinant $\log|\det(2-Z_\alpha/\sqrt{N})|$ 
%using techniques similar to those developed in \cite{johnstone2020logarithmic}. 
Specifically, we first 
consider the tri-diagonal form of $N\times N$ G(O/U)E and prove that, asymptotically, the largest eigenvalue depends only on the lower-right corner of dimension $N^{1/3}\log^{3}N$, while the log determinant asymptotically depends only on the complementary upper-left corner. 

The latter result has some independent interest, as it
gives theoretical grounding to the computational technique (described,
for example, in \cite{edelman2013}), wherein the largest eigenvalue of
an $N$-dimensional tri-diagonal matrix with huge $N$ is computed from
its order \(10N^{1/3}\) minor.

\textcolor{black}{ Our second main result describes the triple point transition in the critical window
\[
\beta = 1 + b N^{-1/3} \sqrt{\log N},\qquad J=1-\omega N^{-1/3},\quad (b,\omega)\in\mathbb{R}^2
\]
for special cases where $W_N$ belongs to G(O/U)E. In this setting, $F_{\alpha,N}$ still has fluctuations of order $\sqrt{\log N}/N$. However, the limiting distribution is a convolution of Gaussian with distribution $\mathrm{BV}_{2/\alpha}(\omega)$, described in theorems 1.5 and 1.7 of \cite{BloVirI}.} 

\textcolor{black}{Precisely, we establish the following result.}

\textcolor{black}{\begin{theorem}
\label{thm: triple}
    In the setting of \ref{thm:main}, let $J=1-\omega N^{-1/3}$ with a constant $\omega\in\mathbb{R}$ instead of $J\in[0,1)$. Furthermore, let $W_N$ be from GOE ($\alpha=2$) or GUE ($\alpha=1$). Then
     \begin{equation}
    \label{eq:main thm'}
    \frac{N}{\sqrt{\frac{\alpha}{12} \log N}}\left(F_{\alpha,N} -
          F(\beta)       \textcolor{black}{-\frac{\log N}{12 N}} \right) 
        \drightarrow \Normal(0, 1) + \sqrt{\frac{3}{\alpha}}
        \textcolor{black}{b_+ \mathrm{BV}_{2/\alpha}(\omega).} 
  \end{equation}
\end{theorem}}

\textcolor{black}{A few remarks are in order. First, the \fix{change in} sign 
  of $\log N/(12 N)$ on the left hand side of \eqref{eq:main
    thm'} (cf. \eqref{eq:main thm})} \fix{reflects an extra shift $\log
    N/6N$ in the free energy during the additional transition from $J<1$
  to $J>1$. This shift is again `predicted' by formal
    substitution of the critical values of $(\beta, J)$ in the term
  $-\frac{1}{2}\log(1-\beta J)$ of (\ref{eq:a1f1}).}
%  is not a typo. An extra shift of the free energy happens

\textcolor{black}{Second, as $\omega\rightarrow \infty$, \cite{BloVirI} show that $\mathrm{BV}_{2/\alpha}(\omega)\drightarrow\mathrm{TW}_{2/\alpha}$, consistent with Theorem \ref{thm:main}
in the spin glass region. As $\omega\rightarrow -\infty$, \cite[Th.~4.1.1]{Bloemendal11} shows that 
\[
\frac{\mathrm{BV}_{2/\alpha}(\omega)+\omega^2}{\sqrt{|\omega|}}\drightarrow\mathcal{N}(0,2\alpha),
\]
consistent with the Gaussian limit for $F_{\alpha,N}$ in the ferromagnetic region.}

\textcolor{black}{Third, in the interior of the para- and ferro-magnetic regions, the scaled limiting law of $F_{\alpha,N}$ is Gaussian, and the limit in Theorem \ref{thm: triple} is consistent with this for $b\rightarrow -\infty$ or $\omega \rightarrow -\infty$. Near the critical line $\beta = 1/J < 1$, if we formally set $J=1-\omega N^{-1/3}$ in the non-Gaussian limit
$G_1 + Q_B(G_2)$ of \cite{BaikLeeWu}, then
the Gaussian term $G_1$ dominates.
Indeed, by Theorem 1.2 of \cite{BaikLeeWu}, $G_1= -\frac{1}{12} \log N +
\sqrt{\frac{1}{6}\log N} Z_1+O_\Pr(1)$  for standard normal $Z_1$, while the
stochastic part of $Q_B(G_2)$ is of smaller order, namely $O_\Pr(1)$.} \textcolor{black}{Here and in what follows, the notation $x=O_\Pr(1)$ means ``bounded in probability'', that is, for any small $\epsilon>0$ there exists $C>0$ such that $\Pr\{|x|>C\}\leq \epsilon$ for all sufficiently large $N$.}

%\textcolor{black}{
%We prove \cref{thm:main} by first considering the GOE and GUE cases (Sections \ref{sec:negative-critical}-\ref{sec: remaining}) and then using the Lindeberg swapping process to extend to Wigner matrices in Section \ref{sec:wigner}. 

When the initial version of this paper \cite{jkop21}
was close to completion, we learned about the related
study \cite{Landon2020},
later published \cite{Lan22}.
That paper \textcolor{black}{considers $W_N$ from GOE}, and establishes the Gaussian
fluctuation limit for $F_{2,N}$ in the case $b\leq 0$ as in
\cref{thm:main}. \textcolor{black}{It obtains the same Gaussian limit for $b\rightarrow 0$ from above, and shows that the limit becomes Tracy-Widom for $b\rightarrow\infty$ at any rate. For fixed} $b>0$, it establishes only the tightness
of the left hand side of \eqref{eq:main thm}, and conjectures that the
limiting distribution exists and equals a sum of independent normal
and Tracy-Widom distributions. Our result confirms that conjecture \textcolor{black}{not only for GOE, but for general symmetric or Hermitian Wigner matrices $W_N$ whose moments match GOE/GUE up to third order}.

\textcolor{black}{
  In determining the limiting fluctuations,
  % \cite{Landon2020}
  \cite{Lan22} uses
recent results of \cite{lambert2020b}, who deal exclusively with Gaussian beta ensembles, in the case of interest here it is GOE and GUE without a spike (i.e. $J = 0$). Instead, we rely on 
% our work [JKOP20]" 
% We do not rely on those results in our proofs,
% using
central limit theorems for logarithmic spectral
statistics established in
another paper of ours \cite{johnstone2020logarithmic},
and recalled here in theorems \ref{CLT1} and \ref{CLT2}. These results hold for general Wigner matrices with a spike in sub-critical region $J \in [0, 1)$. In addition, our results  allow arbitrary variance profile on the diagonal. This is motivated by the fact that the formulation of the SSK model often considers zero diagonal interactions, see e.g. 
\cite{Kosterlitz1976},\cite{Talagrand2006}.}

\textcolor{black}{Finally, \cite{Lan22} considers the
  case $J=0$, \fix{and so}  does not address the triple
  point transition, theorem \ref{thm:
    triple}. Another recent paper that considers only the case $J=0$,
  but in the  bipartite SSK model framework, is \cite{collins23}.}

%\subsection{Phase transition}
\subsection{Statistics background}
Random matrices of the form \cref{eq:spiked-model}, \textcolor{black}{and in particular}
%-- that is, 
matrices that differ %in some sense 
from a white Gaussian or Wishart random matrix 
by a low-rank deviation have been extensively studied in the statistics literature, where their distributions are known as spiked ensembles.
In this context, the parameter \(J\) is known as the spike.

Our interest in the free energy \(F_{\alpha, N}\) stems from its appearance in problems of statistical testing for spiked random matrix models.
As discussed in \cite{Onatski2014a} and cataloged for a much larger family of spiked models in \cite{Johnstone2015}, the joint density of the eigenvalues $\Lambda$ of both spiked Gaussian and spiked Wishart ensembles with a spike of size \(h\) are of the form
\begin{align*}
  p_N(\Lambda; h)
  &= c(\Lambda) d(h) \int_{\mathcal{S}^{N-1}_{\alpha}} \exp \Bigl\{ \frac{N}{\alpha} h \cdot u^\star \Lambda u \Bigr\} (\diff u)
%  &= c(\Lambda) d(h) \pFq{0}{0}^{(\alpha)} \Bigl(\frac{N}{2} h \cdot e_1 e_1^\star, \Lambda\Bigr)
\end{align*}
for some functions \(c\) and \(d\).
This demonstrates the close relationship between \(F_{\alpha, N}\) and the log-likelihood ratio for testing simple hypotheses about \(h\), as well as the close relationship between spiked Gaussian and Wishart models.

Indeed, when \(J = 0\), \(F_{\alpha, N}\) is distributed as the scaled log-likelihood ratio for testing
\begin{align*}
  H_0\colon h = 0
  \qquad
  \text{vs. }
  \qquad
  H_1\colon h = \beta
\end{align*}
in the spiked Gaussian model, under the null hypothesis.
Specifically,
for $\beta \leq 1$ 
\begin{equation*}
  \log \frac{p_N(\Lambda; \beta)}{p_N(\Lambda; 0)}
    = \frac{2N}{\alpha}
    [F_{\alpha,N} - F(\beta)]. 
\end{equation*}
Theorem \ref{thm:main} therefore gives the limiting behavior of the null
distribution of the likelihood ratio. The mean shift and 
variance, both growing of order $\log N$, verify (as is expected from discussion
in \cite{Johnstone2015}, which focuses on sub-critical cases with
$\beta < 1$ fixed) that the
null and alternative distributions fail to be contiguous, and so we
cannot (as there) directly obtain the limiting distribution of the
likelihood ratio under  alternative hypotheses $\beta$ near $1$.

%\newpage
\section{\textcolor{black}{Definitions, preliminary results, and proof strategy}}
\label{sec:some-prel-results}

\subsection{Definitions}
\label{sec: definitions}

\begin{definition}
\label{Wigner definition}
\textcolor{black}{An $N\times N$ Wigner matrix is an Hermitian matrix
$W_N=(\xi_{ij}/\sqrt{N})$ satisfying
\begin{enumerate}
\item[(i)] the upper-triangular components
$\{\operatorname{Re}\xi_{ij},\operatorname{Im}\xi_{ij}\}_{i<j}$ and 
$\{\xi_{ii}\}$ are independent
random variables with mean zero,
\item[(ii)] $\mathbf{E}|\xi_{ij}|^2=1$ for $i\neq j$ and
$\mathbf{E}\xi_{ii}^2\leq B$ for some absolute constant $B$;
\item[(iii)]a moment bound uniform in $N$: for all \(p \in \Z_{>0}\), there
is a constant \(C_p\) such that 
\begin{equation*} %\label{eq:W3}
  \E \abs{\Re\xi_{ij}}^p,
  \E \abs{\Im\xi_{ij}}^p 
  \leq C_p.
\end{equation*}
\end{enumerate}
This definition is standard, e.g.
\cite[][Def 2.2]{BenaychG2018}, except that we also require
independence of $\Re\xi_{ij}$ and $\Im \xi_{ij}$ to simplify our
%swapping 
arguments.
Condition (ii) allows for zero variances on the diagonal, as in the
SSK model of \cite{Kosterlitz1976}. }
\end{definition}

\textcolor{black}{In what follows, we will consider Hermitian complex-valued $W_N$ when $\alpha=1$ and symmetric real-valued $W_N$ when $\alpha=2$.
%When we do not distinguish between the cases, we refer to such $W_N$ as a Wigner matrix with inverse Dyson parameter $\alpha$. 
An important example of a Hermitian/symmetric Wigner matrix is a matrix from scaled G(O/U)E. For the reader's convenience, we recall here the definitions of these classical ensembles.} 
\begin{definition}[GUE and GOE]
For \(1 \leq i \leq j \leq N\), let \(\xi_{ij}\), \(\eta_{ij}\) be independent \(\Normal(0, 1)\) random variables.
Then define a Hermitian matrix \(Z_1\) with entries
  \begin{align*}
    Z_{1,ij}
    &=
      \begin{piecewise}
        \xi_{ij} &\text{if } i = j, \\
        \frac{1}{\sqrt{2}} (\xi_{ij} + \im \eta_{ij}) &\text{if } i < j, \\
        \overline{Z_{1, ji}} &\text{if } i > j.
      \end{piecewise}
  \end{align*}
  Similarly, define a symmetric matrix \(Z_2\) by
  \begin{align*}
    Z_{2, ij}
    &=
      \begin{piecewise}
        \sqrt{2} \xi_{ii} &\text{if } i = j, \\
        \xi_{ij} &\text{if } i < j \\
        Z_{2, ji} &\text{if } i > j.
      \end{piecewise}
  \end{align*}
  We call the distribution of \(Z_1\) the Gaussian Unitary Ensemble (GUE), and that of \(Z_2\) the Gaussian Orthogonal Ensemble (GOE).
  %Generically, when we don't need to distinguish between the cases, we will call the distribution of $Z_\alpha$ the G$\alpha$E.
\end{definition}

If $Z_\alpha$ is an $N\times N$ G(O/U)E matrix, then we call $Z_\alpha/\sqrt{N}$ a scaled G(O/U)E matrix. \textcolor{black}{After the scaling, matrices from G(O/U)E become special cases of Wigner matrices as defined above.}

\begin{definition}[Spiked Wigner Matrix]
\textcolor{black}{We call matrix $W_{J,N}=Jww^\star+W_N$ (see \eqref{eq:spiked-model}) a $J$-spiked Wigner matrix. We call it sub-critically spiked if $J<1$. Sometimes, we refer to $W_{J,N}$ as a spiked $W_N$ or a spiked version of $W_N$.} 
\end{definition}

\textcolor{black}{\begin{definition}[Moment matching]
The off-diagonal moments of two Wigner matrices $W_N, W_N'$
\textit{match to order $m$} if for  integer $0 < a \leq m$
\begin{equation*}
  \E (\Re \xi_{ij})^a = \E (\Re \xi_{ij}')^a, \qquad
  \E (\Im \xi_{ij})^a = \E (\Im \xi_{ij}')^a
\end{equation*}
for all $1 \leq i < j \leq N$.
\end{definition}}
% \textcolor{black}{\begin{definition}[Critically spiked G(O/U)E] We call matrix $W_{J,N}^{(\omega)}=Jww^\ast+W_N$ a \textit{critically spiked} G(O/U)E if $W_N=Z_\alpha/\sqrt{N}$ is scaled G(O/U)E and $J=1-\omega N^{-1/3}$.
% \end{definition}}

\noindent\textbf{\textcolor{black}{Some notations.}} The notation $a_N \lesssim b_N$ means that $a_N \leq C b_N$ for some
$C$ and $N$ large. \textcolor{black}{The notation $a_N\asymp b_N$ means that $a_N \lesssim b_N$ and $b_N \lesssim a_N$.} 
We say that $a_N$ is a $\Theta_{\Pr }(1)$ variable if $a_N$ is
a.s.~positive and $a_N, a_N^{-1}$ are $O_{\Pr}(1)$.
We say that events $\mathcal{E}_N$ hold asymptotically almost surely
(a.a.s.) if $\Pr(\mathcal{E}_N) \to 1$ as $N \to
\infty$. \textcolor{black}{We say that $\mathcal{E}_N$ hold with
  overwhelming probability (w.o.p.) if
  $\Pr(\mathcal{E}_N)=1-O(N^{-c})$ for \fix{each} $c>0$. The term ``with high probability'' means $P(\mathcal{E}_N) =
1 - O(N^{-c})$ for some $c > 0$. Notation $\drightarrow$ indicates convergence in distribution. }

\subsection{\textcolor{black}{Preliminary results}}\label{sec:prelim}

Our analysis is based on the now well known 
contour integral representation of \(\mathcal{I}_{\alpha, J, N}\),
\textcolor{black} {reviewed in} \cref{sec:remarks-cont-repr}:
\begin{align}
  \mathcal{I}_{\alpha, J, N}
   = \frac{C_{\alpha, N}}{2 \pi \im} \int_{\mathcal{K}}
  \exp\{(N/\alpha) G(z)\} \diff z,\label{eq for z}
  \qquad
  G(z)
 = \beta z - \frac{1}{N}\sum_{j=1}^N \log(z - \lambda_{j}),
\end{align}
where for now the integration contour $\mathcal{K}$ is the vertical line from
$\gamma - \im \infty$ to $\gamma + \im \infty$ for any constant
$\gamma > \lambda_{1}$,  
$\lambda_{1} \geq \cdots \geq \lambda_{N}$ are the
eigenvalues of \textcolor{black}{$W_{J,N}$}, and 
\begin{equation*}
    C_{\alpha, N}
    = \frac{\Gamma(N/\alpha)}{ (\beta N/\alpha)^{N/\alpha - 1}}.  
\end{equation*}

Notice that the integrand is an analytic function in $\C\setminus (-\infty,\lambda_{1}]$ and that the integral along the circular arc
\[
  C_{R, K}
  = \{z \in \C : \abs{z} = R, \Re(z) \leq K\}
\]
satisfies, for large enough \(R\),
\begin{align*}
    \abs[\bigg]{\int_{C_{R,K}} \exp\{(N/\alpha) G(z)\} \diff z}
    \leq 2 \pi R \cdot \frac{e^{N \beta K/\alpha}}{(R/2)^{N/\alpha}}
    \xrightarrow{R\rightarrow \infty} 0.
\end{align*}
In particular, Cauchy's theorem implies that $\mathcal{K}$ can be deformed
without affecting the value of the integral as long as
$\lambda_{j}$ are never intersected and as long as the
resulting contour has real part bounded above.

\textcolor{black}{Many of our technical arguments involve properties of the logarithmic statistic entering $G(z)$ and its derivatives at various points $z$. In this subsection, we collect important preliminary results that concern such properties. We will make the following assumption\fix{s}.}
\vspace{0.2cm}

\noindent \textbf{Assumption W.} Suppose $W_N$ is a Wigner matrix
whose off-diagonal moments match scaled GUE ($\alpha=1$) or GOE
($\alpha=2$) up to third order. Let $W_{J,N}$ be a sub-critically
spiked version of $W_N$
% , and let $\lambda_1\geq\cdots\geq\lambda_N$ be
% the eigenvalues of $W_{J,N}$.
% \textcolor{black}{In contrast, we will denote the eigenvalues of a critically spiked G(O/U)E matrix $W^{(\omega)}_{J,N}$ as $\lambda_1^{(\omega)}\geq\cdots\geq\lambda^{(\omega)}_N$.}
%\vspace{0.2cm}

\medskip
\noindent \fix{\textbf{Assumption G$^\omega$.} [Critically spiked G(O/U)E]
We call matrix $W_{J,N}^{(\omega)}=Jww^\ast+W_N$ a \textit{critically
  spiked} G(O/U)E if $W_N=Z_\alpha/\sqrt{N}$ is scaled G(O/U)E and
$J=1-\omega N^{-1/3}$.}

\vspace{0.2cm}

%\subsubsection{CLTs for the \textcolor{black}{logarithmic statistic.}} \ 
We will need the following two central limit theorems, \fix{established in
\cite{johnstone2020logarithmic} for Case W and in Section
\ref{sec:proof of CLT} for case G$^\omega$.
Let $\lambda_1\geq\cdots\geq\lambda_N$ be the eigenvalues of $W_{J,N}$
and $W^{(\omega)}_{J,N}$ in both cases.}

\begin{theorem}
\label{CLT1}
Suppose Assumption W holds. Let $\gamma =2+CN^{-2/3}\log N$, where $C>0$ is an arbitrary constant, \textcolor{black}{and let
\[
\tau_N=\sqrt{\frac{\alpha}{3}\log N},\qquad \mu_N=\frac{1}{2}N+CN^{1/3}\log N-\frac{2}{3}(C\log N)^{3/2}-\frac{\alpha-1}{6}\log N.
\]}
Then
\begin{equation} \label{eq:T2.5}
%\frac{\sum_{j=1}^{N}\log \left( \gamma -\lambda _{j}\right)
%-N/2-N^{1/3}C\log N +\left( 2/3\right) \left( C\log N\right) ^{3/2}+\frac{\alpha-1}{6}\log N}{\sqrt{%
%\frac{\alpha}{3}\log %N}}\drightarrow\Normal(0,1) .
\frac{\sum_{j=1}^{N}\log \left( \gamma -\lambda _{j}\right)-\mu_N}{\tau_N}\drightarrow\mathcal{N}(0,1).
\end{equation}
\fix{Under Assumption G$^\omega$, the result holds with $\mu_N$
  replaced by $\mu_N - \frac{1}{3} \log N$.}
% \textcolor{black}{Furthermore, for the critically spiked G(O/U)E with $J=1-\omega N^{-1/3}$, $\omega\in\mathbb{R}$, we have
% \[
% \frac{\sum_{j=1}^{N}\log \left( \gamma -\lambda^{(\omega)} _{j}\right)-\mu_N+\frac{1}{3}\log N}{\tau_N}\drightarrow\mathcal{N}(0,1).
% \]}
\end{theorem}

\begin{theorem}
\label{CLT2}
\textcolor{black}{Suppose Assumption W holds. Let $\gamma =2+CN^{-2/3}$, where $C\in \mathbb{R}$ is an arbitrary constant,} \textcolor{black}{let $\tau_N$ be as defined in theorem \ref{CLT1}, and let
\[
\mu_N'=\frac{1}{2}N+CN^{1/3}-\frac{\alpha-1}{6}\log N.
\]}
Then
\begin{equation*}
\frac{\sum_{j=1}^N \log \left\vert \gamma -\lambda_{j}\right\vert -\mu_N'}{\tau_N}\drightarrow\Normal(0,1) .
\end{equation*}
\fix{Under Assumption G$^\omega$, the result holds with $\mu_N$
  replaced by $\mu_N - \frac{1}{3} \log N$.}
% \textcolor{black}{Furthermore, for the critically spiked G(O/U)E with $J=1-\omega N^{-1/3}$,  $\omega\in\mathbb{R}$, we have
% \[
% \frac{\sum_{j=1}^{N}\log | \gamma -\lambda^{(\omega)} _{j}|-\mu_N'+\frac{1}{3}\log N}{\tau_N}\drightarrow\mathcal{N}(0,1).
% \]}
\end{theorem}

% \textcolor{black}{For the sub-critically spiked general Wigner matrices, theorems \ref{CLT1} and \ref{CLT2}  are established in \cite{johnstone2020logarithmic}. In section \ref{sec:proof of CLT}, we show how to extend these theorems to the critically spiked G(O/U)E cases.}

\textcolor{black}{All the remaining preliminary results described in
  \fix{the next two sub-sections} are established \fix{in the Appendix,}
  sections \ref{sec:proof key lemmas} and \ref{sec:wigner}.
  First, they are proved for the special case of $W_N$ being a scaled
  G(O/U)E and \textcolor{black}{$J=0$ (no spike) or $J=1-\omega
    N^{-1/3}$ (critical spike)} in \cref{sec:proof key lemmas},
  \fix{ using known sharp bounds for the one-point \textcolor{black}{correlation} function both in
    the bulk and \textcolor{black}{at}  the edge.}
  Then, \cref{sec:wigner} extends the proof to sub-critically spiked Wigner matrices $W_{J,N}$ satisfying Assumption W using the Lindeberg swapping technique.}

\subsubsection{Convergence at the edge.}
\label{sec:convergence-at-edge}
We will rely on the properties of the top eigenvalues of sub-critically spiked Wigner matrices \textcolor{black}{and critically spiked G(O/U)E}. For the special cases of scaled G(O/U)E and $J=0$ \textcolor{black}{(no spike)}, the celebrated papers \cite{trwi94,trwi96,dieng05}
showed that, for each fixed $j$, the
scaled eigenvalues $N^{2/3}(\lambda_j - 2)$ converge in law to the
$j$-th Tracy-Widom distribution,
$\TW_{2/\alpha,j}$. \textcolor{black}{For critically spiked G(O/U)E,
  \cite[thm.~1.5 and prop.~2.8]{BloVirI} show
  \fix{finite dimensional convergence of the scaled eigenvalues to the
    (negative of the) eigenvalues of the stochastic Airy operator
    \textcolor{black}{$\mathcal{H}_{2/\alpha,\omega}$} with
    Robin boundary condition,}
%   that jointly for $j=1,2,...$ in the sense of convergence in finite dimensional distributions,
% \[
% N^{2/3}(\lambda_j^{(\omega)} - 2)\drightarrow -\Lambda_{j-1}, 
% \]
% where $\Lambda_0<\Lambda_1<\cdots$ are the eigenvalues of the
% stochastic operator $\mathcal{H}_{2/\alpha,\omega}$,
  whose spectrum is simple with probability one.}
\textcolor{black}{The definition of $\mathcal{H}_{2/\alpha,\omega}$
  is recalled in Section \ref{sec: proof three poins}.}

We need some further consequences of these convergences, along with the extension of these consequences to sub-critically spiked Wigner matrices.
The particular results are summarized in the following lemma.
\begin{lemma}
\label{lemma three points}
Under Assumption\fix{s W or G$^\omega$,} we have
\begin{enumerate}[label=(\roman*)]
    \item For any \(k \in \Z_{>0}\), let \((\TW_{2/\alpha,j})_{1\leq j
        \leq k}\) be the joint limiting distribution of the \(k\)
      largest eigenvalues for a scaled GUE (\(\alpha = 1\)) or GOE
      (\(\alpha = 2\)), and let $\Lambda_0<\dotsc<\Lambda_{k-1}$ be
      the $k$ smallest eigenvalues of the stochastic
      \textcolor{black}{Airy} operator $\mathcal{H}_{2/\alpha,\omega}$
      \textcolor{black}{acting on functions satisfying Robin boundary
        condition $f'(0)=\omega f(0)$.} Then
      %  (see \cite{BloVirI} for a precise definition).
      % defined in \cite{BloVirI}. Then,
\fix{ \begin{equation*}
    \bigl(N^{\frac{2}{3}}(\lambda_1 - 2), \dotsc,
    N^{\frac{2}{3}}(\lambda_k - 2)\bigr) \drightarrow
    \begin{cases}
      (\TW_{2/\alpha,1}, \dotsc, \TW_{2/\alpha,k}) & \text{case W} \\
      -(\Lambda_0, \dotsc, \Lambda_{k-1}). & \text{case G}^\omega
    \end{cases}
  \end{equation*}}
\fix{We write $\mathrm{BV}_{2/\alpha}(\omega)$ for the law of
  $-\Lambda_0$ in case G$^\omega$.}
  %     \begin{align*}
  %   &\bigl(N^{\frac{2}{3}}(\lambda_1 - 2), \dotsc, N^{\frac{2}{3}}(\lambda_k - 2)\bigr)
  %   \drightarrow (\TW_{2/\alpha,1}, \dotsc, \TW_{2/\alpha,k}),\\
  %   &\textcolor{black}{\bigl(N^{\frac{2}{3}}(\lambda_1^{(\omega)} - 2), \dotsc, N^{\frac{2}{3}}(\lambda_k^{(\omega)} - 2)\bigr)
  %   \drightarrow -(\Lambda_0, \dotsc, \Lambda_{k-1}).}
  % \end{align*}
  
  \item   For any $\varepsilon>0$, there are $C_\varepsilon,N_\varepsilon$ such that for $N\geq N_\varepsilon$, with probability at least $1-\varepsilon$,
  \[
    \lambda_1\geq 2+C_\varepsilon N^{-2/3}.
  \]
   
  \item
  For any fixed $x\in \mathbb{R}$, there exists a constant $C_x$, such that
  \[
    \E \#\left\{j:\; \lambda_j \geq 2-xN^{-2/3}\right\}\leq C_x.
  \]
  
  \item
  For some $c_\varepsilon,N_\varepsilon$ and any $N\geq N_\varepsilon$,  with probability at least $1-\varepsilon$
  \begin{equation*}
  %\label{eq:lam1m2}
    \lambda_1-\lambda_2\geq c_\varepsilon N^{-2/3}.  
\end{equation*}
In other words, $\lambda_1 - \lambda_2 = \Theta_{\Pr }(N^{-2/3})$.

 \item \textcolor{black}{There exists $\kappa > 0$ such that if}
$b_N \rightarrow \infty$ so that $b_N =
  O(N^{\epsilon})$ for any $\epsilon > 0$, then 
 we have a.a.s.\ that
  \[
    \#\{j : \lambda_j > 2 - b_N N^{-2/3}\}
    \geq \kappa  b_N^{3/2}.
  \]

  % \item \textcolor{black}{Statements (ii), (iii), (iv), and (v) of this lemma remain valid when $\lambda_j$, $j=1,2,\dotsc$ are replaced by $\lambda_j^{(\omega)}$, $j=1,2,\dotsc$.}
\end{enumerate}

\end{lemma}

\subsubsection{\textcolor{black}{Derivatives of logarithmic statistics}}
\textcolor{black}{The next two lemmas describe asymptotic behavior of derivatives of $G$ at $\hat{\gamma}=2+b^2N^{-2/3}\log N$, and of the closely related statistics
\[
\frac{1}{N}\sum_{j=2}^N\frac{1}{(\lambda_1-\lambda_j)^k},\qquad\text{for }k=1,2.
\]
} 

\begin{lemma}
\label{lemma derivatives of G}\textcolor{black}{Suppose Assumption W
  \fix{or G$^\omega$} holds. Let $G(z)=\beta z - \frac{1}{N}\sum_{j=1}^N \log(z - \lambda_{j})$ with $\beta=1+bN^{-1/3}\sqrt{\log N}$, and $\hat{\gamma}=2+b^2N^{-2/3}\log N$.} Denote the $l$-th derivative of $G(z)$ as $%
G^{(l)}(z),$ \textcolor{black}{and let $b_+=\max\{0,b\}$. Then, for $b\neq 0$,}%
\begin{align}
\label{eq:DG-bound1}
    G^{(l)}( \hat{\gamma})
    &= \begin{cases}
        2b_+N^{-1/3}\log ^{1/2}N+o_{\Pr}\left( N^{-1/3}\log ^{-1/4}N\right) & 
        \text{for }l=1, \\ 
        (-1)^{l}\frac{(2l-4)!}{(l-2)!}
    \left( \frac{N^{1/3}}{2\left\vert b\right\vert \log^{1/2}N}\right)^{2l-3}
    ( 1+o_{\Pr}(1)) & \text{for }l\geq 2,
    \end{cases}  
\end{align}
% \textcolor{black}{Equation \eqref{eq:DG-bound1} remains valid when $\lambda_j$, $j=1,2,\dotsc$ in the definition of $G(z)$ are replaced by $\lambda_j^{(\omega)}$, $j=1,2,\dotsc$.}
\end{lemma}

\begin{lemma}
\label{lemma lambda1statistic}
\textcolor{black}{Let $C\in\mathbb{R}$ be fixed. Then, under
  \fix{either Assumption W or G$^\omega$,} we have}
\begin{align} \label{eq:p4-bd}
    \frac{1}{N} \sum_{j=1}^N \frac{1}{2 + CN^{-\frac{2}{3}} - \textcolor{black}{\lambda_j}}
    = 1 + O_{\Pr}(N^{-1/3}),
    \qquad
    \frac{1}{N} \sum_{j=1}^N \frac{1}{(2 + C N^{-\frac{2}{3}}- \textcolor{black}{\lambda_j})^2}
    = O_{\Pr}(N^{1/3}),
  \end{align}  
  and
\begin{align}\label{inv mom 1}
\frac{1}{N}\sum_{j=2}^{N}\frac{1}{\lambda _{1}-\lambda _{j}}
= 1+O_{\Pr}\left( N^{-1/3}\right) , \qquad %\\
\frac{1}{N}\sum_{j=2}^{N}\frac{1}{\left( \lambda _{1}-\lambda
_{j}\right) ^{2}}
= O_{\Pr}\left(N^{1/3}\right) .
\end{align}
% \textcolor{black}{All equalities in \eqref{eq:p4-bd} and \eqref{inv mom 1}  remain valid when $\lambda_j$, $j=1,2,\dotsc$  are replaced by $\lambda_j^{(\omega)}$, $j=1,2,\dotsc$.}
\end{lemma}

\subsubsection{Independence of the largest eigenvalue from the linear
  statistic} Our last preliminary result shows the asymptotic
independence of $\lambda_1$ and
$N^{-1}\sum_{j=1}^N\log|2-\lambda_j|$. \textcolor{black}{As discussed
  in the introduction, it may be of independent interest.}
\fix{ The proof
  for G(O/U)E is in Section \ref{sec:proof of CLT}, and for other
cases in the Appendix.}

\begin{proposition}\label{independence}
 \fix{ Suppose Assumption W or G$^\omega$ holds and let $\gamma = 2$.
Let $-\xi_{1N}$ denote the scaled logarithmic statistic in Theorem
\ref{CLT2} with $C = 0$ and again with the $-\frac{1}{3} \log N$ shift
in case G$^\omega$.
Then the random variables
$\xi_{1N}$ and $\xi_{2N} = N^{2/3}(\lambda_1 - 2)$ 
% \begin{align*}
%     \xi_{1N} &= \frac{N/2 - \frac{\alpha - 1}{6}\log N - \sum_{j = 1}^{N} \log |2 - \lambda_j|}{\sqrt{\frac{\alpha}{3}\log N}} \;\; \text{and } \;\;
%     \xi_{2N} = N^{2/3}(\lambda_1 - 2) 
% \end{align*}
are asymptotically independent with limiting distribution given by
\[
  (\xi_{1N},\xi_{2N})\drightarrow
  \begin{cases}
    \mathcal{N}(0,1)\times \mathrm{TW}_{2/\alpha} & \text{case W} \\
    \mathcal{N}(0,1)\times \mathrm{BV}_{2/\alpha}(\omega) & \text{case
                                                            G$^\omega$.} 
  \end{cases}
\]}
% \textcolor{black}{For the critically spiked G(O/U)E cases, the random variables
% \begin{align*}
%     \xi_{1N}^{(\omega)} &= \frac{N/2 - \frac{\alpha + 1}{6}\log N - \sum_{j = 1}^{N} \log |2 - \lambda_j^{(\omega)}|}{\sqrt{\frac{\alpha}{3}\log N}} \;\; \text{and } \;\;
%     \xi_{2N}^{(\omega)} = N^{2/3}(\lambda_1^{(\omega)} - 2) 
% \end{align*}
%  are asymptotically independent with limiting distribution given by
% \[
% (\xi_{1N}^{(\omega)},\xi_{2N}^{(\omega)})\drightarrow\mathcal{N}(0,1)\times \mathrm{BV}_{2/\alpha}(\omega),
% \]
% where $\mathrm{BV}_{2/\alpha}(\cdot)$ is the one parameter family of distributions described in theorems 1.5 and 1.7 of \cite{BloVirI}.}
\end{proposition}

\subsection{Proof strategy}

\textcolor{black}{Many} derivations in this paper revolve around the analysis of the integral
\[
  \int_{\mathcal{K}} \exp\left\{ (N/\alpha) G(z)\right\} \diff z.
\]
The fluctuations of this term differ qualitatively for $b < 0$ and $b
\geq 0$, and are considered in 
\cref{sec:negative-critical,sec:positive-critical} respectively.
In both cases the proofs involve Laplace approximation, but on
different contours.
\fix{In many respects, the role of the assumptions on the disorder
  $W_{J,N}$ and the resulting random matrix theoretic properties are
  concentrated in the results in Section
  \ref{sec:some-prel-results}. Given these 
  preliminaries, the analysis of the the contour integral, sections
  \ref{sec:negative-critical} and \ref{sec:positive-critical}, is the
  same, whether the assumptions on $W_N$ be Gaussian or Wigner.}
\textcolor{black}{We outline below the approach for the
  sub-critically spiked Wigner case (theorem \ref{thm:main}). The same
  strategy applies \textit{mutatis mutandis} for critically spiked
  G(O/U)E, theorem \ref{thm: triple} .} 

\subsubsection{\Cref{sec:negative-critical}: Negative critical case}

We use the vertical contour of \eqref{eq for z}, and a deterministic
choice for $\gamma$ suffices. Indeed, use the Stieltjes transform of
the semicircle law to make the approximation
% To this end, we use the semicircle law approximation of a linear eigenvalue statistics, taking
\begin{align*}
  G'(z)
  = \beta - \frac{1}{N} \sum_{j=1}^N \frac{1}{z - \lambda_j}
  % &\approx \frac{1}{\alpha} \Bigl(\beta - \int \frac{1}{z - \lambda_j} \frac{\sqrt{4 - \lambda^2}}{4 \pi}\Bigr) \\
  \approx \beta - \frac{z - \sqrt{z^2 - 4}}{2}.
\end{align*}
When \(b < 0\), the critical point of the approximation
is $\gamma=\hat{\gamma}+o(N^{-1 + \epsilon})$ for \(\hat{\gamma} = 2 +
b^2 N^{-2/3} \log N\) and any small positive
$\varepsilon$. 
Laplace approximation of the integral
\begin{align*}
  \int_{\mathcal{K}} \exp\left\{(N/\alpha)[G(z) - G(\hat{\gamma})]\right\} \diff z
\end{align*}
requires bounds on derivatives \(G^{(l)}(\hat{\gamma})\), for $l = 1, 2, 3$,
\textcolor{black}{provided by}  \cref{lemma derivatives of G}.
Having established that the fluctuations of \(F_{\alpha,N}\) depend asymptotically only on \(G(\hat{\gamma})\), it remains only to apply  \cref{CLT1}, conclude that \(F_{\alpha,N}\) is asymptotically Gaussian, and compute the correct centering and scaling constants.

\subsubsection{\Cref{sec:positive-critical}: Positive critical case}

When \(b \geq 0\), the deterministic approximation to \(G\) no longer has
a critical point along the real axis, and the approximation
\(\hat{\gamma}\) fails.
Indeed, \cref{lemma derivatives of G} shows that 
\(G'(\hat{\gamma})\) is of greater order than
%in the negative critical case,
when \(b < 0\), so \(G(z)\) oscillates too rapidly
along the vertical contour through \(\hat{\gamma}\).

We consider first $b > 0$, and instead use the contour of
\cref{fig:positive-contour}, which has a vertical part
\(\mathcal{K}_3\) through \(\mu=(\lambda_1 + \lambda_2)/2\) and a
keyhole part \(\mathcal{K}_1 \cup \mathcal{K}_2\) extending
horizontally from \(\mu\) and surrounding \(\lambda_1\).
% We show that in this case, the integral is dominated by the section integrated along the keyhole, and compute the asymptotics of this integral.
% In particular, we first show that the integral along the keyhole part is 
The integral turns out to be dominated by the keyhole part, with
\begin{equation} \label{eq:pos-crit}
  \frac{1}{2\pi\im}\int_{\mathcal{K}_1 \cup \mathcal{K}_2}
  \exp\left\{(N/\alpha) G(z)\right\} \diff z 
  = \exp\Bigl\{(N/\alpha) \hat{G}(\lambda_1) - \frac{\alpha - 1}{3} \log N + O_{\Pr}(\log \log N)\Bigr\},
\end{equation}
where
\begin{align}
  \label{eq:G-alpha}
  \hat{G}(\lambda_1)
  = \beta \lambda_1 - \frac{1}{N}\sum_{j=2}^N \log(\lambda_1 - \lambda_j).
\end{align}
The proof requires bounds on the derivatives
$\hat{G}^{(l)}(\lambda_1), l = 1, 2$ given in \cref{lemma
  lambda1statistic}.

In the boundary case $b=0$, the contributions of $\mathcal{K}_3$ and
$\mathcal{K}_1\cup\mathcal{K}_2$ are of the same order of magnitude,
so we consider instead the contour of the steepest descent. We
establish upper and lower bounds on the integral
and recover the right hand side of \eqref{eq:pos-crit} in this case also.

The analysis of \(\hat{G}(\lambda_1)\) is based on the approximation
%this quantity continues by demonstrating the approximation
\begin{align}\label{eq:lambda vs 2}
   \sum_{j=2}^N \log(\lambda_1 - \lambda_j)
  =  \sum_{j=1}^N \log|2 - \lambda_j|+ N(\lambda_1-2)+O_{\Pr}(1).
\end{align}
%where \(\bar{\lambda} = 2 + O(N^{-2/3})\).
The right side sum can be handled by \cref{CLT2}.
The \(\lambda_1\) terms in \eqref{eq:G-alpha} and \eqref{eq:lambda vs 2}
%are asymptotically non-negligible, and so
both contribute to Tracy-Widom fluctuations.
The last part of the argument hinges on the asymptotic independence of \(\lambda_1\) and  \(N^{-1} \sum_{j=1}^N \log|2 - \lambda_j|\), \textcolor{black}{which is established in \cref{independence}.}

\section{Negative-critical regime}
\label{sec:negative-critical}

% \textcolor{black}{In this section, we prove negative critical parts ($b<0$) of theorems \ref{thm:main} and \ref{thm: triple}. We will focus on the proof of theorem \ref{thm:main}, pointing out changes needed for the proof of theorem \ref{thm: triple}. } 

For the case $b<0,$ we deform $\mathcal{K}$ so that it is the vertical
line passing through $\hat{\gamma},$ a point in $\R$ that
approximates the critical point $\gamma$ of the function $G(z).$
Note that 
\begin{align*}
  G'(z)
  &=\beta \textcolor{black}{+} \frac{1}{N}\sum_{j=1}^N \frac{1}{\textcolor{black}{\lambda_{j}-z}},
\end{align*}%
where $\frac{1}{N}\sum_{j=1}^{N}\frac{1}{\textcolor{black}{\lambda _{j}-z}}$ is the
Stieltjes transform of the spectral distribution of \textcolor{black}{$W_{J,N}.$}
For $z>2,$ it must converge to the Stieltjes transform of
the semi-circle law, that is to%
\begin{equation*}
m_{SC}(z)=\frac{-z+\sqrt{z^{2}-4}}{2}.
\end{equation*}%
\textcolor{black}{Such a convergence follows e.g. from \cite[thm 2.4]{BenaychG2018}, the interlacing inequalities linking the eigenvalues of $W_N$ and $W_{J,N}$, and the convergence of $\lambda_1$ to $2$ implied e.g.~by part (i) of \cref{lemma three points}.} 
Solving $\beta +m_{sc}(z)=0$ for $z$, \textcolor{black}{we obtain $z=1/\beta+\beta$. Therefore, for}
\begin{equation}
    \label{inverse
    temperature}
    \beta
    = 1 + b N^{-1/3} \sqrt{\log N},
\end{equation}
we have $z=\hat{\gamma}+o(N^{-1+\varepsilon })$ for any $\varepsilon >0$, where 
\begin{equation*}
  \hat{\gamma}
  % \equiv \hat{\gamma}_{\alpha}
  =2+b^{2}N^{-2/3}\log N.
\end{equation*}

%To simplify notation, let us omit subscript $\alpha$ from $G_{\alpha}(z)$. 
\begin{lemma}
\label{Lemma C-}Suppose that \textcolor{black}{Assumption W holds and} $b<0$. Then%
\begin{equation*}
\int_{\hat{\gamma} - \im\infty }^{\hat{\gamma} + \im\infty }\exp\left\{\frac{N}{\alpha}[ G(z)-G(\hat{\gamma})]\right\} \diff z
= 2\sqrt{\pi \alpha \abs{b}}\frac{\im \log^{1/4}N}{N^{2/3}}\left( 1+o_{\Pr}(1)\right) .
\end{equation*}
\end{lemma}
\begin{proof} \textcolor{black}{As follows from  part (i) of \cref{lemma three points}, $\lambda_1-2=O_{\Pr}(N^{-2/3})$. Hence, $\lambda_1<\hat{\gamma}$ a.a.s., so we will assume without loss of generality that the latter inequality holds.}
Changing variables $z\mapsto \hat{\gamma}+\im t\frac{\log ^{1/4}N}{N^{2/3}}$, we represent the integral as
\begin{equation*}
\frac{\im \log ^{1/4}N}{N^{2/3}}\int_{-\infty }^{\infty }\exp \left\{ \frac{N}{\alpha}%
\tilde{G}(t)\right\} \mathrm{d}t,
\end{equation*}
where
\begin{equation*}
\tilde{G}(t):= G\left( \hat{\gamma}+\im t\frac{\log ^{1/4}N}{%
N^{2/3}}\right) -G\left( \hat{\gamma}\right) .
\end{equation*}%
Using the Lagrange form of the remainder in the Taylor expansions of the
real and imaginary parts of $G(z)-G\left( \hat{\gamma}\right) ,$ we arrive
at the inequality 
\begin{equation}
\left\vert G(z)-G\left( \hat{\gamma}\right) -G'\left( \hat{\gamma}%
\right) \left( z-\hat{\gamma}\right) -\frac{1}{2}G''\left( 
\hat{\gamma}\right) \left( z-\hat{\gamma}\right) ^{2}\right\vert \leq \frac{%
\left\vert z-\hat{\gamma}\right\vert ^{3}}{3}\sup_{\zeta \in \hat{\gamma}+ 
\im \mathbb{R}}\left\vert G'''\left( \zeta
\right) \right\vert .  \label{Taylor1}
\end{equation}%
On the event $\lambda _{1}<\hat{\gamma},$ the latter supremum is no larger
than $\left\vert G'''\left( \hat{\gamma}\right)
\right\vert $ because, for any $z\in \left( \hat{\gamma}-\im \infty ,%
\hat{\gamma}+\im \infty \right) ,$ $\left\vert z-\lambda
_{j}\right\vert ^{-3} \leq \left( \hat{\gamma}-\lambda _{j}\right) ^{-3}.$ Hence, we have%
\begin{equation}
\label{eq:remainder}
\left\vert \tilde{G}(t)-G'\left( \hat{\gamma}\right) \im t%
\frac{\log ^{1/4}N}{N^{2/3}}+\frac{1}{2}G''\left( \hat{\gamma}%
\right) t^{2}\frac{\log ^{1/2}N}{N^{4/3}}\right\vert \leq \frac{\left\vert
t\right\vert ^{3}}{3}\frac{\log ^{3/4}N}{N^{2}}\left\vert G'''\left( \hat{\gamma}\right) \right\vert .
\end{equation}%

\Cref{lemma derivatives of G} and inequality \eqref{eq:remainder} yield, for any fixed $C>0$ \textcolor{black}{and $b<0$},%
\begin{equation}
\int_{-C}^{C}\exp \left\{ \frac{N}{\alpha}\tilde{G}(t)\right\} \mathrm{d}t=\left( 1+o_{%
\Pr}\left( 1\right) \right) \int_{-C}^{C}\exp \left\{ -\frac{t^{2}}{%
4\alpha \abs{b}}\right\} \mathrm{d}t.  \label{part of integral}
\end{equation}%
Further, for any $t\in \mathbb{R}$, by definition,%
\begin{align*}
\func{Re}\tilde{G}(t)
&= -\frac{1}{ N}\sum_{j=1}^{N}\log \left\vert 1+\frac{\im t\log ^{1/4}N}{N^{2/3}\left( \hat{\gamma}-\lambda_{j}\right) }\right\vert \\
&= -\frac{1}{2 N}\sum_{j=1}^{N}\log \left( 1+\frac{t^{2}\log^{1/2}N}{N^{4/3}\left( \hat{\gamma}-\lambda _{j}\right)^2}\right) .
\end{align*}%
We will use the elementary inequality $\log(1+\delta)\geq\delta/2$ for $\delta \in [0,1]$. Note that the event $\mathcal{E}_0=\{\hat{\gamma}-\lambda_1>\frac{b^2}{2} N^{-2/3} \log N\}$ holds a.a.s. Conditionally on $\mathcal{E}_0$, for all $|t|\leq t_N :=\frac{b^2}{2}\log^{3/4}N$, we have
\[
    N\func{Re}\tilde{G}(t)<-\frac{t^2\log^{1/2}N}{4 N^{4/3}}\sum\nolimits_{j=1}^{N}\left( \hat{\gamma}-\lambda _{j}\right)^{-2}
    = -\frac{t^2\log^{1/2}N}{4 N^{1/3}}G''(\hat{\gamma}).
\]
Using \cref{lemma derivatives of G}, we conclude that for all $|t|\leq t_N$ \textcolor{black}{and $b\neq 0$} ,
\[
N\func{Re}\tilde{G}(t)<-\frac{t^2}{8 |b|}(1+o_{\Pr}(1)).
\]
Therefore, \textcolor{black}{by Chernoff's inequality,}
\begin{equation}
    \label{first section of the integral}
    \int_{-t_N}^{-C}\left\vert\exp\left\{\frac{N}{\alpha}\tilde{G}(t)\right\}\right\vert\diff t+\int_{C}^{t_N}\left\vert\exp\left\{\frac{N}{\alpha}\tilde{G}(t)\right\}\right\vert\diff t<(1+o_{\Pr}(1))%\frac{8|b|\alpha}{C}
    \textcolor{black}{2\sqrt{8\pi\alpha|b|}}\exp \left\{-\frac{C^2}{8|b|\alpha}\right\}.
\end{equation}
Since $C$ can be chosen arbitrarily large, equations \eqref{part of integral} and \eqref{first section of the integral} yield
\begin{equation}
    \label{vertical part}
    \int_{\hat{\gamma}-\mathrm{i}t_N N^{-2/3}\log^{1/4} N }^{\hat{\gamma}+\mathrm{i}t_N N^{-2/3}\log^{1/4} N }\exp\left\{\frac{N}{\alpha}[ G(z)-G(\hat{\gamma})]\right\} \diff z
= 2\sqrt{\pi \alpha \abs{b}}\frac{\im \log^{1/4}N}{N^{2/3}}\left( 1+o_{\Pr}(1)\right).
\end{equation}

It remains to show that the contribution of the remaining parts of the integral is negligible. Clearly, it is sufficient to prove that 
\begin{equation}
\label{negligibility}
\int_{t_N}^{\infty}\left\vert\exp\left\{\frac{N}{\alpha}\tilde{G}(t)\right\}\right\vert\diff t=o_{\Pr}(N^{-k})
\end{equation}
for arbitrarily large fixed $k$.
Note that 
$\func{Re}\tilde{G}(t)$ is a strictly decreasing function of $t\in [t_N,\infty)$. Therefore,
\begin{align*}
\int_{t_N}^{N^2}\left\vert\exp\left\{\frac{N}{\alpha}\tilde{G}(t)\right\}\right\vert\diff t&<N^2 \exp\left\{\frac{N}{\alpha}\func{Re}\tilde{G}(t_N)\right\}\\
&<N^2\exp\left\{-\frac{|b|^3\log^{3/2}N}{32\alpha}(1+o_{\Pr}(1))\right\}\\
&=o_{\Pr}(N^{-k})
\end{align*}
for arbitrarily large fixed $k$.
For $t>N^2$, on the event $(\hat{\gamma}-\lambda_N)^2<\bar{C}$ that holds \textcolor{black}{a.a.s.} for some positive constant $\bar{C}$, we have
\begin{align*}
N\func{Re}\tilde{G}(t)&=-\frac{1}{2}\sum\nolimits_{j=1}^{N}\log\left(1+\frac{t^2 \log^{1/2} N}{N^{4/3} (\hat{\gamma}-\lambda_j)^2}\right)\\
&\leq -\frac{N}{2}\log\left(\frac{t^2 \log^{1/2} N}{N^{4/3} \bar{C}} \right).
\end{align*}
Therefore,
\begin{align*}
\int_{N^2}^{\infty}\left\vert\exp\left\{\frac{N}{\alpha}\tilde{G}(t)\right\}\right\vert\diff t&<
\int_{N^2}^{\infty}\left(\frac{t^2\log^{1/2}N}{N^{4/3}\bar{C}}\right)^{-\frac{N}{2\alpha}(1+o_{\Pr}(1))}\diff t=o_{\Pr}(N^{-k})
\end{align*}
for arbitrarily large fixed $k$ as well. Hence, \eqref{negligibility} indeed holds.
\end{proof}
% \textcolor{black}{All the above arguments hold when $\lambda_j$,
%   $j=1,2,\dotsc$, are replaced by $\lambda_j^{(\omega)}$,
%   $j=1,2,\dotsc$. Therefore, so far, the proof of the negative
%   critical part of theorem \ref{thm: triple} would develop exactly as
%   the above proof of the negative critical part of theorem
%   \ref{thm:main}.}
Now we are ready to prove the following theorem.
% , \textcolor{black}{which makes different statements for the
%   sub-critically spiked Wigner and critically spiked G(O/U)E cases,
%   relevant for theorems \ref{thm:main} and \ref{thm: triple},
%   respectively.}
Recall that $F_{\alpha,N} =\frac{\alpha}{%
2N}\log \mathcal{I}_{\alpha,N},$ where $\mathcal{I}_{\alpha,N}$ is as defined in \eqref{eq for z}. 

\begin{theorem}[Negative-critical regime] \label{thm:negat-crit-regime}
\textcolor{black}{Suppose Assumption W holds and} $\beta =1+bN^{-1/3}\log
^{1/2}N$ with $b<0$.
Then, \textcolor{black}{in the sub-critically spiked Wigner setting of theorem \ref{thm:main},}
\begin{equation} \label{eq:T3.2}
\frac{N}{\sqrt{\frac{\alpha}{12}\log N}}\left( F_{\alpha,N} -\frac{1}{4}\beta^2+\frac{\log N}{12 N}\right) 
\drightarrow \Normal(0,1).
\end{equation}
\fix{In case G$^\omega$, in theorem \ref{thm: triple}, the sign of the
  $\log N/12N$ term is reversed.} 
% \textcolor{black}{In contrast, in the critically spiked G(O/U)E setting of theorem \ref{thm: triple},
% \[
% \frac{N}{\sqrt{\frac{\alpha}{12}\log N}}\left( F_{\alpha,N} -\frac{1}{4}\beta^2-\frac{\log N}{12 N}\right) 
% \drightarrow \Normal(0,1).
% \]}
\proof
After rearranging \eqref{eq for z}, we have
\begin{equation*}
2 N F_{\alpha,N}
= \alpha \log C_{\alpha,N} + NG(\hat{\gamma}) + \alpha \log \frac{1}{2\pi
  \im }\int_{\hat{\gamma}-\im \infty }^{\hat{\gamma}+\im \infty }\exp
\left\{\frac{N}{\alpha}\left[
    G(z)-G\left(\hat{\gamma}\right)\right] \right\}
\diff z. 
\end{equation*}
For the first term, using Stirling's formula,
\begin{equation}
\label{eq:C_N-asymptotics}
\alpha \log C_{\alpha,N} =
\alpha \log \frac{\sqrt{2\pi }
  (N/\alpha)^{N/\alpha-1/2}e^{-N/\alpha}}{(N\beta/\alpha)^{N/\alpha-1}}
+ o(1) 
= -N\left(1+\log \beta \right) +\frac{\alpha}{2}\log N + O(1).
\end{equation}%
For the second term we have 
\begin{align*}
N G(\hat{\gamma})
&= N \beta \hat{\gamma} - \sum_{j=1}^{N}\log \left(
                    \hat{\gamma}-\lambda _{j}\right) \\ 
&= 2\beta N + b^2 N^{1/3}\log N + b^3\log^{3/2}N - \sum\nolimits_{j=1}^{N}\log ( \hat{\gamma}-\lambda_j).
\end{align*}
For the third term, using \cref{Lemma C-},%
\begin{equation*}
\alpha \log \frac{1}{2\pi \im }\int_{\hat{\gamma}-\im \infty }^{\hat{\gamma} + \im \infty }\exp \left\{\frac{N}{\alpha}\left[ G(z)-G(\hat{\gamma}) \right] \right\} \diff z
= -\frac{2\alpha}{3}\log N + O_{\Pr}(\log\log N) .
\end{equation*}%
Combining the three terms, we obtain 
\begin{align*}
2 N F_{\alpha,N}
&= N(-1-\log \beta + 2\beta)  + b^2 N^{1/3}\log N  + b^3\log^{3/2} N
               -\frac{\alpha}{6} \log N \\ 
  &\qquad -\sum_{j=1}^{N}\log (\hat{\gamma}-\lambda_j)
    + O_{\Pr}\left( \log \log N \right).
\end{align*}

Let 
\begin{equation*}
N \xi_N:= \sum_{j=1}^N \log (\hat{\gamma} -\lambda_j) -
\frac{N}{2}
- b^2 N^{1/3} \log N
+ \frac{2}{3} \abs{b}^3 \log^{3/2}N
+\frac{\alpha-1}{6} \log N.
\end{equation*}%
Combining the last two displays and noting that
\[
b^3\frac{\log^{3/2} N}{N}=(\beta-1)^3,
\]
we get (for $b<0$)
\begin{align*}
    2 N F_{\alpha,N}=N\left(
    2\beta-\log\beta-\frac{3}{2}+\frac{1}{3}(\beta-1)^3-\frac{\log N}{6N}\right)-N\xi_N+O_{\Pr}\left(\log\log N\right).
\end{align*}
Using the Taylor expansion
\[
\log\beta=(\beta-1)-\frac{1}{2}(\beta-1)^2+\frac{1}{3}(\beta-1)^3+o(N^{-1})
\]
in the previous display, we obtain
\begin{equation} \label{eq:key-pos}
  2 N F_{\alpha,N}
  = \frac{N}{2}\beta^2-\frac{\log N}{6}- N \xi_N + O_{\Pr}\left( \log \log N \right) .
\end{equation}%
% \textcolor{black}{Until this point, no changes are needed if we replace $\lambda_j$, $j=1,2,\dotsc$, by $\lambda_j^{(\omega)}$, $j=1,2,\dotsc$.}

\fix{Up to this point, all arguments in this section are the same for
  case W (i.e. subcritical) and case G$^\omega$ (critical)}.
\fix{Now, in case W,} the sub-critically spiked Wigner setting of theorem \ref{thm:main}, by \cref{CLT1},%
% \begin{equation*}
% \frac{N}{\sqrt{\frac{\alpha}{3}\log N}}\xi_{N}\overset{d}{\rightarrow }\mathcal{N}(0,1).%
% \text{ }
% \end{equation*}%
\begin{equation*}
N\xi_{N} \Big/ \sqrt{\tfrac{\alpha}{3}\log N} 
\drightarrow \mathcal{N}(0,1).%
\text{ }
\end{equation*}%
This yields the first convergence of theorem
\ref{thm:negat-crit-regime} and hence the negative critical part of
theorem \ref{thm:main}.
\fix{In Case G$^\omega$ and theorem \ref{thm: triple}, a term $\frac{1}{3} \log N$ must be added to
  $N \xi_N$ to obtain convergence in (\ref{eq:T2.5}), and (after scaling) this
  amounts to subtracting $\log N/6N$ on the left side of (\ref{eq:T3.2}), and
  so to the claimed reversal of sign.} \qed
% \textcolor{black}{In contrast, in the critically spiked G(O/U)E setting of theorem \ref{thm: triple}, by theorem \ref{CLT1},
% \[
% \left(N\xi_{N}+\frac{1}{3}\log N\right) \Big/ \sqrt{\tfrac{\alpha}{3}\log N} 
% \drightarrow \mathcal{N}(0,1).
% \]
% This yields the second convergence of theorem \ref{thm:negat-crit-regime} and hence the negative critical part of theorem \ref{thm: triple}.} \qed
\end{theorem}

\section{Positive-critical regime}
\label{sec:positive-critical}
% \textcolor{black}{In this section, we prove theorems \ref{thm:main} and \ref{thm: triple} for the case $b\geq 0$. Again, we focus on the proof of theorem \ref{thm:main}, and indicate necessary changes for theorem \ref{thm: triple}. } 

%\bigskip
\begin{figure}
    \centering
    \begin{tikzpicture}[scale = 1]
    
    \draw[opacity=0.5, ->] (-3, 0) -- (3, 0) node[anchor=north] {$\mathrm{Re}(z)$};
    \draw[opacity=0.5, ->] (0, -3) -- (0, 3) node[anchor=west] {$\mathrm{Im}(z)$};
    
    \node[circle,fill=black,inner sep=0pt,minimum size=3pt,label=below:{$\lambda_1$}] (a) at (2,0) {};
    
    \node[circle,fill=black,inner sep=0pt,minimum size=3pt,label=below:{$\lambda_N$}] (a) at (-2,0) {};
    
    \draw[] (-0.5, 0) node[anchor=north] {$\dots$};
    
    \node[circle,fill=black,inner sep=0pt,minimum size=3pt,label=below:{$\lambda_2$}] (a) at (1,0) {};
    
    \draw[] (1.5, 0) node[anchor=center] {$\times$} node[anchor=south east] {$\mu$};
    
    \draw [arrow] (1.85,-0.1) coordinate (low) arc (-150:150:0.2) coordinate (top) node[midway,above right] {\small\(\mathcal{K}_1\)};
    \draw [arrow] (top) -- (1.5, 0.1) coordinate (topmu) node[near start,above] {\small\(\mathcal{K}_2^+\)};
    \draw [arrow] (1.5, -0.1) coordinate (lowmu) -- (low);
    \draw [arrow] (topmu) -- (1.5, 3) node[midway,left] {\small\(\mathcal{K}_3^+\)};
    \draw[arrow] (1.5, -3) -- (lowmu) node[midway,left] {\small\(\mathcal{K}_3^-\)};

    \end{tikzpicture}
    \caption{Contour of integration for positive \(b\)}
    \label{fig:positive-contour}
\end{figure}
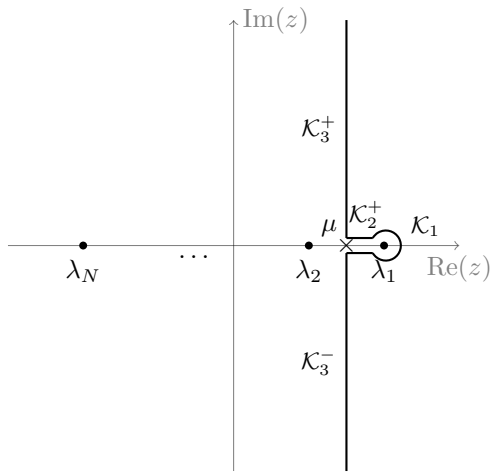

The vertical contour passing through $\hat{\gamma}$ does not work when
$b \geq 0$ because $G'( \hat{\gamma}) $ becomes non-negligible.
As a result, the function $G(z)$  oscillates
quickly along the vertical contour near \(\hat{\gamma}\).
Instead we use contours crossing the real axis closer to
$\lambda_1$.
To this end, we consider the  nonsingular part of $G$ at $\lambda_1$
and define
\begin{equation*}
  \hat{G}(\lambda_1) = \beta \lambda_1 - \frac{1}{N} \sum_{j=2}^N
  \log(\lambda_1 - \lambda_j).
\end{equation*}

\begin{proposition}
  \label{prop:positive-critical}
  \textcolor{black}{Suppose Assumption W or G$^\omega$ holds.} If $b \geq 0$, then for both $\alpha = 1, 2$,
  \begin{equation*}
    \frac{1}{2 \pi \im} \int_{\mathcal{K}} \exp \{(N/\alpha) G(z) \}
    \diff z
    = \exp \left\{ (N/\alpha) \hat{G}(\lambda_1) - \frac{(\alpha-1)}{3}
    \log N + O_{\Pr}(\log \log N) \right\}.
  \end{equation*}
\end{proposition}

For $b > 0$, we consider the vertical \textquotedblleft keyhole
contour\textquotedblright\ $\mathcal{K}$,  \cref{fig:positive-contour},
which is symmetric around the real
axis and has the following form above the axis:%
\begin{equation*}
\mathcal{K}^{+}=\mathcal{K}_{1}^{+}\cup \mathcal{K}_{2}^{+}\cup \mathcal{K}%
_{3}^{+}
\end{equation*}%
with $\mathcal{K}_{1}^{+}$ being a semi-circle with center at $\lambda _{1}$
and small radius $\varepsilon ,$ $\mathcal{K}_{2}^{+}$ being a horizontal
segment connecting $\mu =\frac{\lambda_{1} + \lambda_{2}}{2}$ and $\lambda
_{1}-\varepsilon ,$ and $\mathcal{K}_{3}^{+}$ being a vertical ray starting
from $\mu$.

In the $\alpha = 1$ case,  the integrand is analytic away from \(\lambda_1,
\dotsc, \lambda_N\), and so the contributions of \(\mathcal{K}_2^+\)
and \(\mathcal{K}_2^-\) cancel.
On the other hand, for $\alpha = 2$, 
$\exp \left\{ (N/\alpha)G(z)\right\} $ has a square-root-type singularity at $z=\lambda _{1}.$
Hence, the contribution of $\mathcal{K}_{1}$ to the integral
$\dint\nolimits_{\mathcal{K}}\exp \left\{ (N/\alpha)G(z)\right\}
\mathrm{d}z$ converges to zero as $\varepsilon \rightarrow 0$.
To summarize, let
\begin{equation*}
  I_{N} = \frac{1}{2\pi\im}\int_{\mathcal{K}} \exp \{(N/\alpha) G(z) \} \diff z,
  \qquad
  I_{Nk} = \frac{1}{2\pi\im}\int_{\mathcal{K}_k} \exp \{(N/\alpha) G(z) \} \diff z,
\end{equation*}
Thus, as $\epsilon \to 0$ we have for both $\alpha = 1, 2$ 
\begin{equation*}
  I_N = I_{N\alpha} + I_{N3}.
\end{equation*}
Let
\begin{equation*}
  A_{N \alpha} = \exp \{ (N/\alpha) \hat{G}(\lambda_1) -
\frac{\alpha-1}{3} \log N \}.
\end{equation*}
When $b > 0$, we establish  \cref{prop:positive-critical}
in \cref{sec:proof-proposition-b} by
showing that 
\begin{equation*}
  I_{N \alpha} = A_{N\alpha} \exp \{ O_{\Pr}(\log \log N) \}, \qquad
  I_{N 3} = o_{\Pr}(I_{N\alpha}).
\end{equation*}

The $b = 0$ case is more delicate.
The keyhole contour yields both $I_{N\alpha}, I_{N3} = A_{N\alpha}
\exp \{ O_{\Pr}(1) \}$ which suffices for the upper bound for $I_N$.
Since the $O_{\Pr}(1)$ terms are in general complex, some cancellation
between $I_{N\alpha}$ and $I_{N3}$ cannot be excluded, so further
argument is needed for the lower bound.
In \cref{sec:case-b-=} a separate argument using the steepest descent
contour yields the required lower bound.

\Cref{sec:limit-law-posit} completes the proof of the
positive-critical regime of 
 \cref{thm:main}.

\subsection{Proof of \cref{prop:positive-critical} for \texorpdfstring{$b > 0$}{b > 0}}
\label{sec:proof-proposition-b}

\begin{lemma}
\label{lem:edge-contours}
Suppose \textcolor{black}{Assumption W or G$^\omega$ holds and } $b\geq0$. Then for $\alpha=1,2$ we have
\[
I_{N\alpha}=A_{N\alpha}\exp\{O_{\Pr}(\log\log N)\}.
\]
\end{lemma}

\begin{proof}
  In the complex case, Cauchy's integral formula yields
  $I_{N1} = \exp \{ N \hat{G}(\lambda_1) \}=: A_{N1}$,
since \(\mathcal{K}_1\) encircles only \(\lambda_1\).
The rest of this proof is devoted to the real case.
First, consider
\begin{align*}
    \frac{1}{2\pi \im} \int_{\mathcal{K}_2^+ \cup \mathcal{K}_2^-} \exp\{(N/2) G(z)\} \diff z
    &=\frac{1}{2\pi\im}\int_{\lambda_1}^{\mu}\frac{-\im}{\sqrt{\lambda_1-y}}\exp\left\{\frac{N\beta y}{2}-\frac{1}{2}\sum_{j=2}^{N}\log (y-\lambda_j)\right\}\diff y\\&+\frac{1}{2\pi\im}\int_{\mu}^{\lambda_1}\frac{\im}{\sqrt{\lambda_1-y}}\exp\left\{\frac{N\beta y}{2}-\frac{1}{2}\sum_{j=2}^{N}\log (y-\lambda_j)\right\}\diff y.
\end{align*}
Changing variables $y \mapsto x=\lambda_1-y$, we obtain
\begin{align*}
    \frac{1}{2\pi \im} \int_{\mathcal{K}_2^+ \cup \mathcal{K}_2^-} \exp\{(N/2) G(z)\} \diff z
    &=\frac{1}{\pi}\int_{0}^{\frac{\lambda_1-\lambda_2}{2}}\frac{1}{\sqrt{x}}\exp\left\{\frac{N\beta (\lambda_1-x)}{2}-\frac{1}{2}\sum_{j=2}^{N}\log (\lambda_1-\lambda_j-x)\right\}\diff x\\
    &=\exp\left\{(N/2)\hat{G}(\lambda_1)\right\}\mathcal{I},
\end{align*}
where
\begin{equation*}
\mathcal{I}=\frac{1}{\pi }\dint\nolimits_{0}^{\frac{\lambda _{1}-\lambda
_{2}}{2}}\frac{1}{\sqrt{x}}\exp \left\{ -\frac{N}{2}\beta x-\frac{1}{2}%
\sum_{j=2}^{N}\log \left( 1-\frac{x}{\lambda _{1}-\lambda _{j}}%
\right) \right\} \mathrm{d}x.
\end{equation*}%

Since $0 \leq -\log(1-y) -y\leq y^2$ for $0 \leq y \leq \frac{1}{2}$,
we have for some $\xi \in [0,1]$,
\begin{align*}
- N\beta x - \sum_{j=2}^N\log \left( 1-\frac{x}{\lambda _{1}-\lambda
  _{j}}\right)
  & = - N \beta x +  \sum_{j=2}^N\frac{x}{\lambda _{1}-\lambda _{j}}
    + \xi \sum_{j=2}^N\frac{x^{2}}{\left( \lambda _{1}-\lambda
    _{j}\right) ^{2}}  \\
  & = N^{2/3} x ( - b \sqrt{\log N}  +  \omega_{1N})  + \xi
    N^{4/3} x^2 \omega_{2N}^+ ,
\end{align*}
with random variables $\omega_{1N}$ and $\omega_{2N}^+ > 0$ both being
$O_{\Pr }(1)$ from \cref{lemma lambda1statistic}.
Setting $y = N^{2/3}x$ and $\theta_N = N^{2/3}(\lambda_1 - \lambda_2)/2 =
\Theta_{\Pr}(1)$ (\fix{by \cref{lemma three points}(iv)}) and noting also that
$\omega_{1N}y+\xi \omega_{2N}^+y^2$ is uniformly $O_{\Pr}(1)$ for $0
\leq y \leq \theta_N$, we arrive at
\begin{equation*}
  \mathcal{I}
    = \frac{e^{O_{\Pr}(1)}}{N^{1/3}} \int_0^{\theta_N} \exp\{ -\frac{1}{2}
    b y \sqrt{\log N} \} \frac{\diff y}{\sqrt y}
      =
  \begin{cases}
    \dfrac{e^{O_{\Pr}(1)}}{N^{1/3}} \dfrac{1}{b^{1/2} \log^{1/4} N} & b
    > 0 \\
    \dfrac{e^{O_{\Pr}(1)}}{N^{1/3}}  & b  = 0.
  \end{cases}   \qedhere
\end{equation*}
\end{proof}

%\newpage
In what follows, we define $ G(\mu) = \lim_{t \rightarrow +0} G(\mu + \im t)$, i.e., as a continuation from the upper-half plane, so that we have $\log( \lambda_1 - \mu) = \log|\lambda_1 - \mu| + \pi\im $.

\begin{lemma}
  \textcolor{black}{Suppose Assumption W or G$^\omega$ holds.} For $b \geq 0$, we have $$|I_{N3}| \leq A_{N \alpha} \exp \left\{ -
  \theta_N \frac{b \sqrt{\log N}}{\alpha} + O_{\Pr}(1) \right\},$$
  where $\theta_N=N^{2/3}(\lambda_1-\lambda_2)/2$ is a non-negative $\Theta_{\Pr}(1)$ variable.
\end{lemma}
\begin{proof}
  It suffices to bound 
  \begin{equation*}
      I_{N3}^+ = \frac{1}{2 \pi \im} \int_{\mathcal{K}_3^+} \exp
      \{(N/\alpha) G(z) \} \diff z
      =  \frac{1}{2 \pi} \int_0^\infty \exp
      \{(N/\alpha) G(\mu + \im t) \} \diff z,
  \end{equation*}
as the analysis for $\mathcal{K}_3^-$ is analogous using
$ G(\bar{z}) = \overline{G(z)}$.
Let $\tilde{G}(t)=G\left( \mu +\im t\right) - G(\mu).$ 
We have
\begin{equation}
  \label{eq:IN3pl}
  |I_{N3}^+|
     \leq \frac{1}{2 \pi} \exp \{(N/\alpha) \hat{G}(\lambda_1) \} |J_N|
      K_N,
\end{equation}
with
\begin{equation*}
  J_N  = \exp \{ -(N/\alpha)[\hat{G}(\lambda_1)-G(\mu)] \},
  \qquad 
  K_N  = \int_0^\infty \exp \{ (N/\alpha) \Re [\tilde{G}(t)] \} \diff t.  
\end{equation*}

First we compare $\hat{G}(\lambda_1)$ and $G(\mu)$.
Since $\log( \mu + \im 0 - \lambda_1) = \log[(\lambda_1 -
\lambda_2)/2] + \im \pi$, we have 

\begin{equation*}
  N[\hat{G}(\lambda_1) - G(\mu)]
  = N\beta (\lambda_1 - \mu) + \log \frac{\lambda_1 - \lambda_2}{2}
  + \im\pi + \sum_{j = 2}^{N} \log\left( 1 - \frac{1}{2} \frac{\lambda_1 - \lambda_2}{\lambda_1 - \lambda_j} \right).
\end{equation*}
For $ 0 \leq t \leq 1$ we have that $ |\log (1 - t/2) + t/2| \leq t^2
$.
From \cref{lemma lambda1statistic} we then have
\begin{equation*}
  \sum_{j = 2}^{N} \log\left( 1 - \frac{1}{2} \frac{\lambda_1 -
      \lambda_2}{\lambda_1 - \lambda_j} \right)
  = -\frac{N(\lambda_1-\lambda_2)}{2} [1+O_{\Pr}(N^{-1/3})]
       +N (\lambda_1-\lambda_2)^2 O_{\Pr}(N^{1/3}).
\end{equation*}
% \[
%     \frac{1}{N} \sum_{j = 2}^{N} \frac{1}{\lambda_1 - \lambda_j} = 1 + O_{\Pr}(N^{-1/3}),
%     \qquad
%     \frac{1}{N} \sum_{j = 2}^{N} \frac{1}{(\lambda_1 - \lambda_j)^2} = O_{\Pr}(N^{1/3}) \, .
% \]
In addition, by \cref{lemma three points}(iv) $ \theta_N$ is a
$\Theta_{\Pr}(1)$ variable, and so
$\log(\lambda_1 - \lambda_2) =  -\frac{2}{3} \log N + O_{\Pr}(1)$,
and
\begin{align*}
  N[\hat{G}(\lambda_1) - G(\mu)]
  &= \frac{N \beta}{2} (\lambda_1 - \lambda_2) - \frac{2}{3} \log N - \frac{N}{2} (\lambda_1 - \lambda_2) + O_{\Pr}(1) \\
    &=
    \frac{N(\lambda_1 - \lambda_2)}{2} \frac{b\sqrt{\log N}}{N^{1/3}} - \frac{2}{3} \log N + O_{\Pr}(1) \\
    &=
    -\frac{2}{3} \log N + \theta_N {b} \sqrt{\log N} +
      O_{\Pr}(1). \numberthis{eq:b-lower-bound} 
\end{align*}
Thus
\begin{equation}
  \label{eq:JN}
  |J_N| = \exp \{ (2/3\alpha) \log N - \theta_N \frac{b \sqrt{\log N}}{\alpha} +
  O_{\Pr}(1) \}.
\end{equation}

We turn to $K_N$ in \eqref{eq:IN3pl}. Fix $k > \alpha$ and let
$\textcolor{black}{\zeta} = \max_{j \leq k} N^{2/3}|\lambda_j - \mu|$.
Neglecting negative terms, we have
\begin{equation*}
  N \Re \tilde{G}(t)
   = -\frac{1}{2}\sum_{j=1}^{N}\log \bigg( 1+\frac{t^{2}}{\left( \mu
         -\lambda_{j} \right)^{2}}\bigg)
   \leq - \frac{k}{2} \log \left( 1 + \frac{t^2}{\zeta^2 N^{-4/3}} \right) \, .
 \end{equation*}
Since $\zeta$ is a non-negative $\Theta_{\Pr}(1)$ variable from 
 \cref{lemma three points} (iv), we have
 \begin{align}
   \label{eq:KN}
   K_N & \leq
     \int_{0}^{\infty} \left(1 + \zeta^{-2} N^{4/3} t^2
     \right)^{-k/2\alpha} \diff t \\
     &  = \zeta N^{-2/3} \int_0^\infty (1+s^2)^{-k/2\alpha} \diff s 
     = \exp \left\{-\frac{2}{3}\log N+O_{\Pr}(1)\right\}.\notag
 \end{align}

Combining \eqref{eq:JN} and \eqref{eq:KN}, for each case $\alpha = 1,
2$ we arrive at
\begin{equation*}
  |J_N| K_N \leq
     \exp \{ - \frac{(\alpha-1)}{3} \log N - \theta_N \frac{b \sqrt{\log N}}{\alpha} +
     O_{\Pr}(1) \}. 
\end{equation*}
Together with \eqref{eq:IN3pl}, this yields the lemma.
\end{proof}

\subsection{The case \texorpdfstring{$b = 0$}{b = 0}}
\label{sec:case-b-=}

In this section we use the steepest descent contour to show that
\[
I_N \geq A_{N \alpha} \exp \{ O_{\Pr}(\log \log N) \}.
\]
When combined with the upper bound already established, this yields
\cref{prop:positive-critical} for $b = 0$.

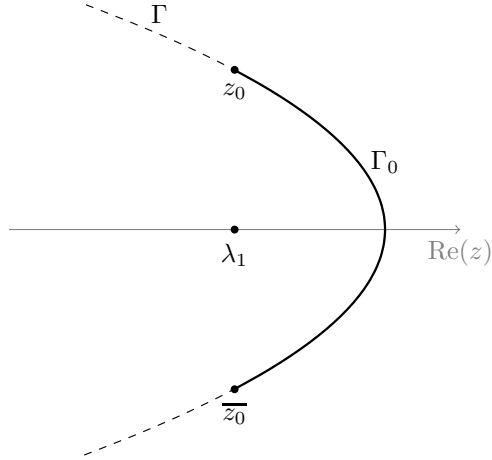
\begin{figure}
    \centering
    \begin{tikzpicture}[scale = 1]
    
    \draw[opacity=0.5, ->] (-3, 0) -- (3, 0) node[anchor=north] {$\mathrm{Re}(z)$};
    
    \draw[black, dashed, line width = 0.15mm]   plot[smooth,domain=-2:2] ({2 - \x*\x}, 1.5 * \x);

    \node[label={$\Gamma$}] (a) at (-1,2.5) {};
    \node[label={$\Gamma_0$}] (a) at (2,0.5) {};
    
    \draw[black, line width = 0.30mm]   plot[smooth,domain=-1.414:1.414] ({2 - \x*\x}, 1.5 * \x);
    
    \node[circle,fill=black,inner sep=0pt,minimum size=3pt,label=below:{$\lambda_1$}] (a) at (0,0) {};
    \node[circle,fill=black,inner sep=0pt,minimum size=3pt,label=below:{$z_0$}] (a) at (0,2.125) {}; 
    \node[circle,fill=black,inner sep=0pt,minimum size=3pt,label=below:{$\overline{z_0}$}] (a) at (0,-2.125) {}; 
    
    \end{tikzpicture}
    \caption{Curve of steepest descent of \(G(z)\) near \(\lambda_1\).}
    \label{fig:steepest-descent}
\end{figure}

Let $\Gamma$ denote the contour of steepest descent of $G(z)$ \textcolor{black}{\textit{crossing} the real line above $\lambda_1$. Such a contour exists because, as is easy to verify, there exists a unique saddle point of $G(z)$ on $z\in(\lambda_1,+\infty)$. Since $\Im G(z)$ must remain constant along such a contour, we have}
\begin{equation*}
    0 = \Im[G(z)] 
      = y - \frac{1}{N} \sum_{j=1}^N \arg((x - \lambda_j) + \im y),
\end{equation*}
for any $z = x + \im y \in \Gamma$.
From this equation, we observe that  $\Gamma$ is symmetric around the real axis.

Next, observe that for a fixed imaginary part $y > 0$, $\arg((x -
\lambda_j) + \im y)$ is strictly decreasing with $x$.
Hence, \textcolor{black}{equation $0=\Im G(x + \im y)$} can have at most one solution for any positive $y$.
By symmetry around the real axis, this also holds for $y < 0$.
This means that it is possible to parameterise
\[
  \Gamma
  = \{\Gamma(t) : 0 < t < 1\}
\]
so that $\Im \Gamma(t)$ is increasing in $t$.

Moreover, \textcolor{black}{since as $y\uparrow \pi$ we must have $x\rightarrow -\infty$,} we see that $\Gamma(0^+) = -\infty - \im \pi$ and $\Gamma(1^-) = -\infty  + \im \pi$.
Therefore, $\Gamma$ must have upper-bounded real part, and so
\[
  \int_{\mathcal{K}} \exp\{(N/\alpha) G(z)\} \diff z
  = \int_{\Gamma} \exp\{(N/\alpha) G(z)\} \diff z.
\]

To continue, we need one last result about $\Gamma$, which formalizes the notion that $\Gamma$ passes above $\lambda_1$ at a distance of roughly $N^{-2/3}$:
\begin{lemma}
  \label{lem:y0-bounds}
  \textcolor{black}{Under Assumption W or G$^\omega$, }the function
\begin{equation*}
    f(y)
    = \Im[ G(\lambda_1 + \im y)] \\
    = y - \frac{\pi}{2N} - \frac{1}{N} \sum_{j=2}^N \arctan     \Bigl(\frac{y}{\lambda_1 - \lambda_j}\Bigr)
  \end{equation*}
  has a unique positive root $y_0$.
If $a_N \to \infty$ such that \textcolor{black}{$a_N = O(N^\epsilon)$ for any $\epsilon>0$}, then a.a.s. 
  \begin{equation}
    \label{eq:y0-bounds}
    \frac{N^{-2/3}}{a_N}
    \leq y_0
    \leq N^{-2/3} a_N.
  \end{equation}
\end{lemma}

\begin{proof}
  Notice that, over $[0, \infty)$, $f$ is convex with $f(0) = -\pi/(2N)$ and $\lim_{y\rightarrow \infty} f(y) = \infty$.
  In particular, this means that it has a unique positive root, which we will call $y_0$.
  
  We will show that
%  , for large enough \(N\),
  \(f(N^{-2/3} a_N^{-1}) < 0 < f(N^{-2/3} a_N)\)
  a.a.s., which implies \eqref{eq:y0-bounds}.

  Let \(y_{-} = N^{-2/3} a_N^{-1}\).
  Using $\arctan(x) \geq x - x^2/4$ for $x \geq 0$
  and then \cref{lemma lambda1statistic}, 
we have 
  \begin{align*}
    f(y_{-})
    &\leq y_{-} \Bigl(1 - \frac{1}{N} \sum_{j=2}^N \frac{1}{\lambda_1 - \lambda_j}\Bigr) + \frac{y_{-}^2}{4N} \sum_{j=1}^N \frac{1}{(\lambda_1 - \lambda_j)^2} - \frac{\pi}{2N} \\
    &= N^{-2/3} a_N^{-1} O_{\Pr}(N^{-1/3}) + \frac{1}{4} N^{-4/3} a_N^{-2} O_{\Pr}(N^{1/3}) - \frac{\pi}{2N} \\
    &= -\frac{\pi}{2N} + o_\Pr(N^{-1})
  \end{align*}
Thus $f(y_{-}) < 0$ a.a.s.

Next, set \(y_{+} = N^{-2/3} a_N\).
  Now some \(y_{+}/(\lambda_1 - \lambda_j)\) terms will diverge to \(\infty\) and so the linear approximation to \(\arctan\) is not helpful.
  To handle these cases, we define
  \begin{align*}
    j^*
    &= \max\Bigl\{j : \frac{y_{+}}{\lambda_1 - \lambda_j} > 1 + \frac{\pi}{2}\Bigr\} \\
    &= \#\Bigl\{j : \lambda_j > \lambda_1 - \Bigl(1 + \frac{\pi}{2}\Bigr)^{-1} a_N N^{-2/3} \Bigr\}.
  \end{align*}

  The significance of the \(1 + \pi/2\) term is that \(x - \arctan x \geq 1\) for \(x\) exceeding \(1 + \pi/2\), and hence
  \[
    \arctan \Bigl(\frac{y_+}{\lambda_1 - \lambda_j}\Bigr)
    \leq \frac{y_+}{\lambda_1 - \lambda_j} - \ind_{j \leq j^*}.
  \]

  We observe that, since \(\lambda_1 - 2 \sim N^{-2/3} \ll a_N
  N^{-2/3}\), we have a.a.s.
  
  \[
    j^*
    \geq
    j_0
    = \#\Bigl\{j : \lambda_j > 2 - \frac{1}{3} a_N N^{-2/3} \Bigr\}.
  \]

  Using part (v) of \cref{lemma three points} we have a.a.s.\ that \(j^* \geq j_0 > C a_n^{3/2}\) for some $C > 0$. 
  Consequently
  \begin{align*}
    f(y_{+})
    &= y_{+} - \frac{1}{N} \sum_{j=2}^N \arctan \Bigl(\frac{y_{+}}{\lambda_1 - \lambda_j}\Bigr) - \frac{\pi}{2N}\\
    &\geq y_+ - \frac{1}{N} \sum_{j=2}^N \frac{y_+}{\lambda_1 - \lambda_j} + \frac{j^*}{N} - \frac{\pi}{2N} \\
    &\geq O_{\Pr} \Bigl(\frac{a_N}{N}\Bigr) + \frac{j^*}{N} - \frac{\pi}{2N}.
  \end{align*}

  Since, a.a.s., \(j^* > C a_N^{3/2}\), this means that \(f(y_{+}) > 0\) a.a.s.
%
%  We have thus shown that for large \(N\), with high probability, \(f(y_{-}) < 0 < f(y_{+})\), and the result follows. 
\end{proof}

%\newpage
Having established necessary results about $\Gamma$, we also define $z_0 = \lambda_1 + \im y_0$ and
\[
  \Gamma_0
  = \{z \in \Gamma : \abs{\Im(z)} \leq y_0\}.
\]
Since $\Gamma$ can be parameterized with increasing imaginary part, this curve is connected.

Using the fact that $G(z)$ is purely real on $\Gamma$ together with
the parameterisation of $\Gamma$ with increasing imaginary part, we
have that 
\begin{align}
  \frac{1}{2\pi \im}\int_\Gamma e^{(N/\alpha)[G(z) -
  \hat{G}(\lambda_1)]} \diff z 
  % &= \frac{1}{2\pi}\int_{\Gamma} e^{N\Re[G_\alpha(z) -
  %   \hat{G}_\alpha(\lambda_1)]} \diff y \notag \\
  &\geq \frac{1}{2\pi}\textcolor{black}{\int_{-y_0}^{y_0}} e^{(N/\alpha)\Re[G(z) -
    \hat{G}(\lambda_1)]} \diff y  \notag \\
   &\geq \frac{y_0}{\pi} e^{(N/\alpha)\Re[G(z_0) -
     \hat{G}(\lambda_1)]},   \label{eq:descent-bd}
\end{align}
since the integrand is minimized on $\Gamma_0$ at the endpoints $z_0, \bar{z}_0 =
\lambda_1 \pm \im y_0$.

Appealing again to \cref{lemma lambda1statistic}, we obtain for
$\alpha = 1, 2$ 
\begin{align*}
  \log y_0 + (N/\alpha)  \Re[G(z_0) - \hat{G}(\lambda_1)]
  &= \log y_0 -\frac{1}{\alpha} \log y_0 -\frac{1}{2\alpha} \sum_{j=2}^N \log \Bigl(1 + \frac{y_0^2}{(\lambda_1 - \lambda_j)^2} \Bigr) \\
  &\geq \Bigl(1 - \frac{1}{\alpha}\Bigr) \log y_0  -\frac{y_0^2}{2\alpha} \sum_{j=2}^N \frac{1}{(\lambda_1 - \lambda_j)^2} \\
  &= \Bigl(\frac{\alpha - 1}{\alpha}\Bigr)
    \log  y_0
    % https://www.overleaf.com/project/5fa18347593fbe8c3fehttps://www.overleaf.com/project/5fa18347593fbe8c3fe3b2f43b2f4
    -\frac{y_0^2 }{2\alpha}
    O_{\Pr}(N^{4/3}). \numberthis{eq:G-bound-y0} \\
  & \geq -\Bigl(\frac{\alpha - 1}{3}\Bigr) \log N + O_{\Pr}(\log \log N),
\end{align*}
since $N^{-2/3}/\log N \leq y_0 \leq N^{-2/3} \sqrt{ \log \log N}$
a.a.s. according to \cref{lem:y0-bounds}. 

Inserting this bound into \eqref{eq:descent-bd}, we obtain 
\[
  \frac{1}{2\pi \im}\int_\Gamma e^{(N/\alpha) G(z)} \diff z
  \geq \exp\Bigl\{ (N/\alpha) \hat{G}(\lambda_1) -\Bigl(\frac{\alpha - 1}{3}\Bigr) \log N + O_{\Pr}(\log \log N)\Bigr\},
\]
which is the lower bound required to complete the proof \textcolor{black}{of \cref{prop:positive-critical} for $b=0$}.

\subsection{Limiting law in positive-critical regime}
\label{sec:limit-law-posit}

% \textcolor{black}{All arguments of sections \ref{sec:proof-proposition-b} and \ref{sec:case-b-=} go through verbatim for the critically spiked G(O/U)E cases, that is after replacing $\lambda_j$ by $\lambda_j^{(\omega)}$. In this section, we will see some differences between sub-critically spiked Wigner and critically spiked G(O/U)E cases.}

\begin{theorem}
%  [Positive-critical regime]
\label{theorem_positive}
  Suppose Assumption W holds and  \(\beta = 1 + b N^{-1/3} \log^{1/2} N\) with \(b \geq 0\).
  Then, \textcolor{black}{in the sub-critically spiked Wigner setting of theorem \ref{thm:main},}
  \[
    \frac{N}{\sqrt{\frac{\alpha}{12}\log N}}\left(F_{\alpha,N} - F(\beta)
%      \beta+\frac{1}{2}\log\beta+\frac{3}{4}
      +\frac{\log N}{12 N}\right)
    \drightarrow \Normal(0, 1) + \sqrt{\frac{3}{\alpha}} b \TW_{2/\alpha}
  \]
  with independent \(\Normal(0, 1)\) and \(\TW_{2/\alpha}\).
\fix{In contrast, in case G$^\omega$, the sign of $\log N/12N$ is
  reversed and $TW_{2/\alpha}$ is replaced by $BV_{2/\alpha}(\omega)$,
  still independent of \(\Normal(0, 1)\).}
  % \textcolor{black}{In contrast, in the critically spiked G(O/U)E setting of theorem \ref{thm: triple},
  % \[
  %  \frac{N}{\sqrt{\frac{\alpha}{12}\log N}}\left(F_{\alpha,N} - \beta+\frac{1}{2}\log\beta+\frac{3}{4}-\frac{\log N}{12 N}\right)
  %   \drightarrow \Normal(0, 1) + \sqrt{\frac{3}{\alpha}} b \mathrm{BV}_{2/\alpha}(\omega),
  % \]
  % with independent \(\Normal(0, 1)\) and \(\mathrm{BV}_{2/\alpha}(\omega)\).}
\end{theorem} 
\begin{proof}
   From \eqref{eq for z}  and Proposition \ref{prop:positive-critical} we have
\begin{equation} \label{eq:firstd}
  2 NF_{\alpha,N}
   = \alpha \log C_{\alpha,N} +  N \hat{G}(\lambda_1) - \frac{\alpha(\alpha-1)}{3}
   \log N + O_{\Pr}(\log \log N).
\end{equation}
The behavior of $N \hat{G}(\lambda_1)$ is governed by the approximation
\begin{equation}
  \label{lam1to2}
     \sum_{j = 2}^{N} \log(\lambda_1 - \lambda_j) = \sum_{j = 2}^{N} \log|2 - \lambda_j| + N(\lambda_1 - 2) + O_{\Pr}(1)  .
\end{equation}
To verify its validity, let $\Delta_N$ denote the difference between right and
left sides, without the error term. We set
\begin{align*}
  \Delta_N
    & = S_N + N(2-\lambda_1) \bigg[ \frac{1}{N} \sum_2^N
      \frac{1}{\lambda_1-\lambda_j} - 1 \bigg], \\
  S_N
    & = \sum_{j=2}^N X_{Nj}, \qquad \qquad
      X_{Nj} = \log|2-\lambda_j| - \log(\lambda_1-\lambda_j) -
      \frac{2-\lambda_1}{\lambda_1-\lambda_j}.  
\end{align*}
The second term of $\Delta_N$ is $O_{\Pr }(1)$ from \cref{lemma three points} (i) and \cref{lemma
  lambda1statistic}.
For each fixed $j$, $X_{Nj} =O_{\Pr }(1)$  since 
both $|2-\lambda_j|$ and $\lambda_1-\lambda_j$ are
$\Theta_{\Pr}(N^{-2/3})$,
the latter by Lemma \ref{lemma three points}, part (iv).

To show that $S_N$ is $O_{\Pr }(1)$, \textcolor{black}{we use the
  following criterion: if for each $\epsilon$ small there exist events
  $\mathcal{E}_{N,\epsilon}$ of probability at least $1 - \epsilon$
  such that on $\mathcal{E}_{N,\epsilon}$ for
$N > N(\epsilon)$ we have
$S_N = S_{N1}(\epsilon) + S_{N2}(\epsilon)$ with
$S_{Nk}(\epsilon) = O_{\Pr }(1)$, then $S_N = O_{\Pr }(1)$.}

First, we argue that for each $\epsilon > 0$, there exist
$k = k(\epsilon), C=C(\epsilon)>0$ such that the event
\begin{equation*}
  \mathcal{E}_{N,\epsilon}
    = \{ \lambda_1 \leq 2 + C N^{-2/3}, \lambda_k \leq 2 - C N^{-2/3} \}
\end{equation*}
has $\Pr(\mathcal{E}_{N,\epsilon}) > 1 - \epsilon$ for large enough $N$. 
Indeed, \cref{lemma three points} (i) provides $C_\epsilon$ such
  that $ \lambda_1 \leq 2 + C_{\epsilon} N^{-2/3} $ with probability
  at least $1 - \epsilon / 2$. Lemma \ref{lemma three points}
  (iii) and Markov's inequality show that
  $\Pr(\lambda_k \geq 2 - x N^{-2/3}) \leq C_x/k$.
  With $x = C_\epsilon$, 
  this can be made at most $\epsilon/2$ by choosing 
  $ k(\epsilon) = \lceil 2 C_x / \epsilon \rceil $.
  Hence $\Pr(\mathcal{E}_{N,\epsilon}) \geq 1 - \epsilon$.

Let $S_{N1}(\epsilon), S_{N2}(\epsilon)$ denote the sum in $S_N$
restricted to $j < k(\epsilon)$ and $j\geq k(\epsilon)$ respectively.
On $\mathcal{E}_{N,\epsilon}$, the sum $S_{N1}(\epsilon)$ has a
finite number of $O_{\Pr }(1)$ terms and so is itself $O_{\Pr }(1)$.
Also on $\mathcal{E}_{N,\epsilon}$, observe that
$ (2 - \lambda_1)/(\lambda_1 - \lambda_j) \geq -\frac{1}{2} $
%$ \frac{2 - \lambda_1}{\lambda_1 - \lambda_j} \geq -\frac{1}{2} $
for all $j \geq k$. Since $ |\log(1 + x) - x| \leq C_1x^2 $ for $x >
-\frac{1}{2}$ and some $C_1>0$, we have the bound 
\begin{equation*}
  |S_{N2}(\epsilon)| \leq C_1 (\lambda_1-2)^2 \sum_{j=2}^N
  \frac{1}{(\lambda_1-\lambda_j)^2}
       = O_{\Pr }(1),
\end{equation*}
from \cref{lemma three points} (i) and \cref{lemma lambda1statistic}.
This completes the proof of \eqref{lam1to2}.

Returning to $N \hat{G}(\lambda_1)$, using \eqref{lam1to2} and $\beta
- 1 = b N^{-1/3} \log^{1/2} N$, we 
obtain the key decomposition

\begin{align*}
  N \hat{G}(\lambda_1)
  &  = 2N\beta + N\beta(\lambda_1-2) - \sum_{j=2}^N
    \log(\lambda_1-\lambda_j)  \\
  & = 2N\beta 
    - \sum_{j=2}^N \log |2-\lambda_j| 
    + b \sqrt{\log N}  N^{2/3}(\lambda_1-2)
    + O_{\Pr }(1)  \\
  & = 2 N \beta     - \sum_{j=1}^N \log |2-\lambda_j| 
        - \frac{2}{3} \log N + b \sqrt{\log N} \xi_{2N} + O_{\Pr }(1),
  % & = \mu_N + \sqrt{\frac{\alpha}{3} \log N} \xi_{1N}
  %           + b \sqrt{\log N} \xi_{2N} + O_{\Pr }(1),
  %\label{eq:NG1}
\end{align*}
after adding and subtracting $\log |2-\lambda_1| = -\frac{2}{3} \log N
+ O_{\Pr}(1)$ and setting $ \xi_{2N}
= N^{2/3}\left( \lambda_{1}-2\right)$.

Combining this with \eqref{eq:firstd} and \eqref{eq:C_N-asymptotics},
we obtain
\begin{equation*}
  2 N F_{\alpha,N}
  = N(-1-\log \beta + 2 \beta) - \frac{\alpha}{6} \log N
     - \sum_{j=1}^N \log |2-\lambda_j| 
     + b \sqrt{\log N} \xi_{2N} + O_{\Pr }(\log \log N),
\end{equation*}
where we note that the coefficient of $\log N$, namely
$\frac{1}{2}\alpha - \frac{2}{3} - \frac{1}{3}\alpha(\alpha-1)$,
reduces to $-\frac{\alpha}{6}$ when $\alpha = 1$ or $2$.

Let
\begin{equation*}
  N \check \xi_N
   = \sum_{j=1}^N\log\abs{2-\lambda _{j}} - \frac{N}{2} + \frac{\alpha
     - 1}{6} \log N.
\end{equation*}
Combining the two previous displays we obtain (compare \eqref{eq:key-pos})
\begin{equation*}
  2 N F_{\alpha,N}
  = N\left(-\frac{3}{2}-\log \beta+2\beta\right)-\frac{\log N}{6}- N \check \xi_N + b \sqrt{\log N} \xi_{2N} +
  O_{\Pr }(\log \log N). 
\end{equation*}

% \textcolor{black}{Up this point, no changes are needed if we replace
%   $\lambda_j$, $j=1,2,\dotsc$, by $\lambda_j^{(\omega)}$,
%   $j=1,2,\dotsc$, and respectively redefine $\check{\xi}_{N}\mapsto
%   \check{\xi}_{N}^{(\omega)}$ and $\xi_{2N}\mapsto
%   \xi_{2N}^{(\omega)}$.}

\fix{To this point, the arguments are the same for Cases W and G$^\omega$.
Now, for case W, set}
%  Now, for the sub-critically spiked Wigner case set
  $\xi_{1N}=-N\check{\xi}_N/\sqrt{\frac{\alpha}{3}\log N}$,
  % whereas for the critically spiked G(O/U)E case set
  % $\xi_{1N}^{(\omega)}=(-N\check{\xi}_N^{(\omega)}-\frac{1}{3}\log
  % N)/\sqrt{\frac{\alpha}{3}\log N}$,
  and rewrite the \fix{last} display as
\begin{equation*}
   N F_{\alpha,N}
   = N \left(F(\beta)
%     -\frac{3}{4}-\frac{1}{2}\log\beta+\beta
     -\frac{\log N}{12 N}\right) + \sqrt{\tfrac{\alpha}{12} \log N} \, \xi_{1N} + \frac{b}{2}
  \sqrt{\log N} \, \xi_{2N} + 
  O_{\Pr }(\log \log N)
\end{equation*}
% \textcolor{black}{for the sub-critically spiked Wigner case, and as
% \begin{equation*}
%    N F_{\alpha,N}
%   = N \left(-\frac{3}{4}-\frac{1}{2}\log\beta+\beta+\frac{\log N}{12 N}\right) + \sqrt{\tfrac{\alpha}{12} \log N} \, \xi_{1N}^{(\omega)} + \frac{b}{2}
%   \sqrt{\log N} \, \xi_{2N}^{(\omega)} + 
%   O_{\Pr }(\log \log N)
% \end{equation*}
% for the critically spiked G(O/U)E case.}

%By \cref{prop:asymp-indep} 
By \cref{independence},
$(\xi_{1N},\xi_{2N})\drightarrow\mathcal{N}(0,1)\times
\mathrm{TW}_{2/\alpha}$
% \textcolor{black}{and
%   $(\xi_{1N}^{(\omega)},\xi_{2N}^{(\omega)})\drightarrow\mathcal{N}(0,1)\times
%   \mathrm{BV}_{2/\alpha}(\omega)$},
so $\xi_{1N} + c \xi_{2N} \stackrel{\rm d}{\to}
\mathcal{N}(0,1) + c \TW_{2/\alpha}$ with independent
$\mathcal{N}(0,1)$ and $\TW_{2/\alpha}$
% \textcolor{black}{and $\xi_{1N}^{(\omega)} + c \xi_{2N}^{(\omega)} \stackrel{\rm d}{\to}
% \mathcal{N}(0,1) + c \mathrm{BV}_{2/\alpha}(\omega)$ with independent $\mathcal{N}(0,1)$ and $\mathrm{BV}_{2/\alpha}(\omega)$}.
%Theorem 4.4. is proved.
This completes the proof of \cref{theorem_positive} and thus the
non-negative critical part of theorem \ref{thm:main}.
% \textcolor{black}{theorems \ref{thm:main} and \ref{thm: triple}}.

\fix{In case G$^\omega$, replace $N\check{\xi}_N$ by
  $N\check{\xi}_N + \frac{1}{3}\log N$ in the definition of $\xi_{1N}$
  and (to preserve the identity in the previous display) reverse the
  sign of $\log N/12N$. 
  Now apply \cref{independence} for case G$^\omega$ in a parallel way
  to obtain the corresponding part of theorem \ref{thm: triple}.}
\end{proof}

\section{Some key results for G(O/U)E settings}\label{sec:proof of CLT}

\subsection{\textcolor{black}{Proof of CLT's \ref{CLT1} and  \ref{CLT2}}}
\label{sec:proofs-clts}

\fix{
We prove Theorems \ref{CLT1} and \ref{CLT2} together in the spiked G(O/U)E
settings, including G$^\omega$.
Define notation for cases
$C_{\rm fix}, C_{\rm log}, J_{\rm fix}, J_{\rm crit}$ as follows
\begin{equation*}
  \gamma =
  \begin{cases}
    2 + CN^{-2/3} \log N & C>0 \quad  (C_{\rm \log}) \\
    2 + C N^{-2/3}  & C \in \mR \quad (C_{\rm fix})
  \end{cases}
  \qquad
  J_N =
  \begin{cases}
    J & J \in (0,1) \quad (J_{\rm fix}) \\
    1 - \omega N^{-1/3} & \omega \in \mR \qquad \ (J_{\rm crit})  
  \end{cases}
\end{equation*}
We will use the identity 
\begin{equation}
\label{identity1}
\log|\det
(W_{J,N}-\gamma)|
=\log|\det(W_N-\gamma)|+ 
\log|1+J_Nw^\ast(W_N-\gamma)^{-1}w|. 
\end{equation}
The scaled distribution of $\log|\det(W_N-\gamma)|$ for $J=0$ is
obtained in JKOP22, so it suffices to show that for both cases
$C_{\rm \log}$ and $C_{\rm fix}$, we have
\begin{equation}
  \label{eq:shift}
  \delta L_N(J,\gamma)  =
  \log|1+J_Nw^\ast(W_N-\gamma)^{-1}w| = 
   \begin{cases}
     \log(1-J) + o_\Pr(1) = o_\Pr(\tau_N)  & (J_{\rm fix}) \\
     -\frac{1}{3} \log N + o_\Pr(\tau_N)  & (J_{\rm crit})
   \end{cases}
\end{equation}
}
The invariance of G(O/U)E with respect to
orthogonal/unitary transformations yields
\begin{equation}
\label{quadratic}
w^\ast (W_N-\gamma)^{-1}w\overset{\mathrm{d}}{=}
\frac{\frac{1}{N}\sum_{i=1}^N|\xi_i|^2/(\lambda_i-\gamma)}{\xi^\ast
  \xi/N}
% = \frac{T_{1N}}{\xi^\ast \xi/N}
\end{equation}
where $\xi$ is an $N$-dimensional random vector with
i.i.d.~components, distributed as complex normal random variables
$\mathcal{N}_{\mathbb{C}}(0,1)$ in case of GUE and as real normal
random variables $\mathcal{N}(0,1)$ in case of GOE. Furthermore, $\xi$
is independent from $\lambda_1,\dotsc,\lambda_N$ (the eigenvalues of
$W_N$).

Denote the numerator of (\ref{quadratic}) by $T_{1N}$. We decompose
\begin{equation*}
  T_{1N} = \frac{1}{N}\sum_{i=1}^N\frac{1}{\lambda_i-\gamma} +
  \Delta_N, \qquad
  \Delta_N  = \frac{1}{N} \sum_{i=1}^N \frac{|\xi_i|^2 - 1}{\lambda_i
    - \gamma}.          
\end{equation*}
From Lemma \ref{lemma derivatives of G}   for $C_{\rm log}$ (with $C = b^2$) and Lemma \ref{lemma
  lambda1statistic} for $C_{\rm
  fix}$, we have
\begin{equation*}
  \frac{1}{N} \sum_{i=1}^N \frac{1}{\lambda_i - \gamma}
      = -1  + \delta_{N} + O_\Pr(N^{-1/3}).
\end{equation*}
where $\delta_{N} = N^{-1/3} \sqrt{C \log N}$ in case $C_{\rm log}$
and $0$ for case $C_{\rm fix}$.
The same two lemmas also yield
\begin{equation*}
  \mathbf{E}[ \Delta_N| \lambda] = 0, \qquad
  \mathbf{V}[ \Delta_N| \lambda] =
  \frac{1}{N^2}\sum_{i=1}^N\frac{\alpha}{(\lambda_i-\gamma)^2}
  =O_{\Pr}(N^{-2/3})   
\end{equation*}
Hence $\Delta_N = O_\Pr(N^{-1/3})$ (\textcolor{black}{see \cref{lem:conditional-unconditional}}) and so
$T_{1N} = -1 + \delta_N + O_\Pr(N^{-1/3})$.
Combined with
%$\xi^\ast\xi/N=1+O_{\Pr}(N^{-1/2})$, 
% \bigskip
% \bigskip
% Now,
\begin{equation}
\label{denominator}
\xi^\ast\xi/N=1+O_{\Pr}(N^{-1/2}),
\end{equation}
\fix{we arrive at
\begin{equation}
  \label{eq:p-decomp}
  \begin{split}   
    1 + J_N w^\ast(W_N-\gamma)^{-1} w
    & = 1 + J_N T_{1N}/(1+O_\Pr(N^{-1/2})) \\
   & = 1 + J_N(-1+ \delta_{N} + O_\Pr(N^{-1/3}))[1+O_\Pr(N^{-1/2})].
  \end{split}
\end{equation}
}

\fix{In case $J_{\rm fix}$, the right side equals $1 - J + o_\Pr(1) =
o_\Pr(\tau_N)$ in both cases $C_{\rm log}$ and $C_{\rm fix}$ giving
(\ref{eq:shift}).
% so the
% normalized log-det statistics for $W_{J_N}$ and $W_N$ have the same
% limiting distribution.
We turn to the critical case $J_{\rm crit}$. From (\ref{eq:p-decomp}),
we have
\begin{equation}
  \label{eq:NVN}
  N^{-1/3} V_N
   := 1 + J_N w^\ast(W_N-\gamma)^{-1} w
  = \delta_{N} + \omega N^{-1/3} + O_\Pr(N^{-1/3}).
\end{equation}
In case $C_{\rm log}$, the right side equals $N^{-1/3}(\sqrt{C \log N}
+ O_\Pr(1))$, and so
$\delta L_N(J,\gamma) = -\tfrac{1}{3} \log N + o_\Pr(\tau_N)$
% \begin{equation*}
%   \log (1 + J_N w^\ast(W_N-\gamma)^{-1} w) = -\tfrac{1}{3} \log N + o(\tau_N),
% \end{equation*}
as required. In case $C_{\rm fix}$, we show that
$V_N = \Theta_\Pr(1)$, for then
$\delta L_N(J,\gamma) = \log( N^{-1/3} V_N) = -\tfrac{1}{3} \log N + O_\Pr(1).$
  % \begin{equation*}
  %  \log|1+Jw^\ast (W_N-\gamma)^{-1}w|   = \log( N^{-1/3} V_N)
  % =-\tfrac{1}{3}\log N +O_{\Pr}(1).
  % \end{equation*}
The upper bound $V_N = O_\Pr(1)$ follows from (\ref{eq:NVN}) since
$\delta_{N}$ vanishes in this case.
}
  
\fix{Now set
$N^{-1/3} \tilde{V}_N = 1 + J_NT_{1N}$,
% \begin{equation*}
%   N^{1/3} \tilde{V}_N
%     = 1 + \frac{J}{N} \sum_{i=1}^N \frac{|\xi_i|^2}{\lambda_i -
%     \gamma},
% \end{equation*}
so that from (\ref{eq:p-decomp})
% (\ref{quadratic}), (\ref{denominator}) and (\ref{eq:T1Nasy}),
we have
$V_N - \tilde{V}_N  = J_N N^{1/3} T_{1N} O_\Pr(N^{-1/2})
  = O_\Pr(N^{-1/6})$.
% \begin{equation*}
%   V_N - \tilde{V}_N = N^{1/3} \left( \frac{J T_{1N}}{\xi^\ast \xi/N} -
%     J T_{1N} \right)
%   = J N^{1/3} T_{1N} O_\Pr(N^{-1/2})
%   = O_\Pr(N^{-1/6}).
% \end{equation*}
We now show that $\tilde{V}_N^{-1} = O_\Pr(1)$ 
so that $V_N^{-1} = O_\Pr(1)$ as required. 
$\tilde{V}_N$ has representation
\[
\tilde{V}_N =
N^{1/3}\left(\frac{J_N}{N}\sum_{i=1}^N\frac{|\xi_i|^2}{\lambda_i-\gamma}+1\right)
=\frac{J_N|\xi_1|^2 -\psi_N}{N^{2/3}(\lambda_1-\gamma)} 
= \frac{X_N - Y_N}{Z_N}, 
\]
where $\psi_N$ depends on $N$, $\gamma$, $\omega$, $\lambda_1,\dotsc,\lambda_N$, and
$\xi_i$ with $i=2,\dotsc,N$, but is stochastically independent from
$\xi_1$.
% It suffices to show that $\tilde{V}_N^{-1} = O_{\Pr}(1)$.
Since
\[
Z_N =
N^{2/3}(\lambda_1-\gamma)=N^{2/3}(\lambda_1-2-CN^{-2/3})\drightarrow \TW_{2/\alpha}-C, 
\]
and $X_N = J_N|\xi_1|^2 \sim \textcolor{black}{J_N}\frac{\alpha}{2}\chi^2(2/\alpha)$
is independent of $Y_N = \psi_N$ and
has a density $f_{X_N}(x)$ such that $x^{1/2}f_{X_N}(x)$ is bounded above by a constant $B = B(\alpha,\textcolor{black}{\omega})$ \textcolor{black}{for all sufficiently large $N$},
our conclusion follows from the next lemma (proof in Section \ref{sec:OP1proof}.)
}
\begin{lemma} \label{lem:OP1}
  Suppose $\{X_N \}, \{Y_N\}, \{Z_N\}$ are real valued random variables such
  that
  (i) $Z_N = O_{\Pr}(1)$,
  (ii) $X_N$ is non-negative and, \textcolor{black}{for all sufficiently large $N$}, has density $f_{X_N}(x)$ w.r.t. Lebesgue measure such that $\sqrt{x}f_{X_N}(x)$ is bounded by $B$,
  (iii) $X_N$ and $Y_N$ are independent, for each $N$. Then
  \begin{equation*}
    W_N = Z_N/(X_N - Y_N) = O_{\Pr}(1).
  \end{equation*}
\end{lemma}

\subsection{Proof of \cref{independence} for
  G(O/U)E} \label{sec:independence for Gaussian case}
The proof is given here for $W_N$ from G(O/U)E. The extension to cases
G$^\omega$ and W follows directly from (\ref{identity1})
(Sections \ref{sec:proof-prop-ref} and \ref{sec:proof-indep-wigner} in Appendix).

The marginal convergences follow from \cref{CLT2} and \cref{lemma three points} (i). Hence, we only need to prove the asymptotic independence.  Our proof is based on the tridiagonal representation of GUE and GOE. Recall that the eigenvalues of the matrix
\begin{equation}\label{tridiag_M_N}
    \textcolor{black}{\sqrt{n}}M_{n} = 
    \begin{pmatrix}
        a_1 & b_1 & & \\
        b_1 & a_2 & b_2 &   \\
         & b_{2} & \ddots & \ddots \\
         && \ddots & \ddots & b_{n-1} \\
         &&& b_{n-1} & a_{n}
    \end{pmatrix},
\end{equation}
with independent \( a_{i} \sim \mathcal{N}(0, \alpha) \) and \( b_{i}
\sim \chi(2i/\alpha) / \sqrt{2 /\alpha} \), are distributed as
eigenvalues of GUE matrices for $\alpha = 1$ and as eigenvalues of GOE
matrices for $\alpha = 2$. (e.g. \cite[][Ch. 4]{anderson2010introduction},
\cite{TaoVu}. 
\textcolor{black}{Therefore, in our proof, we may and will reinterpret
  $\lambda_j$ as the eigenvalues of $M_N$.} 
%(both scaled by $1/\sqrt{N}$, so that the bulk of the spectrum is contained in $[-2, 2]$).

% For a proof of the tridiagonal representation for general Gaussian $\beta$-ensembles, we refer the reader to Chapter~4 of \cite{anderson2010introduction}. See also a simple proof in \cite{TaoVu} that works only for GOE and GUE. 

%We further observe that the triangularization transform constructed in \cite{TaoVu} does not change the bottom right value of the original %Wigner 
%matrix. Hence to obtain deformed GUE or GOE in a tridiagonal form \textcolor{black}{(that is, to obtain tridiagonal matrices whose eigenvalues are distributed as those of $\sqrt{N}W_{J,N}$ for $\alpha=1$ or $\alpha=2$) }, we can simply take $a_N \sim \mathcal{N}(\sqrt{N}J, \alpha )$, and the rest of the elements remain the same.\footnote{Note that the ``direction of the spike'' $w$ does not influence the distribution of the eigenvalues. Hence, we can take $w$ having all but the last elements equal to zero without loss of generality.}  

In what follows, we establish the representations
\[
    \xi_{1N} = X_N/\sqrt{\frac{\alpha}{3}\log N} + o_{\Pr}(1),
    \qquad
    \xi_{2N} = \textcolor{black}{N^{2/3}(Y_N-2)} + o_{\Pr}(1),
\]
where \( X_N \) depends only on \( a_i, b_{i-1} \) with \( i \leq N - 2N^{1/3} \log^3 N\), and \( Y_N \) depends only on \( a_i, b_{i-1} \) with \( i > N - 2 N^{1/3} \log^3 N\). Since $X_N$ and $Y_N$ are independent, %each of $\xi_{1N}$, $\xi_{2N}$ converges to a non-trivial distribution, 
such representations yield \cref{independence}
% \textcolor{black}{for the case of unspiked G(O/U)E. We will then extend the proof to the case of critically spiked G(O/U)E.}

\textit{Representation for $\xi_{1N}$ (log-determinant).} 
In \cite{johnstone2020logarithmic}, the derivation of the CLT for
$\xi_{1N}$ \textcolor{black}{(\cref{CLT2} here)} is based on the
tridiagonal representation. Let us recall needed elements of that proof. \textcolor{black}{Equation (18) of \cite{johnstone2020logarithmic}} shows that %this equation has label eq:strategy in JKOP20
\begin{equation}
\label{eq:18 of JKOP}
    \sum_{j = 1}^{N} \log |2 - \lambda_j| = \sum_{j = 1}^{N} \log |2 + N^{-2/3} \bar{\sigma}_N - \lambda_j | - N^{1/3} \bar{\sigma}_{N} + \textcolor{black}{O_{\Pr}(\bar{\sigma}_{N}^2)},
\end{equation}
where \( \bar{\sigma}_N := (\log\log N)^{3} \). Further, the sum on the right hand side of the above display can be well approximated by a deterministic shift of a linear combination of the independent variables \( a_i, b_i^2 \).

To be precise, consider \( c_i = (b_{i}^2 - i)/{\sqrt{i}} \), so that \( \E c_i = 0 \) and \( \mathrm{Var}(c_i) = \alpha \). For \( \theta_N = 1 + N^{-2/3} \bar{\sigma}_N / 2 \), define \textcolor{black}{recursively}
\textcolor{black}{
\[
L_i=\xi_i+\gamma_i L_{i-1}\qquad \text{for } i\geq 1,
\]}
where %for $i = 3, \dots, N$,
\[
    \xi_{i} = \alpha_{i} + \beta_{i} \, ,
    \qquad
    \alpha_{i} = \frac{a_i}{\sqrt{N}\theta_N r_i},
    \qquad
    \beta_{i} = \sqrt{\frac{\gamma_i}{N}} \frac{c_{i-1}}{\theta_N r_{i-1}} \, 
\]
\textcolor{black}{with $\beta_1:=0$,}
and
\begin{align}\label{r_definition}
    r_{i} &= 1 + \sqrt{1 - \frac{i-1}{N \theta_N^2}}, \qquad
    m_{i} = 1 - \sqrt{1 - \frac{i-1}{N \theta_N^2}}, \qquad
    \gamma_{i} = \frac{m_i}{r_{i}}.
\end{align}
Then, equations \textcolor{black}{(49), (50), and the last equation of section 4.1.6} of \cite{johnstone2020logarithmic} show that %these equations and the section has labels, respectively, \label{log expansion}, \label{replacing R by L}, and \label{sec:CLT forD}
\begin{equation*}%\label{sigma_N log det}
    \sum_{i = 1}^{N} \log |2 + N^{-2/3} \bar{\sigma}_N - \lambda_i | = \frac{N}{2} + N^{1/3} \bar{\sigma}_N - \frac{\alpha - 1}{6}\log N - \sum_{i = 1}^{N} L_{i} + O_{\Pr}(\textcolor{black}{\bar{\sigma}_N^2}) \, .
\end{equation*}
\textcolor{black}{Combining this with \eqref{eq:18 of JKOP}and recalling the definition of $\xi_{1N}$, we obtain
\begin{equation}\label{eq:xi1N via L}
    \xi_{1N}=\sum_{i=1}^N L_i/\sqrt{\frac{\alpha}{3}\log N}+o_{\Pr}(1).
\end{equation}
Now we are ready to prove the following.}

\begin{lemma}
We have that $ \xi_{1N}=X_N/\sqrt{\frac{\alpha}{3}\log N} + o_{\Pr}(1) $,
where $X_N$ depends only on $a_i, b_{i-1}$ for $i \leq N - 2N^{1/3}\log^3 N$.
\end{lemma}

\begin{proof}
Let us rewrite the sum from \eqref{eq:xi1N via L} in the following form:
\[
    \sum_{i = 1}^{N} L_{i}  = \sum_{i = 1}^{N} (\xi_{i} + \gamma_{i} \xi_{i-1} + \dots + \gamma_{i} \dots \gamma_{2} \xi_{1}) = \sum_{i = 1}^{N} g_{i + 1} \xi_{i},
\]
where \( g_{i} = 1 + \gamma_{i} + \dots + \gamma_{i} \dots \gamma_{N} \).
Set \(m = \lfloor N - 2N^{1/3} \log^3 N \rfloor \) and consider 
$
    X_N = \sum_{i = 1}^{m} g_{i + 1} \xi_{i} \, ,
$
so that it only depends on \( (\xi_i)_{i \leq m}, \) and hence on \( (a_i)_{i \leq m}, (b_{i})_{i \leq m-1}\). Since \( \xi_{i} \) are independent and centred,
\begin{align*}
    \E \bigg( \sum_{i = 1}^{N} L_i - X_N \bigg)^{2} = \E \bigg( \sum_{i > m} g_{i + 1} \xi_{i} \bigg)^{2} \leq \left(\max_{i} \E \xi_{i}^{2}\right) \sum_{i > m} g_{i + 1}^{2}
    = O(N^{-1}) \sum_{i > m} g_{i + 1}^{2} \, .
\end{align*}
By Lemma~{6} in {\cite{johnstone2020logarithmic}} we have, for \( N \) large enough,%this is lemma with \label{lemma si}
\[
    g_{i} < \frac{r_i}{r_{i} - 1} \leq \frac{2}{r_i - 1} \, ,
    \qquad
    i = 1, \dots, N \, .
\]
Further, since \( (r_{i} - 1)^{-1} \) is an increasing sequence, we have
\begin{align*}
    \sum_{i > m} g_{i + 1}^{2} &< \sum_{i = m + 1}^{N + 1} \frac{4}{(r_{i + 1} - 1)^{2}} = 4 \sum_{i = m + 1}^{N + 1} \left(1 - \theta_N^{-2} \frac{i}{N} \right)^{-1} < 4 N \int_{\frac{m + 1}{N}}^{\frac{N + 2}{N}} (1 - \theta_N^{-2} x)^{-1} \diff x \, ,
\end{align*}
where we use the inequality \( \frac{N + 2}{N} \theta_N^{-2} < 1 \),
which holds for large enough \( N \).
Since \( m + 1 > N - 2N^{1/3} \log^3 N \), the integral is bounded
above by
\begin{equation*}
  \int_{1 - 2N^{-2/3} \log^3 N}^{1 + 2N^{-1}} (1 - \theta_N^{-2}
  x)^{-1} \diff x
   = \theta_N^2 \log \frac{2\log^3 N + \bar{\sigma}_N + O(N^{-1/3})}{\bar{\sigma}_N + O(N^{-1/3})} 
     = O\left(\log\log N\right).
\end{equation*}
where we used \( \theta_N = 1 + N^{-2/3} \bar{\sigma}_N / 2\)
and $1 - \theta_N^{-2}(1+N^{-2/3}v) = N^{-2/3}[\bar{\sigma}_N - v
  +O(N^{-1/3})]$ for $v = O(\log^3 N)$.
% \begin{align*}
%     \int_{\frac{m + 1}{N}}^{\frac{N + 2}{N}} (1 - \theta_N^{-2} x)^{-1} \diff x
%     &\leq \int_{1 - 2N^{-2/3} \log^3 N}^{1 + 2N^{-1}} (1 - \theta_N^{-2} x)^{-1} \diff x \\
%     &= \theta_N^{2} \log \frac{1 - \theta_N^{-2} (1 - 2N^{-2/3} \log^3 N)}{1 - \theta_N^{-2}(1 + 2N^{-1})} \\
%     %&= \theta_N^2 \log \frac{\sigma_N N^{-2/3} + \sigma_N^2 N^{-4/3}/4 + 2N^{-2/3} \log^3 N}{\omega_N N^{-2/3} + \sigma_N^2 N^{-4/3}/4 - 2N^{-1}} \\
%     &= \theta_N^2 \log \frac{2\log^3 N + \bar{\sigma}_N(1 + o(1))}{\bar{\sigma}_N (1 + o(1))} \\
%     & = O\left(\log\log N\right) \, .
% \end{align*}
Collecting  all together, we get that
\[
    \E \bigg( \sum_{i = 1}^{N} L_{i} - X_{N} \bigg)^{2} =
    O\left(\log\log N\right). 
\]
\textcolor{black}{Combining this with \eqref{eq:xi1N via L} yields the lemma.}
\end{proof}

\textit{Representation for $\xi_{2N}$ (the largest eigenvalue).}
Since the entries of the tridiagonal matrix \eqref{tridiag_M_N} become larger towards the bottom right corner, we expect that the largest eigenvalue \( \lambda_1 \) does not depend too much on the upper values. \cite{edelman2005random} use a heuristic argument and numerical evidence to suggest that the bottom \( 10N^{1/3} \times 10 N^{1/3} \) minor's largest eigenvalue is a good approximation of \( \lambda_1 \). For our purpose, it is sufficient to multiply \(N^{1/3}\) by  \textcolor{black}{\(2\log^3 N\)} instead of \(10\). We prove the following lemma.

\begin{lemma}\label{tilde_lambda}
Let \( Y_N \) be the largest eigenvalue of the bottom-right minor of \(M_N\) of size \( l > 2 N^{1/3} \log^{3} N \). Then for any $K>0$,
\[
    |\lambda_{1} - Y_N| = O_{\Pr}(N^{-K}) \, .
\]
\end{lemma}

\begin{proof} 
Let $ Y_N:=\tilde{\lambda}_1$ be the largest eigenvalue of
\begin{equation*}
  \sqrt{N} \widetilde{M}_{l,N} =
  \begin{pmatrix}
    0 & 0 \\
    0 & T_{l,N}
  \end{pmatrix}
\end{equation*}
where $T_{l,N}$ is the block of $\sqrt{N} M_N$ in (\ref{tridiag_M_N})
formed by the rows and columns from $N-l+1$ to $N$.
% \[
%     \textcolor{black}{\sqrt{N}}\widetilde{M}_{l,N} = 
%     \begin{pmatrix}
%         0 & 0 \\
%         0 & \ddots & \ddots \\
%         &\ddots &0 & 0 & & \\
%         &&0 & a_{N-l + 1} & b_{N -l + 1} &   \\
%          &&& b_{N - l + 1} & \ddots & \ddots \\
%          &&&& \ddots & \ddots & b_{N-1} \\
%          &&&&& b_{N-1} & a_{N}
%     \end{pmatrix}.
% \]
Then $ \lambda_1 \geq \tilde{\lambda}_1 $ and it is sufficient to bound the difference \( \lambda_1 - \tilde{\lambda}_{1} \) from above. 

Let $ v $ be the normalized \textcolor{black}{principal} eigenvector of the original matrix $M_N$, so that $ v^{\top} M_N v = \lambda_1 $, and we also have $ \tilde{\lambda}_1 \geq v^{\top} \widetilde{M}_{l,N} v $.
We  then have
\begin{align}
    \lambda_{1} - \tilde{\lambda}_{1} &\leq \lambda_{1} - v^{\top} \widetilde{M}_{l,N} v = v^{\top} (M_N - \widetilde{M}_{l,N}) v \notag \\
    & = \sqrt{\frac{N-l}{N}} v_{:N-l}^{\top} M_{N-l} v_{:N-l} + \textcolor{black}{\frac{{2 b_{N-l}}}{\sqrt{N}}} v_{N-l + 1} v_{N-l} \, ,\label{lam1-lam1tilde}
\end{align}
where \( v_{(:N-l)} = (v_1, \dots, v_{N-l})^{\top} \) and
$\sqrt{N-l}M_{N-l}$ is an instance of (\ref{tridiag_M_N}).
% \[
%     \textcolor{black}{\sqrt{N-l}}M_{N-l} = 
%     \begin{pmatrix}
%         a_1 & b_1 & & \\
%         b_1 & a_2 & b_2 &   \\
%          & b_{2} & \ddots & \ddots \\
%          && \ddots & \ddots & b_{N-l - 1} \\
%          &&& b_{N-l - 1} & a_{N - l}
%     \end{pmatrix}\, .
% \]
\textcolor{black}{Note that $\norm{M_{N-l}}=O_{\Pr}(1)$ and \( b_{N-l}/\sqrt{N} = O_{\Pr}(1)\).}

\textcolor{black}{A proof of the following auxiliary lemma is given in  \cref{section_main_eigenvector}. The lemma controls the initial $N-2N^{1/3}\log^3N$ components of the principal eigenvector $v$.} 

\begin{lemma}\label{lemma_main_eigenvector}
Let \( v \) be a principal eigenvector of \( M_N\), \textcolor{black}{standardized to have unit Euclidean norm.} Then for any \( D > 0\)
\[
    \max_{i \leq N - 2N^{1/3} \log^3 N} \textcolor{black}{|v_{i}|} = O_{\Pr}(N^{-D}) \, .
\]
\end{lemma}

Since $ l > 2N^{1/3} \log^3 N$, we have $ \max_{i \leq N-l + 1} |v_i| = O_{\Pr}(N^{-D}) $. Therefore, \( \| v_{:N-l}\|^2 \leq N O_{\Pr}(N^{-2D}) \) and \( \max\{|v_{N-l}|,|v_{N-l+1}|\} \leq O_{\Pr}(N^{-D}) \). \textcolor{black}{Hence, from \eqref{lam1-lam1tilde},}
\[
    \lambda_{1} - Y_N \leq O_{\Pr}(1) \times O_{\Pr}(N^{-2D + 1}) + O_{\Pr}(1) \times O_{\Pr}(N^{-D}) \leq O_{\Pr}(N^{-D + 1}) \, ,
\]
where taking \( D = K+5/3 \) yields \cref{tilde_lambda}.
\end{proof}
This completes the proof of \cref{independence} for G(O/U)E with
$J=0$.

\bigskip
\textbf{Acknowledgements} \ We are grateful to the reviewers for
comments leading to improvements in presentation. 
The first and fourth authors were supported in part by NSF grant DMS 1811614.

\section*{Declarations}

\textbf{Data Availability} \ Data sharing not applicable to this article as no datasets were generated or analyzed during the current study.

\bigskip
\textbf{Conflict of Interest}  The authors have no relevant financial or non-financial interests to disclose. 

\printbibliography
\newpage
\begin{center}
  \textsc{Appendix}
\end{center}

\fix{This Appendix contains the proofs for Lemmas \ref{lemma three points}-\ref{lemma
  lambda1statistic} and (the rest of) Proposition \ref{independence}.
Section \ref{sec:proof key lemmas} does Gaussian cases. Key tools
are bounds for the one-point function: in the bulk sharp uniform order
$N^{-1}$ bounds due to G\"otze and Tikhomirov \cite{Gotze2005} for the error in
approximation by the semi-circle, and at the edge, bounds based on
Tracy-Widom asymptotics.
Section \ref{sec:wigner} deals with the Wigner case, mostly by Lindeberg
swapping, using elaborations of now standard methods as organized 
in JKOP22.
Some of the technical ingredients are further deferred and collected in 
Section \ref{sec: appendix}.
}

\section{Proof of the key lemmas for \fix{two} G(O/U)E
  cases}
\label{sec:proof key lemmas}

In this section, we prove lemmas \ref{lemma three points}-\ref{lemma
  lambda1statistic} and \cref{independence}
\fix{in two specific cases}. 
  \fix{The first is when $\lambda_1,\dotsc,\lambda_N$ are the eigenvalues of a scaled G(O/U)E without any spike ($J=0$). To contrast this case with the general Wigner cases, we will denote eigenvalues of such special scaled G(O/U)E matrices $W_{J,N}=W_{0,N}=W_N$ as $\mu_1\geq\cdots\geq\mu_N$, instead of $\lambda_1\geq\dotsc\geq\lambda_N$.
In section \ref{sec:wigner}, we will extend the proof to general
sub-critically spiked Wigner matrices $W_{J,N}$ satisfying Assumption
W (Wigner case) by using the Lindeberg swapping technique. The
eigenvalues of such $W_{J,N}$ will be again denoted as
$\lambda_1\geq\cdots\geq\lambda_N$, as in previous sections.}

\fix{The second case is case G$^\omega$: for clarity we write
  $\lambda_1^{(\omega)},\dotsc,\lambda_N^{(\omega)}$ for the eigenvalues
  of a critically spiked scaled G(O/U)E.}

\subsection{\textcolor{black}{Some useful tools}}
\textcolor{black}{ Let us first describe two important background results that we are going to use in the Gaussian part of the proof.}  

\smallskip
\textbf{One-point correlation function.} \
Let $\rho_N$ be the level density or one-point function of \textcolor{black}{scaled} GUE.
Then the  expectation of a linear spectral statistic
is given by
\begin{equation}
  \label{eq:linear-stat}
  \E \Big[ N^{-1} \sum_{i=1}^N f(\mu_i) \Big] = \int f(\mu)
    \rho_N(\mu) \diff \mu. 
\end{equation}

A key tool in approximating such expectations will be a uniform bound,
due to G\"otze and Tikhomirov, for 
the deviation of the one-point function in GUE %(and GOE)
from the semicircle density
$p_{\rm SC}(x) = (2\pi)^{-1} \sqrt{4-x^2} \mathbf{1}_{|x| \leq 2}$.
Indeed,
%\cite{Gotze2005}
\cite[Theorem 1.2]{Gotze2005}
show the existence of
absolute constants $\gamma, C > 0$ such that for all $|x| \leq 2 -
\gamma N^{-2/3} $,
\begin{equation}\label{goetze_tihkomirov1}
    |\rho_{N}(x) - p_{SC}(x)| \leq \frac{C}{N(4 - x^2)} \, .
\end{equation}
In addition, the one-point function decays at least exponentially at
the edge:
% it follows from \cite{Johnstone2012},
% see Section \ref{one-point-tail}, that
for all $ s > -\gamma $, for large enough $N$ (see \cref{Johnstone Ma
  derivation}),
\begin{equation}\label{johnstone_ma1}
    \rho_{N}(2 + sN^{-2/3}) \leq C(\gamma) N^{-1/3} e^{-2s}\, .
\end{equation}
% For the sake of completeness, we demonstrate how this inequality
% follows from the results of Johnstone and Ma in Section FILL
% IN. Furthermore, 
A similar bound holds at the
negative edge, by symmetry. Corresponding bounds also hold
for $p_{SC}$.

\smallskip
\textbf{Comparing GOE with GUE.} \ 
Forrester and Rains \parencite*{forrester2001inter} found a relation between the eigenvalues of GOE
and GUE that can be used to compare linear statistics from the two
ensembles.
Let \textcolor{black}{$\tilde{Z}_\alpha$ denote scaled} $N \times N$ GUE and GOE for $\alpha = 1,2$
respectively.
Given $f: \R \to \R$, let $f(\tilde{Z}_\alpha) = \sum_{i=1}^N
f(\mu_{\alpha,i})$, \textcolor{black}{where $\mu_{\alpha,i}$ are the eigenvalues of $\tilde{Z}_\alpha$}, and let $\TV(f)$ denote the total variation of
$f$. 
In \cite[][\textcolor{black}{Lemma 33} and Corollary 34]{johnstone2020logarithmic} it
is shown that
\begin{align}
  \label{eq:compare-means}
  | \E f(\tilde{Z}_1) - \E f(\tilde{Z}_2)|
  & \leq O(\TV (f)) \\
  \Var f(\tilde{Z}_2)
  & \leq 2 \Var f(\tilde{Z}_1) + 2 \TV^2(f)
    \label{eq:compare-vars}
\end{align}
Also, if $f_N$ is a series of functions such that
\begin{equation}
  \label{eq:GUE-comp}
f_N\left(\tilde{Z}_1\right)=a_N+O_{\Pr}(b_N),
\end{equation}
for some sequences $a_N$ and $b_N$, then,
\begin{equation}
  \label{eq:GOE-comp}
f_N\left(\tilde{Z}_2\right)
  = a_N+O_{\Pr}(b_N+\mathrm{TV}(f_N)).
\end{equation}

\subsection{Proof of lemma~\ref{lemma three points} for G(O/U)E}\label{sec: proof three poins}

Part (i) is \fix{shown} in the G(O/U)E cases with $J=0$
\textcolor{black}{and $J=1-\omega N^{-1/3}$} \fix{in the references
  already cited.} 

Part (ii)  follows e.g.~from the convergence of $N^{2/3}(\mu_1-2)$
to $\TW_{2/\alpha}$. \textcolor{black}{Its equivalent for $\lambda_1^{(\omega)}$ follows from $N^{2/3}(\lambda_1^{(\omega)}-2)=O_\Pr(1)$, which is a consequence of part (i).}

For GUE, part (iii)  follows from the one-point function decay bound
\eqref{johnstone_ma1} and \eqref{eq:linear-stat} applied to the
counting function statistic built from $f_N(\mu) = \ind \{ \mu
\geq 2 - x N^{-1/3} \}$. The extension to GOE follows from the
comparison bound \eqref{eq:compare-means}. \textcolor{black}{The equivalent of part (iii) for $\lambda_j^{(\omega)}$ follows by interlacing: $\#\left\{j:\; \lambda_j^{(\omega)} \geq 2-xN^{-2/3}\right\}$ and $\#\left\{j:\; \mu_j \geq 2-xN^{-2/3}\right\}$ differ by at most one.}

To see that part (iv) holds, consider the \textcolor{black}{stochastic
  Airy} operator\footnote{ \textcolor{black}{We give the definition in
  \cite{anderson2010introduction} (making our replacement $\alpha =
  2/\beta$), who use the negative of that of \cite{BloVirI}. Thus
  $\mathbf{H}_\alpha$ here corresponds to
  $-\mathcal{H}_{2/\alpha,\omega}$ in \cite{BloVirI}.}}
\[
\mathbf{H}_\alpha=\frac{\diff^2}{\diff x^2}-x+\sqrt{2\alpha}B^{\prime}_{x},
\]
where $B^{\prime}_{x}$ is the \textcolor{black}{``derivative''} of the Brownian motion on $(0,\infty)$, and the operator acts on \textcolor{black}{a} %some 
Hilbert space $\mathcal{L}_{*}$, \textcolor{black}{obtained from smooth continuous functions supported on $(0,\infty)$ via completion with respect to the inner product
\[
\langle f,g\rangle_{*}=\int_0^{\infty}f'(x)g'(x)\mathrm{d}x+\int_0^\infty(1+x)f(x)g(x)\mathrm{d}x.
\]}%which consists of continuous functions supported on $(0,\infty)$,
\textcolor{black}{Also let $\langle f, g \rangle_2 = \int_0^\infty f(x) g(x)
\mathrm{d}x$ and $\bar{B}_x = \int_x^{x+1} B_y \mathrm{d}y$
and $\tilde{B}_x = B_x - \bar{B}_x$. Then $\mathbf{H}_\alpha$ is
defined on $f \in \mathcal{L}_{*}$ through the quadratic form
\begin{equation*}
  \langle f, \mathbf{H}_\alpha f \rangle_2 
  = \| f' \|_2^2 + \| \sqrt{x} f(x) \|_2^2
  - \sqrt{2 \alpha} \big[ \langle f, \bar{B}_x' f \rangle_2
    - 2 \langle f', \tilde{B} f \rangle_2 \big].
\end{equation*}
see pp.~308--311 in \cite{anderson2010introduction}.}
The following result is theorem 4.5.42 in \cite{anderson2010introduction} for the special case of just the two top eigenvalues,
\begin{equation}
\label{eq:top-eigs-joint-convergence}
    \left( N^{2/3} (\mu_1 - 2), N^{2/3} (\mu_2 - 2) \right) \drightarrow (\Lambda_1, \Lambda_2) \, ,
\end{equation}
where \(\Lambda_1, \Lambda_2\) are the top two eigenvalues of random operator $\mathbf{H}_{\alpha}$. In addition, in lemma~{4.5.47} of \cite{anderson2010introduction} it is shown that the operator has simple spectrum with probability one. This implies (iv). \textcolor{black}{The equivalent of (iv) for $\lambda_1^{(\omega)}-\lambda_2^{(\omega)}$ also holds. Indeed, from part (i), we have $N^{2/3}(\lambda_1^{(\omega)}-\lambda_2^{(\omega)})\drightarrow \Lambda_1-\Lambda_0=\Theta_\Pr(1)$ since the spectrum of %$\mathcal{H}_{2/\alpha}(\omega)$ 
\textcolor{black}{$\mathcal{H}_{2/\alpha,\omega}$} is simple a.s.}

%Part (iv) for GUE follows from Theorem 3.1.4 in  \cite{anderson2010introduction}. For GOE, it follows e.g. from eqs (2.4) and (2.19a) of \cite{bornemann2010numerical}.

For part (v),
\textcolor{black}{Let $\mathcal{N}_{b_N} = \#\{j : \mu_j > 2 - b_N
  N^{-2/3}\}$ with $b_N\rightarrow\infty$, so that $b_N = O(N^\epsilon)$ for all $\epsilon>0$. 
Lemmas 2.2 and 2.3 of \cite{Gustavsson2005} yield that, in the GUE case,
\begin{align*}
  \E \mathcal{N}_{b_N}
  & = \frac{\textcolor{black}{2}}{3 \pi} b_N^{3/2} + O(1) \\
  \Var \mathcal{N}_{b_N}
  & = \frac{3}{4 \pi^2}(\log b_N - \log 2)(1+o(1))
    \lesssim \log b_N.
\end{align*}}
Since the function $f_N(\mu) = \ind[\mu \geq  2 - b_N N^{-2/3}]$ has
$\TV(f) = 1$, 
\cref{eq:compare-means,eq:compare-vars} yield mean and variance
bounds of the same order in the GOE case. 
From Chebychev's inequality, if $\kappa < \kappa_0 = \textcolor{black}{2}/(3\pi)$, then
in both cases
\begin{equation*}
  \Pr \{ \mathcal{N}_{b_N} \leq \kappa b_N^{3/2} \}
  \lesssim \frac{\log b_N}{(\kappa_0 - \kappa)^2 b_N^3}
  \to 0.  \qedhere
\end{equation*}
\textcolor{black}{The equivalent of (v) for $\lambda_j^{(\omega)}$ immediately follows by interlacing 
\[
\#\left\{j:\; \lambda_j^{(\omega)} \geq 2-xN^{-2/3}\right\}\geq\#\left\{j:\; \mu_j \geq 2-xN^{-2/3}\right\}.
\]}

\subsection{Proof of lemma \ref{lemma derivatives of G}  for G(O/U)E}
\label{sec: proof lemma deriv G}

First, recall that for $l \geq 1$,
\begin{equation} \label{eq:Gl}
  G^{(l)}(\gamhat)
  = (1 + \sgn(b) b_N N^{-1/3}) \delta_{l1}
    +\frac{c_l}{N} \sum_{j=1}^N (\gamhat - \mu_j)^{-l},
\end{equation}
where \textcolor{black}{$b_N = |b| \sqrt{\log N}$}, $\delta_{l1} = 1$ if $l=1$ and $0$ otherwise,
and $c_l = (-1)^l(l-1)!$.
%We have set $b_N = |b| \sqrt{\log N}$, but it will be seen that the
%argument requires only $b_N \nearrow \infty$ as $N \to \infty$.
Henceforth, we write $\gamhat = 2 + \epsilon_N$, with $\epsilon_N =
b_N^2 N^{-2/3}$. 
We also set $\eta = \eta_N = \epsilon_N/2$ and consider truncated functions\footnote{The expression $f^{r}(x)$ is
  short for $(f(x))^r$, while $f^{(r)}(x)$ denotes $r$th order derivative.}
\begin{equation*}
    s_{N}(l) =  \frac{1}{N} \sum_{j = 1}^{N}
    f_{\eta}^l(\hat{\gamma} - \mu_j), 
    \qquad
    f_{\eta}(x) = x^{-1} \boldsymbol{1}_{|x| > \eta} \, .
\end{equation*}
Since \textcolor{black}{by \cref{lemma three points} (ii)} $\Pr(\mu_1 > 2 + \epsilon_N/2) \to 0$ as $N \to \infty$, we
have a.a.s.
\begin{equation*}
  G^{(l)}(\gamhat)
  = (1+\sgn(b) b_N N^{-1/3})\delta_{l1} + c_l s_N(l).
\end{equation*}

Recall that if $X_N = c_N + o_{\Pr}(d_N)$ and $Y_N = X_N$ a.a.s., then
$Y_N = c_N + o_{\Pr}(d_N)$ also.
With these preparations, the proof of \cref{lemma derivatives of G} reduces to showing
that
\begin{equation} \label{eq:snlbd}
  s_N(l) =
  \begin{cases}
    1 - b_NN^{-1/3} + o_{\Pr}(N^{-1/3} b_N^{-1/2})  & l = 1 \\
    % c_l' \left( \frac{N^{1/3}}{b_N} \right)^{2l-3} (1+o_{\Pr}(1)) & l \geq 2.
    d_l (N^{1/3}/b_N)^{2l-3} (1+o_{\Pr}(1)) & l \geq 2,
  \end{cases}
\end{equation}
with $d_l =\frac{(2l-4)!}{(l-1)!(l-2)!2^{2l-3}}$ if $l\geq 2$. 

We consider first %$\alpha = 1$ with no spike, i.e. 
GUE, and
begin with the case $l \geq 2$, for which it will be enough to use
\begin{equation*}
  s_N(l) = \E s_N(l) + O_{\Pr}(\sqrt{\Var(s_N(l))}).
\end{equation*}
Let $\rho_N(\mu)$ denote the normalized one-point
correlation function. As in \eqref{eq:linear-stat}, 
we have

\begin{equation}
  \label{eq:ensl}
  \E s_N(l) = \int f_\eta^l(\gamhat-\mu) \rho_N(\mu) \diff \mu.
\end{equation}

% Our key tool will be a uniform bound due to G\"otze and Tikhomirov for
% the deviation of the one-point function in GUE (and GOE)
% from the semicircle density
% $p_{\rm SC}(x) = (2\pi)^{-1} \sqrt{4-x^2} \mathbf{1}_{|x| \leq 2}$.
% Indeed,
% %\cite{Gotze2005}
% \cite[Theorem 1.2]{Gotze2005}
% show the existence of
% absolute constants $\gamma, C > 0$ such that for all $|x| \leq 2 -
% \gamma N^{-2/3} $,
% \begin{equation}\label{goetze_tihkomirov1}
%     |\rho_{N}(x) - p_{SC}(x)| \leq \frac{C}{N(4 - x^2)} \, .
% \end{equation}
% In addition, it follows from \cite{Johnstone2012},
% see Section \ref{one-point-tail}, that for all $ s > -\gamma $, for large enough $N$,
% \begin{equation}\label{johnstone_ma1}
%     \rho_{N}(2 + sN^{-2/3}) \leq C(\gamma) N^{-1/3} e^{-2s}\, .
% \end{equation}
% % For the sake of completeness, we demonstrate how this inequality
% % follows from the results of Johnstone and Ma in Section FILL
% % IN. Furthermore, 
% A similar bound holds at the
% negative edge, by symmetry. Corresponding bounds also hold
% for $p_{SC}$. 

To bound the error in replacing $\rho_N$ by $p_{\rm SC}$ in integrals
such as
\eqref{eq:ensl}, set $\delta_N = \gamma N^{-2/3}$ with $\gamma>0$,
 decompose $\R$ into
$I_N = [-2+\delta_N,2-\delta_N]$ along with
$J_N = (2-\delta_N,\infty)$ and $J_N^- = (-\infty, -2+\delta_N)$
and write
\begin{equation*}
  \int g \rho_N - g p_{\rm SC}
  = \int_{I_N} g(\rho_N - \psc)
     + \int_{J_N \cup J_N^-} g \rho_N 
     - \int_{J_N \cup J_N^-} g \psc. 
\end{equation*}
Integrals over $J_N$ may be bounded using \eqref{johnstone_ma1}:
\begin{equation}
  \label{eq:boundA}
  \int_{J_N} |g(\mu)| \rho_N(\mu) \diff \mu
   \leq  C_\gamma N^{-1} \sup_{\mu > 2 - \delta_N} |g(\mu)|,
\end{equation}
since $\int_{2-\delta_N}^\infty \rho_N(\mu) \diff \mu
\leq N^{-1} C(\gamma) \int_{-\gamma}^\infty e^{-2s} \diff s
= N^{-1} C_\gamma.$
The integral over $J_N^-$ is dealt with similarly, and so are the
integrals with respect to $p_{\rm SC}$.

The integrals over $I_N$ require the G\"otze-Tikhomirov bound
\eqref{goetze_tihkomirov1}. For later use,
we bring in a bounded function $\psi$ and set
  $a_\psi = \sup_{-1 \leq y \leq 0} |\psi(y)|$ and 
  $b_\psi = \sup_{ y \leq -1} |\psi(y)|$. \textcolor{black}{Let
  \[
  \mathcal{I}_N
    = \int_{I_N} (\gamhat-\mu)^{-r}
    \psi\Big(\dfrac{\mu-2}{\epsilon_N}\Big)
  [\rho_N(\mu) - \psc(\mu)] \diff \mu.
  \]}
 Then, for $C = C(\gamma)$ and any $r>0$,
  we have
  \begin{equation}
    \label{eq:boundB}
    |\mathcal{I}_N|
     \leq \frac{C}{N \epsilon_N^r} \left( a_\psi \log
    \frac{\epsilon_N}{\delta_N} + b_\psi \right)
  \end{equation}
  \textcolor{black}{for all sufficiently large $N$.}

  Indeed, on the interval $[-2+\delta_N,0]$ \textcolor{black}{the absolute value of} the integrand is bounded by
  $2^{-r} b_\psi C N^{-1} 2^{-1} (2+\mu)^{-1}$ and so \textcolor{black}{the absolute value of} the corresponding
  integral is   at most $b_\psi C N^{-1} \log \delta_N^{-1}$, \textcolor{black}{for all sufficiently large $N$}.
  Writing $\mathcal{I}_N'$ for the integral over $[0,2-\delta_N]$,
  setting $x = 2- \mu$, and noting that $\delta = \delta_N <
  \epsilon_N = \epsilon$, we have
  \begin{equation*}
    |\mathcal{I}_N'|
    \leq \frac{C}{N} \int_\delta^2 \frac{1}{(x+\epsilon)^r x}
     \vert\psi \Big(\frac{-x}{\epsilon}\Big)\vert \diff x
    \leq \frac{a_\psi C}{N \epsilon^r} \int_\delta^\epsilon \frac{ \diff x}{x}
    + \frac{2 b_\psi C}{N} \int_\epsilon^2 \frac{\diff
      x}{(x+\epsilon)^{r+1}}, 
  \end{equation*}
which yields the claimed bound.

Returning to the approximation of \eqref{eq:ensl} and setting
$g(\mu) = f_\eta^l(\gamhat - \mu)$, we have
$g(\mu) = O(\eta^{-l}) = O(\epsilon_N^{-l})$ in \eqref{eq:boundA},
and $\psi := 1, r = l$ in \eqref{eq:boundB}, so that with
$\epsilon_N/\delta_N = b_N^2/\gamma$,
\begin{equation} \label{eq:6.24}
  \E s_N(l)
  = \int (\gamhat - \mu)^{-l} \psc(\mu) \diff \mu
    +O\left(\frac{\log b_N}{N \epsilon_N^l}\right).
\end{equation}
To approximate the latter integral we use derivatives of 
$m_{\rm SC}(z) = \frac{1}{2} (-z + \sqrt{z^2 - 4})$, the Stieltjes
transform of the semi-circle distribution \textcolor{black}{(extended to real $z>2$)}.
Recalling that $\gamhat = 2 + \epsilon_N$, we write
\begin{align*}
  m_{\rm SC}(2+\epsilon)
  & = -1 + \epsilon^{1/2}(1+\epsilon/4)^{1/2} - \epsilon/2, \\
  m_{\rm SC}^{(l)}(2+\epsilon)
    & = m_l \epsilon^{1/2 - l}(1+O(\epsilon)) - \tfrac{1}{2} \delta_{l1}, \quad
     \quad  l \geq 1,
\end{align*}
where $m_l =\frac{(-1)^{l-1}}{2^{2l-1}}\frac{(2l-2)!}{(l-1)!}$.
We then have
\begin{equation*}
  \int (\gamhat - \mu)^{-l} \psc(\mu) \diff \mu
   = c_l^{-1} m_{SC}^{(l-1)}(\gamhat)
   = \begin{cases}
     1 - \epsilon_N^{1/2}+ O(\epsilon_N)  & l = 1 \\
     d_l \epsilon_N^{3/2 - l} -\frac{1}{2} \delta_{l2} +
     O(\epsilon_N^{5/2 - l}) \quad & l  \geq 2,
   \end{cases}
\end{equation*}
where $d_l = m_{l-1}/c_l =\frac{(2l-4)!}{(l-1)!(l-2)!2^{2l-3}}$ for $l\geq 2$ as claimed above.
The error term $O(\frac{\log b_N}{N \epsilon_N^l})$ in \eqref{eq:6.24}
dominates the term $O(\epsilon_N)$ for $l = 1$ and all but the leading
term for $l \geq 2$. Consequently
\begin{equation}  \label{eq:esnlapp}
  \E s_N(l)   =
   \begin{cases}
     1 - b_N N^{-1/3} + O(N^{-1/3} (\log b_N)/b_N^2)  & l = 1 \\
     d_l (N/b_N^3)^{2l/3-1} +
         O\left( (N/b_N^3)^{2l/3-1} (\log b_N)/b_N^3 \right)  \quad & l \geq 2.
   \end{cases}
\end{equation}

To bound the variances of $s_{N}(l)$ we use the following inequality
{\color{black}(see \cite{johnstone2020logarithmic}, \textcolor{black}{lemma 16})}%\label{lem:variance-bound}
\begin{equation}\label{Iains inequality alt1}
    \Var\bigg( \frac{1}{N} \sum_{j = 1}^{N} f(\mu_j) \bigg) \leq \frac{1}{N} \int f^2(\mu) \rho_{N}(\mu) \diff\mu,
\end{equation}
which holds for GUE for arbitrary  \(f\) and \(N\), as long as the
integrals involved exist. In particular,
\begin{equation*}
\Var\left[ s_{N}(l)\right] \leq \frac{1}{N}\int  f_{\eta
  }^{2l} \left( 
\hat{\gamma}-\mu \right) \rho _{N}\left( \mu \right)
\diff\mu
  = O(N^{\frac{4l-6}{3}} b_N^{-4l+3})
\end{equation*}%
using \eqref{eq:esnlapp}.
This establishes \eqref{eq:snlbd} for $l \geq 2$.

For $l = 1$, however, this bound yields only error control at $O_{\Pr}(N^{-1/3}
b_N^{-1/2})$ in \eqref{eq:snlbd}.
To improve on the bound near the edge, we use the mesoscopic CLT of
Basor and Widom \cite{basor1999determinants}.
Let $c > 1$ and consider a $C^\infty$ partition of unity $\psi_1 +
\psi_2 = 1$ with $0 \leq \psi_i \leq 1$ and 
\begin{equation*}
\psi_{1}(x)=
\begin{cases}
  1 \quad & -c \leq x \leq 1/4 \\
  0 \quad & x \leq -2c \text{ and } x \geq 1/2.
\end{cases}
\end{equation*}%
We then have the decomposition $ s_{N}(1) = \bar{s}_1(c) +
\bar{s}_2(c)$, where 
\begin{equation*}
  \bar{s}_{i} =   \bar{s}_{i}(c) 
  =\frac{1}{N}\sum_{j=1}^{N}f_{\eta }\left( \hat{\gamma}
    -\mu _{j}\right) \psi _{i}( \Delta_N(\mu_j)),
  \qquad \Delta_N(\mu) = \frac{N^{2/3}(\mu -2)}{b_N^2}.
\end{equation*}

To show that $s_N(1) - \E s_N(1) = o_{\Pr}(N^{-1/3} b_N^{-1/2})$, we
use the following convergence criterion: if for each $c$ large, $X_N = Y_{N1}(c) + Y_{N2}(c)$ with $Y_{N1}(c) = o_{\Pr }(1)$ and $\Var Y_{N2}(c) \leq \textcolor{black}{c^{-1/2}}$ for $N >N(c)$, then $X_N = o_{\Pr }(1)$.

First, our previous tools suffice to bound fluctuations of the
$\bar{s}_2$ term. Indeed,
from \eqref{Iains inequality alt1} we have
\begin{equation*}
  \Var( \bar{s}_2(c))
   \leq \frac{1}{N} \int f_\eta^2(\gamhat-\mu)
   \psi_2^2(\Delta_N(\mu)) \rho_N(\mu) \diff \mu.
\end{equation*}
We approximate the integral by $\mathcal{J}_N = \int g(\mu)
\psc(\mu) \diff \mu$, using bounds \eqref{eq:boundA} with
$g(\mu) = f_\eta^2(\gamhat-\mu) \psi_2^2(\Delta_N(\mu)) =
O(\eta^{-2})$ and \eqref{eq:boundB} with $a_\psi = 0, b_\psi = 1$, 
so that the error is bounded in order by
\begin{equation*}
  N^{-1} \left[ \eta_N^{-2} + \epsilon_N^{-2} \right]
   \asymp N^{-1} \epsilon_N^{-2} \asymp N^{1/3} b_N^{-4}.
\end{equation*}
On the interval $[-2,2]$, we have $g(\mu) = (\gamhat-\mu)^{-2}
\psi_2^2(\Delta_N(\mu))$, which vanishes if $\Delta_N(\mu)
\geq -c$, so that
\begin{equation} \label{eq:c1bd}
  \mathcal{J}_N
  \leq \frac{1}{2\pi} \int_{-2}^{2-c \epsilon_N} (\gamhat-\mu)^{-2}
  \sqrt{4-\mu^2} \diff \mu
  \leq \frac{1}{\pi} \int_{c \epsilon_N}^4 x^{-3/2} \diff x
   \leq \frac{2}{\pi} (c \epsilon_N)^{-1/2}.
\end{equation}
Consequently for $N \geq N(c)$,
\begin{equation} \label{eq:vars2bar}
  \Var \bar{s}_2(c)
  \leq N^{-1} \left[ 2\pi^{-1} c^{-1/2} N^{1/3}b_N^{-1}
    + O(N^{1/3} b_N^{-4}) \right]
  \leq 4\pi^{-1} c^{-1/2} N^{-2/3} b_N^{-1}.
\end{equation}

Turning now to $\bar{s}_1(c)$, since $\psi_1(\Delta_N(\mu))=0$ for
$\mu > 2 + \epsilon_N/2 = \gamhat - \eta_N$, we have
\begin{equation*}
  \bar{s}_1 = N^{-1} \sum_{j=1}^N (\gamhat - \mu_j)^{-1}
  \psi_1(\Delta_N(\mu_j)). 
\end{equation*}
Let us rewrite
\begin{equation*}
  (\gamhat-\mu)^{-1} \psi_1(\Delta_N(\mu))
   = \epsilon_N^{-1} \varphi(\Delta_N(\mu))
\end{equation*}
where $\varphi \left( x\right) ={\psi _{1}(x)}/(1-x)$ is a Schwartz
function because $\supp \psi_1 = [-2c, \frac{1}{2}]$.
By \cite{basor1999determinants},
\begin{equation*}
  T_N = \sum_{j=1}^N \varphi(\Delta_N(\mu_j)) = \alpha_N + o_{\Pr}(1),
\end{equation*}
with
\begin{equation*}
  \alpha_N = \frac{b_N^3}{\pi} \int_0^\infty \sqrt{x} \varphi(-x) \diff x
        = \frac{N^{1/3} b_N^2}{\pi} \int_{-\infty}^2 (2-\mu)^{1/2}
        \epsilon_N^{-1} \varphi(\Delta_N(\mu)) \diff \mu.
\end{equation*}
Hence
\begin{equation*}
  \bar{s}_1 = \frac{T_N}{N^{1/3} b_N^2}
   = \int_{-\infty}^2 
   (\gamhat-\mu)^{-1} \psi_1(\Delta_N(\mu))
    \frac{\sqrt{2-\mu}}{\pi} \diff \mu
          +  o_{\Pr}\left(\frac{1}{N^{1/3} b_N^2} \right).
\end{equation*}
We also have
\begin{equation*}
  \E \bar{s}_1
   = \int_{-2}^2 (\gamhat-\mu)^{-1}
  \psi_1(\Delta_N(\mu)) \psc(\mu) \diff \mu
  + O(N^{-1/3} b_N^{-2} \log b_N),
\end{equation*}
where again we approximate the first integral by 
$\mathcal{J}_{N2} = \int g_2(\mu)
\psc(\mu) \diff \mu$, using bound \eqref{eq:boundA} with
$g_2(\mu) = (\gamhat-\mu)^{-1} \psi_1(\Delta_N(\mu)) =
O(\epsilon_N^{-1})$ and bound \eqref{eq:boundB} with $a_\psi = b_\psi
= 1$, so that the error is bounded in order by 
$ N^{-1} \epsilon_N^{-1} + N^{-1} \epsilon_N^{-1}
  \log(\epsilon_N/\delta_N)
  = O(N^{-1/3} b_N^{-2} \log b_N).$

Since, \textcolor{black}{$\psi_1(\Delta_N(\mu))=0$  for $\mu\leq 2-2c\epsilon_N$} %$\psi_1(\Delta_N(\lambda)) \leq \mathbf{1}_{2-2c\epsilon_N \leq
%  \lambda \leq 2}$ 
and $\gamhat > 2$, the leading term of
$\bar{s}_1 - \E \bar{s}_1$ is bounded by
\begin{equation*}
  \begin{split}
  \int_{2-2c \epsilon_N}^2  \bigg( \frac{\sqrt{2-\mu}}{\pi} 
    & - \frac{\sqrt{4-\mu^2}}{2\pi} \bigg)
    \frac{\diff  \mu}{2-\mu}  \\
    & = \frac{1}{\pi} \int_0^{2c \epsilon_N} \left[1 - \left(1 -
      \frac{x}{4}\right)^{1/2} \right] \frac{\diff x}{\sqrt{x}}
    = O((c \epsilon_N)^{3/2})
    = O(c^{3/2} b_N^3/N).
  \end{split}
\end{equation*}

Therefore, combining the error terms, we obtain for all $c > 1$
\begin{equation*}
  \bar{s}_1(c) - \E \bar{s}_1(c)
    = O_{\Pr}(N^{-1/3} b_N^{-2} \log b_N).
\end{equation*}
Together with \eqref{eq:vars2bar} and \textcolor{black}{the convergence criterion described above},
% Lemma \ref{lem:suffcond},
we conclude that
$s_N(1) - \E s_N(1) = o_{\Pr}(N^{-1/3} b_N^{-1/2})$.

For the GOE case %(still with no spike, $J = 0$), 
we use
G(O/U)E comparison for linear statistics, cf
\eqref{eq:GUE-comp}--~\eqref{eq:GOE-comp}. 
% the following theorem, established in
% \textcolor{black}{\cite{johnstone2020logarithmic}.} 
% \begin{theorem}
% \label{theorem Damian}
% Let $Z_\alpha$ be an $N\times N$ (unscaled) GUE or GOE for $\alpha=1$ or $\alpha=2$, respectively. Suppose that $f_N$ is a series of functions such that
% \[
% f_N\left(Z_1\right)=a_N+O_{\Pr}(b_N),
% \]
% for some sequences $a_N$ and $b_N$. Then,
% \[
% f_N\left(Z_2\right)
%   = a_N+O_{\Pr}(b_N+\mathrm{TV}(f_N)),
% \]
% where $\mathrm{TV}(f_N)$ is the total variation of $f_N$.
% \end{theorem}
We apply this to $f_N(x) = N^{-1} f_\eta^l(\gamhat - x)$, for which
$\mathrm{TV}(f_N) = 4 N^{-1} \eta_N^{-l} = 2^{l+2} N^{2l/3 -1}/b_N^{2l}$.
Since this is of smaller order than the error terms in \eqref{eq:snlbd},
\eqref{eq:snlbd} remains valid in the GOE case. 

\textcolor{black}{Validity of \eqref{eq:DG-bound1} for the critically spiked G(O/U)E case follows via interlacing. Indeed, the interlacing inequality for rank one perturbations implies that $\mu_j\leq \lambda_j^{(\omega)}\leq \mu_{j-1}$ for $j=1,\dotsc,N$ (and $\mu_0=\infty$). Consequently
\begin{equation*}
  0 \leq \sum_{j=1}^N (\gamhat - \lambda_j^{(\omega)})^{-l}
  - \sum_{j=1}^N (\gamhat - \mu_j)^{-l}
  \leq (\gamhat - \lambda_1^{(\omega)})^{-l}
  = O_{\Pr }(N^{2l/3}/\log^l N),    
\end{equation*}
since part (i) of lemma~\ref{lemma three points} implies that $N^{2/3}(\gamhat - \lambda_1^{(\omega)}) =
\Theta_{\Pr}(\log N)$.
Thus, replacing $\mu_j$ by the critically spiked versions $\lambda_j^{(\omega)}$
in $G^{(l)}(\hat{\gamma})$ does not affect the error terms in \eqref{eq:DG-bound1}. \qed}

\subsection{Proof of \cref{lemma lambda1statistic} for G(O/U)E}\label{sec:lemma2.9 Gaussian}
The validity of \eqref{eq:p4-bd} for G(O/U)E is established in
proposition 3 of \cite{johnstone2020logarithmic}.

%We will prove \eqref{inv mom 1} later, directly for the general case of sub-critically spiked Wigner $W_N$ without making the G(O/U)E step.
\textcolor{black}{Let us now establish \eqref{eq:p4-bd}
%  and \eqref{inv mom 1}
  for the critically spiked G(O/U)E. Consider the identity
\begin{equation}
    \label{identity2}
    (W_{J,N}^{(\omega)}-\gamma)^{-1}=(W_N-\gamma)^{-1}-J\frac{(W_N-\gamma)^{-1}ww^\ast(W_N-\gamma)^{-1}}{1+Jw^\ast (W_N-\gamma)^{-1}w}.
\end{equation}
Taking trace and dividing by $-N$, we arrive at
\begin{equation}
\label{after taking trace}
\frac{1}{N}\sum_{i=1}^N\frac{1}{\gamma-\lambda_i^{(\omega)}}=\frac{1}{N}\sum_{i=1}^N\frac{1}{\gamma-\mu_i}+\frac{J}{N}\frac{w^\ast (W_N-\gamma)^{-2}w}{1+Jw^\ast (W_N-\gamma)^{-1} w}.
\end{equation}
From the discussion below (\ref{eq:NVN}),
%\eqref{eq:-1 quadratic},
\begin{equation}
\label{adjustement factor}
1+Jw^\ast (W_N-\gamma)^{-1}w=\Theta_{\Pr}(N^{-1/3}).
\end{equation}
Further, similarly to \eqref{quadratic},
\begin{equation}
\label{quadratic1}
w^\ast (W_N-\gamma)^{-2}w\overset{\mathrm{d}}{=}\frac{\frac{1}{N}\sum_{i=1}^N|\xi_i|^2/(\mu_i-\gamma)^2}{\xi^\ast \xi/N}.
\end{equation}
On the other hand, by \eqref{eq:p4-bd} for the non-spiked G(O/U)E,
\[
\mathbf{E}\left\{\left.\frac{1}{N}\sum_{i=1}^N\frac{|\xi_i|^2}{(\mu_i-\gamma)^2}\right|\mu_1,...,\mu_N\right\}=\frac{1}{N}\sum_{i=1}^N\frac{1}{(\mu_i-\gamma)^2}=O_{\Pr}(N^{1/3}).
\]
Therefore, also,
\begin{equation}
    \label{numerator1}
 \frac{1}{N}\sum_{i=1}^N\frac{|\xi_i|^2}{(\mu_i-\gamma)^2}=O_{\Pr}(N^{1/3}).   
\end{equation}
Combining with \eqref{quadratic1} and \eqref{denominator}, we obtain
\begin{equation}
\label{numerator2}
w^\ast (W_N-\gamma)^{-2}w=O_{\Pr}(N^{1/3}).
\end{equation}
Using this and \eqref{adjustement factor} in \eqref{after taking trace} yields
\[
\frac{1}{N}\sum_{i=1}^N\frac{1}{\gamma-\lambda_i^{(\omega)}}=\frac{1}{N}\sum_{i=1}^N\frac{1}{\gamma-\mu_i}+O_{\Pr}(N^{-1/3})=1+O_{\Pr}(N^{-1/3}),
\]
where the last equality follows from \eqref{eq:p4-bd} for the non-spiked G(O/U)E. This establishes the first equality in \eqref{eq:p4-bd} for the critically spiked G(O/U)E case.}

\textcolor{black}{To see that the second equality in \eqref{eq:p4-bd} also holds, square both sides of \eqref{identity2}, take traces and divide by $N$ to obtain
\begin{eqnarray}
\label{long}
\frac{1}{N}\sum_{i=1}^N\frac{1}{(\gamma-\lambda_i^{(\omega)})^2}&=&\frac{1}{N}\sum_{i=1}^N\frac{1}{(\gamma-\mu_i)^2}-\frac{2J}{N}\frac{w^\ast (W_N-\gamma)^{-3}w}{1+Jw^\ast (W_N-\gamma)^{-1}w}\\ \notag
&&+\frac{J^2}{N}\frac{(w^\ast (W_N-\gamma)^{-2}w)^2}{(1+Jw^\ast (W_N-\gamma)^{-1}w)^2}.
\end{eqnarray}
By \eqref{adjustement factor} and \eqref{numerator2}, the last term on the right hand side is $O_{\Pr}(N^{1/3})$. For the numerator of the second term, we have
\[
|w^\ast (W_N-\gamma)^{-3} w|\leq \|(W_N-\gamma)^{-1}\||w^\ast (W_N-\gamma)^{-2}w|=O_{\Pr}(N),
\]
where we have used \eqref{numerator2} and the equality $\|(W_N-\gamma)^{-1}\|=O_{\Pr}(N^{2/3})$, which follows from the second equation in (24) in Lemma 4 of \cite{johnstone2020logarithmic}. The latter display and \eqref{adjustement factor} imply that the second term on the right hand side of \eqref{long} is $O_{\Pr}(N^{1/3})$. Hence, overall,
\[
\frac{1}{N}\sum_{i=1}^N\frac{1}{(\gamma-\lambda_i^{(\omega)})^2}=\frac{1}{N}\sum_{i=1}^N\frac{1}{(\gamma-\mu_i)^2}+O_{\Pr}(N^{1/3})=O_{\Pr}(N^{1/3}),
\]
where the latter equality follows from \eqref{eq:p4-bd} for the non-spiked G(O/U)E case. This finishes our proof of the second equality in \eqref{eq:p4-bd} for the critically spiked G(O/U)E.}

\bigskip
We turn to proving \eqref{inv mom 1}. Although written for G(O/U)E,
the argument equally works for case G$^\omega$.
\textcolor{black}{ %To see that \eqref{inv mom 1} holds, note
\fix{First, recall} 
  that by \eqref{eq:p4-bd}, for fixed $C\in\mathbb{R}$,
\begin{equation}
  \label{eq:prop4a}
  \frac{1}{N}\sum_{j=1}^{N}(2-\mu_j)^{-1}=1+O_{\Pr}(N^{-1/3})
  \qquad \text{and} \qquad
  \frac{1}{N}\sum_{j=1}^{N}(2 - CN^{-2/3} -
  \mu_j)^{-2}=O_{\Pr}(N^{1/3}).  
\end{equation}
By the Cauchy-Schwarz inequality,
\begin{align*}
    \left\vert\frac{1}{N}\sum\nolimits_{j=2}^{N}\frac{1}{\mu_1-\mu_j}-\frac{1}{N}\sum\nolimits_{j=2}^{N}\frac{1}{2-\mu_j}\right\vert & =\left\vert\frac{1}{N}\sum\nolimits_{j=2}^{N}\frac{2-\mu_1}{(2-\mu_j)(\mu_1-\mu_j)}\right\vert\\
    &\leq |2-\mu_1|\left(\frac{1}{N}\sum\nolimits_{j=2}^{N}\frac{1}{(2-\mu_j)^2}\right)^{1/2}\left(\frac{1}{N}\sum\nolimits_{j=2}^{N}\frac{1}{(\mu_1-\mu_j)^2}\right)^{1/2}.
\end{align*}
The bounds in \cref{eq:prop4a} with $C = 0$, along with the Tracy-Widom
convergence $|\mu_1-2|=O_{\Pr}(N^{-2/3})$ (\cref{lemma three points} part (i)) show that to establish
\eqref{inv mom 1}, it is
sufficient to show that
$\frac{1}{N}\sum\nolimits_{j=2}^{N}\frac{1}{(\mu_1-\mu_j)^2}=O_{\Pr}(N^{1/3})$. For this, note that by \cref{lemma three points} (ii), for each $\epsilon > 0$ there exists 
%the Tracy-Widom convergence %,  \cref{lemma three points} (i), yields 
a constant $C$ such that event
$\mathcal{E}_{N,\epsilon} = \{ \mu_1 > 2 - CN^{-2/3} \}$ has probability at least $1 - \epsilon$
for large $N$. On this event,
\begin{equation*}
  \mu_1 - \mu_j \geq
  \begin{cases}
    2 - CN^{-2/3} - \mu_j  & \text{if } \mu_j \leq  2 - C N^{-2/3}
    \\
    \mu_1 - \mu_2      & \text{if } \mu_j >  2 - C N^{-2/3}
  \end{cases}
\end{equation*}
We have $\chi_N(C) := \# \{ j ~:~ \mu_j > 2 - CN^{-2/3}
\} = O_{\Pr}(1)$ and $\mu_1 - \mu_2 = \Theta_{\Pr}(N^{-2/3})$
by Lemma \ref{lemma three points} parts (iii) and (iv) respectively. Using also
\eqref{eq:prop4a} we obtain, on $\mathcal{E}_{N,\epsilon}$,
\begin{equation*}
  \frac{1}{N}\sum\nolimits_{j=2}^{N}\frac{1}{(\mu_1-\mu_j)^2}
  \leq
  \frac{1}{N}\sum\nolimits_{j=1}^{N}\frac{1}{(2-CN^{-2/3}-\mu_j)^2}
  % \frac{1}{N}\sum\nolimits_{j=2}^{N}\frac{\mathbf{1}[\lambda_j\leq2-C
  %   N^{-2/3}]}{(\lambda_1-\lambda_j)^2}
  +\frac{\chi_N(C)}{N(\mu_1-\mu_2)^2}
  = O_{\Pr}(N^{1/3}).
\end{equation*}
This completes the proof of \cref{lemma lambda1statistic} for unspiked G(O/U)E cases. % for $J=0$. we use convergence criterion \textbf{C.2}, introduced after \eqref{Iains inequality alt1}.
}

\textcolor{black}{\fix{In case G$^\omega$} the equalities in \eqref{inv mom 1} follow from \eqref{eq:p4-bd} and Lemma \ref{lemma three points} in exactly the same way as in the unspiked G(O/U)E case. \qed}

\subsection{Proof of proposition \ref{independence} for spiked G(O/U)E}
\label{sec:proof-prop-ref}

\textcolor{black}{Now we turn to the critically spiked G(O/U)E with $J=1-\omega N^{-1/3}$.}

\textcolor{black}{As explained e.g.~in \cite{johnstone2020logarithmic}, the analogue of \eqref{tridiag_M_N} for spiked G(O/U)E is
\begin{equation}\label{tridiag_M_Nspike}
    \sqrt{N}M_{J,N}^{(\omega)} = 
    \begin{pmatrix}
        a_1 & b_1 & & \\
        b_1 & a_2 & b_2 &   \\
         & b_{2} & \ddots & \ddots \\
         && \ddots & \ddots & b_{N-1} \\
         &&& b_{N-1} & a_{N}+\sqrt{N}J
    \end{pmatrix}.
\end{equation}
Hence, we will interpret $\lambda_j^{(\omega)}$ as the eigenvalues of $M_{J,N}^{(\omega)}$.
From \eqref{identity1} and \eqref{adjustement factor}, 
\[
\sum_{j=1}^N\log|2-\lambda_j^{(\omega)}|=\sum_{j=1}^N\log|2-\mu_j|-\frac{1}{3}\log N +O_{\Pr}(1).
\]
Therefore,
\[
\xi_{1N}^{(\omega)}=\xi_{1N}+o_{\Pr}(1)=X_N/\sqrt{\frac{\alpha}{3}\log N}+o_{\Pr}(1),
\]
where $X_N$ depends only on $a_i$ and $b_i$ with $i<N-2N^{1/3}\log^3 N$, as explained above.}

\textcolor{black}{Further, let $Y_N^{(\omega)}$ be the largest eigenvalue of the bottom-right minor of $M_{J,N}^{(\omega)}$ of size $l= 2N^{1/3}\log^3 N$. Following the proof of \cref{tilde_lambda}, we find that $|\lambda_1^{(\omega)}-Y_N^{(\omega)}|=O_{\Pr}(N^{-K})$ for any $K>0$. Therefore,
\[
(\xi_{1N}^{(\omega)},\xi_{2N}^{(\omega)})=(X_N/\sqrt{\frac{\alpha}{3}\log N},N^{2/3}(Y_N^{(\omega)}-2))+o_{\Pr}(1),
\]
where $X_N$ and $Y_N^{(\omega)}$ are independent random variables such that
\[
X_N/\sqrt{\frac{\alpha}{3}\log N}\drightarrow\mathcal{N}(0,1),\quad N^{2/3}(Y_N^{(\omega)}-2)\drightarrow \mathrm{BV}_{2/\alpha}(\omega).
\]\qed
}

\section{\textcolor{black}{Extension to Wigner cases}}\label{sec:wigner}
In this section we extend the results that were proven in section \ref{sec:proof key lemmas} for the G(O/U)E cases with $J=0$ \textcolor{black}{to sub-critically spiked Wigner cases.}

\subsection{Proof strategy and preliminary results}
Proving that a Wigner matrix $W'_N$ satisfies a certain property as
long as a matrix $W_N$ from \textcolor{black}{scaled} G(O/U)E satisfies this property is often
based on the Lindeberg swapping process, where elements of $W_N$ are
replaced by the elements of $W'_N$ one by one without losing the
property. 
Typically, one needs to show that any individual swap does not change
the expectation $\mathbf{E}Q(M)$ of some smooth function $Q(\cdot)$ of
the matrix $M$ participating in the swapping process too much. 

\textcolor{black}{In more detail, let $\gamma$ index an ordering of the independent components $\{\Re \xi_{ij}, \Im \xi_{ij}\}_{i<j}$ and $\{ \xi_{ii} \}$ of $\sqrt{N}W_N$.
Thus $\gamma$ runs over $N^2$ and $N(N+1)/2$ elements in the Hermitian and symmetric cases respectively.
Let $W^\gamma$ refer to a matrix in which the elements prior to $\gamma$ come from $W_N'$ while those at $\gamma$ or later come from $W_N$.}

\textcolor{black}{At stage $\gamma$ in the swapping process, we can write
$W^{(0)} = W^{\gamma}$, $W^{(1)} = W^{\gamma +1}$, and
\begin{equation} \label{eq:swapdefs}
  W^{(0)} = W_0^\gamma + \frac{\xi^{(0)}}{\sqrt N} V_\gamma, \qquad
  W^{(1)} = W_0^\gamma + \frac{\xi^{(1)}}{\sqrt N} V_\gamma,
\end{equation}
and $ W^\gamma_{0}$ is independent of both $\xi^{(0)}$ and $\xi^{(1)}$.
In the symmetric case, $V_\gamma$ is one of the elementary matrix
$e_a e_a^*$ or $ e_a e_b^* + e_b e_a^*$. In the Hermitian case, we add matrices $ \im e_a e_b^* -  \im e_b e_a^*$. Here $e_j$ denotes the $j$-th column of the identity matrix $I_N$, and $a,b$ correspond to the row and column of the components of $W_N$ and $W_N'$ being swapped at stage $\gamma$. These components are denoted as $\xi^{(0)}/\sqrt{N}$ and $\xi^{(1)}/\sqrt{N}$ respectively.
All matrices $W^\gamma$
% $W_{0}^\gamma$
are Wigner matrices.}

\textcolor{black}{We consider $W^{(0)}$ and $W^{(1)}$ as perturbations of $W_0^\gamma$.
Thus, set $W^\gamma_t = W_0^\gamma + t N^{-1/2} V_\gamma$, and 
introduce $Q_\gamma(t) = Q(W_t^\gamma)$ for some smooth function $Q(\cdot)$. We first summarize the properties of such a $Q(\cdot)$ that are sufficient to carry out the swapping process.}

\begin{definition}
Fix $c_0>0$ and set $ \| F \|_{c_0} = \sup \{ |F(t)|, |t| \leq   N^{c_0} \}$.
Let  \(\delta_N \rightarrow 0\) in such a way that \(\delta_N \gtrsim
N^{-c_1}\) for some \(c_1 > 0\).
Let $Q$ be a function on $N \times N$ Hermitian/symmetric matrices
taking values in $[0,1]$.
Let Wigner matrices $W_N, W_N'$ be given and define
$Q_\gamma(t) = Q(W_t^\gamma)$ as above.
We say that $Q$ satisfies \textbf{condition F} or $F(\delta_N)$ if for
all $\gamma$ and $1 \leq k \leq 4$ we have w.o.p. that 
  \begin{align}
    \label{eq:D-cond-1}
    \norm{Q_{\gamma}^{(k)}}_{c_0}
    \lesssim N^{-\frac{k}{2}} \delta_{N}.  \tag{F}
  \end{align}
\end{definition}

%\subsection{Results from \texorpdfstring{\cite{johnstone2020logarithmic}}{[JKOP20]}}

Next, we present proposition 20 of \cite{johnstone2020logarithmic}, which details how condition F is used to control differences in the distributions of Wigner matrices:%this is the proposition with \label{prop:multi-matching}
\begin{proposition}
  \label{prop:multi-matching}
  Let \(W_N, W_N'\) be Wigner matrices
  % satisfying \textcolor{black}{\textbf{W3}} whose off-diagonal entries
  % have real and imaginary parts
  whose moments match to third order.
  Let \(c_0, c_1> 0\) be fixed and for
each \(j = 1, \dotsc, m\), let \(Q_j \maps \C^{N \times N} \rightarrow
[0, 1]\) satisfy condition $F(\delta_{j,N})$.  If \(Q = \prod_{j=1}^m
Q_j\), then, 
  \begin{align}
    \label{eq:Q-swap}
    \E Q(W_N) - \E Q(W_N')
    \lesssim \max_{j=1, \ldots m} \delta_{j,N}.
%    \lesssim \sum_{j=1}^m \delta_{j,N}.
  \end{align}
\end{proposition}

\textcolor{black}{Let 
\[
s_{W_N}(z):=\frac{1}{N}\tr\left(W_N-z\right)^{-1}
\]
be the Stieltjes transform of the empirical spectral distribution of $W_N$.} The following result is proposition 24 of \cite{johnstone2020logarithmic}, %labeled there as \label{prop: system}
which is the main tool that we will use to show that a function satisfies condition \textbf{F}.
%Part (1) of the result as stated in \cite{johnstone2020logarithmic} is not relevant to this paper, but we keep it here for consistency. 
\begin{proposition}
  \label{prop:system}
  Let \(W_N\) be a Wigner matrix.
  % satisfying \textbf{W1} and \textbf{W3}. 
  \textcolor{black}{Fix $ \epsilon > 0 $ small and $ 0 < c_0 < 1/2$. Let \(E, E_1, E_2 \in \R\) be such that \(\abs{E - 2} \lesssim N^{-\frac{2}{3} + 2\epsilon}\) and \(\abs{E_i - 2} \lesssim N^{-\frac{2}{3} + 10\epsilon}\), $i=1,2$.}

  For each of the following statistics, define functions \(g \maps \C^{N\times N} \rightarrow \R\), \(G\maps \R \rightarrow \R\) and a sequence \(\delta_N\) according to the following specifications, \textcolor{black}{in each case for $1\leq j\leq 4$}:
  \begin{enumerate}
   \item Log-determinant: with $\gamma_N = N^{-2/3-\epsilon}$,
     \begin{equation*}
       g(W_N) = N \int_{\gamma_N}^{N^{100}} \Im s_{W_N}(E+ \im \eta) \diff \eta,
       \qquad \|G^{(j)}\|_\infty \leq (\log N)^{-j/4},
       \qquad \delta_N = (\log N)^{-1/4}.
     \end{equation*}
  \item Eigenvalue counting: with $\eta = N^{-2/3-9\epsilon}$,
    \begin{equation*}
      g(W_N) = \frac{N}{\textcolor{black}{\pi}} \int_{E_1}^{E_2} \Im s_{W_N}(y + \im \eta) \diff y,
      \qquad \|G^{(j)}\|_\infty \leq (\log N)^{Cj},
      \qquad \delta_N = N^{-\frac{1}{3} + O(\epsilon)}.
    \end{equation*}
  \item Inverse moments:
    with \(\eta = N^{-2/3-\epsilon}\) and $l\in\mathbb{Z}_+$,
    \begin{equation*}
      g(W_N)
      = N^{-\frac{2}{3}l + 1} \Re s_{W_N}^{(l-1)}(E + \im \eta),
      \qquad \|G^{(j)}\|_\infty \leq (\log N)^{Cj}, \\
      \qquad \delta_N = N^{-\frac{1}{3} + O(\epsilon)}.
    \end{equation*}
  \end{enumerate}

  In each of the cases listed above, the corresponding function \(Q = G \circ g\) satisfies condition $F(\delta_N)$.
\end{proposition}

Our main tool for establishing distributional results about joint convergence is proposition 21 of \cite{johnstone2020logarithmic}:
%labeled there as \label{prop:joint-convergence}

\begin{proposition}
  \label{prop:joint-convergence}
     Let \(W_N, W_N'\) be Wigner matrices
   whose off-diagonal moments match up to third order.
  Let $\bxi_N = \bxi_N(W_N)$ and $\bxi_N' = \bxi_N(W_N')$ both be
  $\mR^m$ valued random vectors.
  Suppose that $\bm{\xi}_N \stackrel{\rm d}{\to} \bm{\xi}$, and that
  each component $\xi_j$ of the limit has a continuous distribution
  function.
  
  Let $\eta_N \to 0$ be given, and suppose that for each $1 \leq j
  \leq m$ and $s \in \mR$ 
  %there exist functions $Q_j^\pm(\cdot,s)$   satisfying condition $F(\delta_{j,N})$ such that for $W = W_N, W_N'$, w.o.p.
  %\begin{alignat}{3}
  %\label{eq:Qjp}
  %  \bo \{ \xi_{Nj}(W) \leq s\}
  %& \leq Q_j^+(W,s)  &
  %& \leq \bo \{ \xi_{Nj}(W) \leq s + \eta_N\}  \\
 %   \bo \{ \xi_{Nj}(W) \leq s - \eta_N\}
 % & \leq Q_j^-(W,s)  &
 % & \leq \bo \{ \xi_{Nj}(W) \leq s \}
 %    \label{eq:Qjm}
 %  \end{alignat}
 there exists a function $Q_j(\cdot,s)$   satisfying condition $F(\delta_{j,N})$ such that for $W = W_N, W_N'$, w.o.p.
 \begin{equation}
     \label{eq:Qj}
      \bo \{ \xi_{Nj}(W) \leq s-\eta_N\}
   \leq Q_j(W,s)  
  \leq \bo \{ \xi_{Nj}(W) \leq s + \eta_N\}.
 \end{equation}
   
   Then we also have (joint) convergence
   $\bm{\xi}_N' \stackrel{\rm d}{\to} \bm{\xi}$.
\end{proposition}

Throughout the remaining of this section we denote by $ W_N $ a matrix from scaled GUE ($\alpha = 1$) or scaled GOE ($\alpha = 2$) and by $W_N'$ the corresponding real ($\alpha = 1$) or complex ($\alpha = 2$) Wigner matrix, \textcolor{black}{whose off-diagonal moments match scaled G(O/U)E up to third order.}
In order to complete the proofs we also need to add a sub-critical spike. For fixed $J \in [0, 1)$, consider the matrix
\begin{equation}\label{spiked_definition}
    W_{J,N}' = W_N' + J \bv\bv^*,
\end{equation}
where $\bv$ is arbitrary unit vector from $\mathbb{C}^N$ (from $\mathbb{R}^N$ if $\alpha = 2$).

\textcolor{black}{Consistent with notation used in \cref{sec:proof key lemmas}}, we denote by \textcolor{black}{$  \mu_{1} \geq \dots \geq \mu_N $ the eigenvalues of $W_N$, by $  \mu_{1}' \geq \dots \geq \mu_N' $ the eigenvalues of $W_N'$, and by $ \lambda_{1} \geq \dots \geq \lambda_N $ the eigenvalues of $W_{J,N}'$. We transfer the properties of $ W_N' $ to $W_{J,N}'$ using the Cauchy interlacing theorem
\[
    \lambda_{1} \geq \mu_{1}' \geq \lambda_{2} \geq \mu_{2}' \geq \dots \geq \lambda_{N} \geq \mu_{N}'\,.
\]
}
In addition, we will rely on the \emph{stickiness} of the top eigenvalues of $W_N'$ to its deformed counterpart $W_{J,N}'$.

\begin{proposition}[Stickiness of top eigenvalues]
\label{prop:sticky}
Suppose $W_N'$ is a Wigner matrix whose off-diagonal moments match G(O/U)E up to third order and fix arbitrary {$ \epsilon \in(0, 1/6) $.} Let $J \in (0, 1)$ and $ W_{J,N}' = W_N' + J \bv\bv^* $ for a unit vector $\bv$. Then, w.o.p.,
\[
    \max_{j \leq N^{\epsilon}}|\lambda_{j} - \mu_{j}'| = O(N^{-1 + 2\epsilon}).
\]
\end{proposition}

\textcolor{black}{A proof of this proposition is given in \cref{proof:sticky}.} We remark that \cite{Knowles2013b} state the above bound for a constant number of top eigenvalues \textcolor{black}{and a somewhat different definition of Wigner matrices.} Our proof reproduces their arguments for up to $N^{\epsilon}$ eigenvalues \textcolor{black}{under definition \ref{Wigner definition} of a Wigner matrix.}

%The last ingredient that we need is the main result of \cite{johnstone2020logarithmic}:
%\begin{proposition}[Log determinant CLT] \label{prop:zero-diag}
%  Suppose that $W_N$ is drawn from (scaled) GOE or GUE, while
%  Let $W_N$ be a Wigner matrix satisfying \textbf{W1-4}.
  % whose off-diagonal entries have real and imaginary parts whose moments match GOE or GUE to third order.
 % Let \(E \in \R\) be such that \(\abs{E - 2} \lesssim N^{-\frac{2}{3}} \check{\sigma}_N\).
% \fix{Let $E = E_N = 2 + \ts_N N^{-2/3}$ with $\ts_N$
%a monotone sequence for which $\ts_N \geq -C$ for some positive constant $C$ and
%$\ts_N = o(\log^2 N)$.}
%  Then
 % Suppose that
 %  $E = 2 + \ts_N$ with $\ts_N$ a monotone sequence for which
 %  $\ts_N \geq 0$ and $\ts_N = o(\log^2 N)$.
 % With   $\delta_N = \sqrt{\sigma_N}$ we have
%  \begin{equation}
%  \label{eq:zW-clt}
%  \tau_N^{-1} (\log\abs{\det(W_N - E)} - \mu_N)  \drightarrow N(0,1).
%\end{equation}
%\end{proposition}

\subsection{Proof of \cref{lemma three points} for Wigner case} 

 \textit{Part (i):} Let $\zeta_{Nj}(W) = N^{2/3}(\lambda_j(W)-2)$, where $\lambda_j(W)$ denotes the $j$-th largest eigenvalue of matrix $W$.
  The convergence in distribution of $\mathbf{\zeta}_N = \mathbf{\zeta}_N(W_N)$ to
  $\mathbf{\zeta} = (\TW_{\frac{2}{\alpha},1}, \ldots, \TW_{\frac{2}{\alpha},k})$
  holds by definition, as discussed in \cref{sec: proof three poins}.
We use  \cref{prop:joint-convergence}
to carry this over to 
convergence of $\mathbf{\zeta}_N' = \mathbf{\zeta}_N(W_N')$: the key step is approximation by the
Stieltjes functional $g(W,E)$ below, and then use of the derivative
bounds in \cref{prop:system}.

Introducing the rescaling $E(s) = 2+N^{-2/3}s$, we may write
\begin{equation}
  \label{eq:switch}
  \bo \{ \zeta_{Nj}(W) < s \} = \bo \{ \mathcal{N}_W(E(s),\infty) < j \},
\end{equation}
where $\mathcal{N}_W(E,\infty)$ denotes the number of eigenvalues of $W$ that fall into the interval $[E, \infty)$. \textcolor{black}{Fix a small positive $\epsilon$.}
Let $E_\infty = 2 + 2N^{-2/3 + \epsilon}, \eta = N^{-2/3 - 9\epsilon}$
and
\begin{equation*}
  g(W, E)
    = \frac{N}{\pi} \int_E^{E_\infty} \Im s_W(y + \im \eta) \diff y.
\end{equation*}
%Under our assumptions on $W = W_N, W_N'$, 
Corollary 17.3 of \cite{ey17}
says that for Wigner matrices $W$, $|E-2| \leq N^{-2/3+\epsilon}$ and
$\ell = \frac{1}{2} N^{-2/3-\epsilon}$, and with overwhelming
probability, we have inequalities
\begin{equation}
  \label{eq:ey17.3}
  \mathcal{N}_W(E+\ell,\infty) - N^{-\epsilon}
  \leq g(W,E)
  \leq \mathcal{N}_W(E-\ell,\infty) + N^{-\epsilon}.
\end{equation}
\cite{ey17} use a somewhat different definition of Wigner matrices, but we show that \eqref{eq:ey17.3} still holds with \cref{Wigner definition} in \cref{sec:proof of eq 6.6}.
 Let \(G_j\) be a smooth \textcolor{black}{decreasing} function such that
\begin{equation*}
  G_j(x) =
  \begin{cases}
    0 & x \geq j -1/3 \\
    1 & x \leq j -2/3.
  \end{cases}
\end{equation*}
From \eqref{eq:ey17.3} we have w.o.p. for $W = W_N$ and $W_N'$ that
\begin{equation*}
  \bo \{ \mathcal{N}_W(E+\ell,\infty) < j \}
  \geq G_j(g(W,E))
  \geq \bo \{ \mathcal{N}_W(E-\ell,\infty) < j \}.
\end{equation*}
Applying this with $E = E(s) + \ell$ along with \eqref{eq:switch},
we obtain
\begin{equation*}
      \bo \{ \zeta_{Nj}(W) < s + N^{-\epsilon} \}\geq G_j(g(W,E(s)+\ell))
  \geq \bo \{ \zeta_{Nj}(W) < s \},
\end{equation*}
which implies
\begin{equation}
    \label{eq:lgstsw}
   \bo \{ \zeta_{Nj}(W) \leq s-N^{-\epsilon} \}\leq  G_j(g(W,E(s)+\ell))  \leq \bo \{ \zeta_{Nj}(W) \leq s + N^{-\epsilon}\}.
\end{equation}
\textcolor{black}{Setting $Q_j(W,s) = G_j(g(W,E(s)+\ell))$, we obtain bounds
\eqref{eq:Qj}.}

The functions $Q_j(\cdot,s)$ satisfy Proposition \ref{prop:system}
(2) with $\delta_{j,N} = N^{-1/3 + O(\epsilon)}$ and hence also
condition F. Consequently the joint convergence for
$\bxi_N' = \bxi_N(W_N')$ follows
from Proposition \ref{prop:joint-convergence}.

The result follows for subcritically-spiked Wigner matrix $W_{J,N}' = {W_N'} + J \vec{v}\vec{v}^*$ applying \cref{prop:sticky} so that
\begin{align*}
    \bigl(N^{\frac{2}{3}}(\lambda_1 - 2), \dotsc, N^{\frac{2}{3}}(\lambda_k - 2)\bigr)
    &= \bigl(N^{\frac{2}{3}}(\mu_1' - 2), \dotsc, N^{\frac{2}{3}}(\mu_N' - 2)\bigr) + o_{\Pr}(N^{-1/3 + 2\epsilon}).
\end{align*}

\textit{Part (ii) and part (iv):} The statement of part (ii) \textcolor{black}{is implied by} $ N^{2/3}(\lambda_{1} - 2) = O_{\Pr}(1) $, which follows from the fact that $ N^{2/3}(\lambda_{1} - 2) \drightarrow \TW_{{2}/{\alpha},1}$.
 Similarly, we have that
\[
    N^{2/3}(\lambda_1 - \lambda_2) \drightarrow \TW_{{2}/{\alpha},1} - \TW_{{2}/{\alpha},2},
\]
and the latter is $ \Theta_{\Pr}(1) $, since we already know that $ N^{2/3}(\mu_1 - \mu_2) \drightarrow \TW_{{2}/{\alpha},1} - \TW_{{2}/{\alpha},2} $, and $ \mu_1 - \mu_2 = \Theta_{\Pr}(N^{-2/3}) $ from the proof for the Gaussian ensembles.

\textit{Part (iii) and part (v):} We first derive the corresponding bounds for $ W_N' $, i.e. for the eigenvalues $\mu_j'$. We use the counting function approximation and it's universality, similar to how we did in the proof of \textcolor{black}{part (i)}. 
%Recall the notation $ \mathcal{N}_{W}(E; \infty) = \#\{ j : \lambda_j(W) \geq E\} $, $E_{\infty} = 2 + 2N^{-2/3 + \varepsilon}$, and $\ell = \frac{1}{2}N^{-2/3 - \varepsilon}$, and let
%\[
%    g(W, E) = \frac{N}{\pi} \int_{E}^{E^{\infty}} \Im s_{W}(u + \im N^{-2/3 - 9 \varepsilon}) du \, .
%\]
%Then, 
It follows from \cref{eq:ey17.3} that w.o.p. for both $W_N, W_N'$ in place of $W$,
\begin{equation}\label{counting_to_integral}
    g(W, E + \ell) - N^{-\varepsilon} \leq \mathcal{N}_{W}(E, \infty) \leq g(W, E - \ell) + N^{-\varepsilon}.
\end{equation}
For part (iii), first take the expectation of the above inequalities. Notice that $ \mathcal{N}_{W}(E, \infty) \leq N $ a.s. Similarly, using $ \Im (\lambda - y - \im \eta)^{-1} \leq \eta^{-1} $, it holds a.s. that $ g(E, W) \leq \frac{N}{\pi} \eta^{-1} (E^{\infty} - E) $, which can be bounded as $ O(N^{1 + O(\epsilon)}) $. Since the complement to a w.o.p. event happens with probability at most $ N^{-D} $ for any $D > 0$ and large $N$, we conclude that
\[
    \E g(W, E + \ell) - N^{-\varepsilon} \leq \E \mathcal{N}_{W}(E, \infty) \leq \E g(W, E - \ell) + N^{-\varepsilon} .
\]
Using Propositions~\ref{prop:multi-matching} and \ref{prop:system} part(2) applied with $G(x) = x$, we have that
\[
    \abs{\E g(W_N, E) - \E g(W_N', E)} \lesssim N^{-1/3 + O(\epsilon)}.
\]
Take $ E_{x} = 2 - x N^{-2/3} $ and $ E_{x+1} = 2 - (x + 1) N^{-2/3} $ so that $ E_{x} - 2\ell \geq E_{x+1}$ for large $N$. Then,
\begin{align*}
    \E \mathcal{N}_{W_N'}(E_x, \infty) & \leq \E g(W_N', E_x - \ell) + N^{-\varepsilon} \\
    & \leq \E g(W_N, E_x - \ell) + N^{-\epsilon} + O(N^{-1/3 + O(\epsilon)}) \\
    & \leq \E \mathcal{N}_{W_N}(E_{x+1}, \infty) + 2N^{-\epsilon} + O(N^{-1/3 + O(\epsilon)}) \\
    &\leq O(1) \, ,
\end{align*}
where we used that $ \E \mathcal{N}_{W_N}(E_{x+1}, \infty) = O(1) $ for the Gaussian case, which is shown in Section~\ref{sec: proof three poins}. To extend part (iii) to the spiked case, observe that the values $ \#\{ j: \; \lambda_j \geq 2 - x N^{-2/3}\} $ and $ \#\{j: \; \mu_{j}' \geq 2 - xN^{-2/3} \} $ differ by at most one thanks to the Cauchy interlacing theorem.

For part (v), 
%Let $W_N$ be GUE and $W_N'$ be a Wigner matrix whose moments match
%those of $W_N$ up to third order.
let us show that there exists $\kappa'$ such that if $b_N \to \infty$
so that $b_N = O(N^\epsilon)$ for all $\epsilon >0$, then a.a.s.
\begin{equation*}
  \mathcal{N}_{W'}(E_{b_N},\infty) \geq \kappa' b_N^{2/3},
\end{equation*}
where $E_{b_N} = 2- b_N N^{-2/3}$. %As before, let $E_\infty = 2 +
%2N^{-2/3 + \epsilon}, \eta = N^{-2/3 - 9\epsilon},$ and $\ell =
%\frac{1}{2} N^{-2/3 - \epsilon}$. Now define
%\begin{equation*}
%  g(W) = \frac{N}{\pi} \int_{E_{b_N}+\ell}^{E_\infty} \Im s_W(x + \im
%  \eta) \diff x.
%\end{equation*}
%
From \eqref{counting_to_integral} we have for $b_N = O(N^\epsilon)$, w.o.p.
\begin{equation}
  \label{eq:bracket}
  \mathcal{N}_W(E_{b_N - N^{-\epsilon}}, \infty) - N^{-\epsilon}
  \leq g(W, \textcolor{black}{E_{b_N - N^{-\epsilon} / 2}})
  \leq \mathcal{N}_W(E_{b_N}, \infty) + N^{-\epsilon},
\end{equation}
where we note that $E_{b_N}+2 \ell = E_{b_N - N^{-\epsilon}}$ and $E_{b_N}+ \ell = E_{b_N - N^{-\epsilon} / 2}$. Below we denote for short $ g(W) = g(W, E_{b_N - N^{-\varepsilon}/2}) $.

Let $\mu = 5 \kappa'/4$ and $H$ be a smooth decreasing function on
$[0,\infty)$ such that
\begin{equation*}
  H(x) =
  \begin{cases}
    1 & x \leq \mu \\
    0 & x \geq 2 \mu,
  \end{cases}
\end{equation*}
and define $G(x) = H(x/b_N^{3/2})$, so that
$\| G^{(j)} \|_\infty \leq b_N^{-3j/2} = o(1) $. % \check{b}_N^j$.

%We apply \cite{johnstone2020logarithmic} 
 %Proposition 24, writing here $\check{b}_N$ here for $b_N$ there
%\texttt{[due to NOTATION CLASH to fix!]} with $\check{b}_N  =
%  b_N^{-3/2}$,
%  and to $g$ with $a_N = N^{-1/3 + O(\epsilon)}$ and hence
%  $\delta_N = a_N \check{b}_N = o(N^{-1/3 + O(\epsilon)})$ so
%  Proposition \ref{prop:multi-matching} yields
%  $\E G(g(W_N')) \leq \E G(g(W_N))  + O(\delta_N)$.

We apply Proposition~\ref{prop:multi-matching} and Proposition~\ref{prop:system} part (2), from where it follows that $ \E G(g(W_N')) \leq \E G(g(W_N)) + O(N^{-1/3 + O(\epsilon)}) $.
 
Using \eqref{eq:bracket} in the first and fourth lines below, and
setting $\mu = 5\kappa'/4$, we obtain
\begin{align*}
  \mP\{ \mathcal{N}_{W_N'}(E_{b_N},\infty) < \kappa' b_N^{3/2} \}
  & \leq  \mP \{ g(W_N')/b_N^{3/2} \leq \mu \} + O(N^{-D}) \\
  & \leq \E G(g(W_N')) + O(N^{-D}) \leq \E G(g(W_N))  + O(\delta_N) \\
  & \leq \mP \{ g(W_N)/b_N^{3/2} \leq 2\mu \} + O(N^{-1/3 + O(\epsilon)}) \\
  & \leq  \mP\{ \mathcal{N}_{W_N}(E_{b_N - N^{-\epsilon}},\infty) <
    \kappa  b_N^{3/2} \} + O(N^{-1/3 + O(\epsilon)})
\end{align*}
if we set $2 \mu = 3 \kappa/4$ say. The final bound is $o(1)$,
according to the Gaussian case applied to $b_N \leftarrow b_N -
N^{-\epsilon}$, and so if we take $\kappa' = 4 \mu/5 = 3 \kappa/10$,
we obtain the claimed result.

Finally, part (v) extends trivially to the spiked case using Cauchy interlacing theorem, since
\[
\pushQED{\qed}
    \#\{ j: \; \lambda_j > 2 - b_NN^{-2/3} \} \geq \#\{ j: \; \mu_j' > 2 - b_NN^{-2/3} \}  \, .\qedhere
    \popQED
\]

\subsection{\textcolor{black}{Proof of lemma \ref{lemma derivatives of G}  and lemma \ref{lemma lambda1statistic} for Wigner case}}
\label{sec: proof lemma deriv G Wigner}
\color{black}

Let us first extend lemma \ref{lemma derivatives of G}  and eq. \ref{eq:p4-bd} of lemma \ref{lemma lambda1statistic} to the non-spiked Wigner case, with $W_N'$ in place of $W_N$ and $ \mu_j' $ in place of $\mu_j$.

First we rewrite \eqref{eq:DG-bound1} and \eqref{eq:p4-bd} in terms of
Stieltjes transforms.
%Since $G_W(z) = \beta z - N^{-1} \sum_1^N \log(z - \lambda_j(W))$, we have 
For $l \geq 1$, we have
\begin{equation*}
  G^{(l)}(z) = \beta \bo(l=1) + s_W^{(l-1)}(z)
\end{equation*}
\textcolor{black}{for $W_N$ and $W_n'$ in place of $W$.}
Suppose that $|E-2| \lesssim \check{\sigma}_N N^{-2/3}$, \textcolor{black}{where
\[
\check{\sigma}_N:=(\log N)^{O(\log\log N)},
\]} 
and define
\begin{equation*}
  g_0(W)= \textcolor{black}{g_0(W,E) :=} N^{-2l/3+1} s_W^{(l-1)}(E)
         = (l-1)! N^{-2l/3} \sum_{j=1}^N (\lambda_j\textcolor{black}{(W)}-E)^{-l}.
\end{equation*}
\textcolor{black}{Here $\lambda_j(W)$ denotes the $j$-th largest eigenvalue of $W$.}
Then \eqref{eq:DG-bound1} and \eqref{eq:p4-bd} take the simpler form
\begin{equation}
  \label{eq:g0prop}
  g_0(W\textcolor{black}{, E}) = \textcolor{black}{\alpha_N} + \sigma_N Z_N, \qquad
 Z_N = o_\Pr(1) \text{\; {or} \;} O_\Pr(1)
\end{equation}
\textcolor{black}{with the following $\alpha_N$, $\sigma_N$ and $Z_N$.}

In the case of \eqref{eq:DG-bound1}, we take $E = \hat{\gamma}$. Further, 
in \eqref{eq:g0prop} we have $Z_N = o_\Pr(1)$,
\begin{align*}
  \textcolor{black}{\alpha_N} = -N^{1/3}\bo(l=1) + c_l |b|^{3-2l} \log^{3/2-l} N, \qquad
  \sigma_N = \log^{-L}N, \qquad L =
             \begin{cases}
               1/4 & l=1 \\
               l - 3/2  &  l \geq 2
             \end{cases} 
\end{align*}
with $c_1 = 1, c_2 = 1/2$ and with the general
form of $c_l$ visible in \eqref{eq:DG-bound1}.

In the case of \eqref{eq:p4-bd}, we take $E = 2+CN^{-2/3}$, and for $l
= 1, 2$ in \eqref{eq:g0prop} we have
\begin{equation*}
  \textcolor{black}{\alpha_N} = -N^{1/3}\bo(l=1), \qquad \sigma_N = 1, \qquad Z_N = O_\Pr(1).
\end{equation*}

Thus %Lemma 3.2 of BL transition paper and Proposition 4 of log-det CLT
%paper establish 
the validity of \eqref{eq:g0prop} for $W=W_N$ drawn from
Gaussian ensembles \textcolor{black}{has been already established in sections \ref{sec: proof lemma deriv G} and \ref{sec:lemma2.9 Gaussian}}. We wish to carry this over to
$Z_N' = (g_0(W_N')-\textcolor{black}{\alpha_N})/\sigma_N$ for a Wigner matrix $W_N'$.
To do this, we approximate $g_0(W)$ by the Stieltjes functional
\begin{equation*}
  g(W) = N^{-2l/3 +1} \Re s_W^{(l-1)}(E+ \im \eta).
\end{equation*}

\bigskip

\begin{lemma} \ (Approximation step) \
  \label{lem:complex-inv-moment}
  Let \(W\) be an \(N\times N\) Wigner matrix satisfying \textcolor{black}{Assumption W}
 and let  \(E \in \R\) be such that \(\abs{E - 2} \leq
 N^{-\frac{2}{3}} \check{\sigma}_N\).   
  Let \(\epsilon > 0\) and define \(\eta = N^{-\frac{2}{3} - 3\epsilon}\).

  Then, for all \(l \in \Z_{+}\), we have with high probability that
  \begin{align}
    \label{eq:off-real1}
    N^{-2l/3+1} s_W^{(l-1)}(E)
      = N^{-2l/3+1} \Re s_W^{(l-1)}(E + \im \eta) +O(N^{-\epsilon}).
    % \sum_{j=1}^N \frac{1}{(E - \lambda_j)^l}
    % = (-1)^l N \Re s_{W}^{(l-1)}(E + \im \eta) + O(N^{\frac{2}{3}l - \epsilon}),
  \end{align}
%  where \(\{\lambda_j\}\) are the eigenvalues of \(W\).

  \proof

  Let \(\epsilon_0 = \epsilon/(l+1)\).
  By eigenvalue non-concentration \textcolor{black}{(see proposition 25 of \cite{johnstone2020logarithmic})}, %prop:wign-non-conc-2
  there then exists a constant \(d > 0\) such that the event
  \begin{align*}
    E_N
    &= \Bigl\{\min_{1 \leq j \leq N} \abs{\lambda_j(W) - E} \geq N^{-\frac{2}{3} - \epsilon_0}\Bigr\}
  \end{align*}
  holds with probability at least \(1 - N^{-d}\).
  The rest of the argument occurs on the event \(E_N\).
  
  Now the function \(\sum_{j=1}^N \frac{1}{(z - \lambda_j(W))^l}\) is holomorphic in the open disk \(\{z : \abs{z - E} < N^{-\frac{2}{3} - \epsilon_0}\}\).
  Since \(\epsilon_0 < \epsilon\), the vertical segment \(\gamma\) connecting \(E\) to \(E + \im \eta\) lies entirely within this disk, so the fundamental theorem of calculus applies, rendering
  \begin{align*}
    \abs[\Big]{\Re \sum_{j=1}^N \frac{1}{(E - \lambda_j(W) + \im \eta)^l} - \sum_{j=1}^N \frac{1}{(E - \lambda_j(W))^l}}
    &= \abs[\Big]{\Re \int_\gamma \sum_{j=1}^N -\frac{l}{(z - \lambda_j(W))^{l+1}} \diff z} \\
    &\leq l \eta  \sum_{j=1}^N \frac{1}{\abs{E - \lambda_j(W)}^{l+1}}.
  \end{align*}
  
   By lemma 26 of \cite{johnstone2020logarithmic},
   this is %\cref{lem:neg-mom-bd}, this is
  \(O(N^{-\frac{2}{3} -3\epsilon} \cdot
  N^{(\frac{2}{3}+\epsilon_0)(l+1)+\epsilon})
  = O(N^{\frac{2}{3}l-\epsilon})\)
  w.o.p. on \(E_N\), from which the result follows.\qed
\end{lemma}

Lemma \ref{lem:complex-inv-moment} implies that $g(W)$ satisfies
\eqref{eq:g0prop} exactly when $g_0(W)$ does. So we carry out the
Lindeberg swapping with $g(W)$. 

Let $\kappa > 0$ and $H: \mR \to [0,1]$ be a smooth cutoff function
satisfying
\begin{equation*}
  H(x) =
  \begin{cases}
    0 & \text{if } \ |x| \leq \kappa \\
    1 & \text{if } \ |x| \geq 2 \kappa.
  \end{cases}
\end{equation*}
Let $G(x) = H((x-\alpha_N)/\sigma_N)$, so that
$\| G^{(j)} \|_\infty \lesssim \sigma_N^{-j}$.
Propositions \ref{prop:system} (3) 
%to $G$ with $b_N =
%\sigma_N^{-1}$ and to $g$ with
%$a_N = N^{-1/3+O(\epsilon)}$ and hence
%$\delta_N = N^{-1/3+O(\epsilon)}$, so that 
and \ref{prop:multi-matching} yield
$\E G(g(W_N')) = \E G(g(W_N)) + O(\delta_N)$ \textcolor{black}{with $\delta_N=N^{-1/3+O(\epsilon)}$}.
Write $\check{Z}_N = (g(W_N)- \alpha_N)/\sigma_N$ and similarly for
$\check{Z}_N'$. We conclude that
\begin{align*}
  \Pr( |\check{Z}_N'| > 2\kappa )
  & = \Pr( |g(W_N')-\mu_N| > 2\kappa \sigma_N) \\
  & \leq \E G(g(W_N')) \leq \E G(g(W_N)) + O(\delta_N) \\
  & \leq \Pr(  |\check{Z}_N| > \kappa ) + O(\delta_N).
\end{align*}
A similar bound holds reversing the roles of $W_N$ and $W_N'$.
Consequently $\check{Z}_N'$ is $o_\Pr(1)$ or $O_\Pr(1)$ exactly when
$\check{Z}_N$ is.
From Lemma \ref{lem:complex-inv-moment}, both $\check{Z}_N - Z_N$ and
$\check{Z}_N' - Z_N'$ are $O(N^{-\epsilon})$ with high probability.
Thus \eqref{eq:g0prop} carries over to $W_N'$ and so the validity of \textcolor{black}{\eqref{eq:DG-bound1} and \eqref{eq:p4-bd} are established for Wigner matrices without a spike.}

\medskip

We conclude the proof by extending the bounds from non-spiked Wigner matrix $W_N'$ to the spiked one $W_{J,N}'$ as in \eqref{spiked_definition}. Let us show that \eqref{eq:DG-bound1} and \eqref{eq:p4-bd} still hold in this case. Recall that $\lambda_j$ denote the eigenvalues of $ W_{J,N}' $ in the descending order, and $\mu_j'$ are the eigenvalues of $W_N'$.

Let $ \gamma$ equals to either $ \hat{\gamma} $ from \eqref{eq:DG-bound1} or $ 2 + CN^{-2/3} $ from \eqref{eq:p4-bd}. In addition, let $ i^{*} $ denotes the index of the nearest to $\gamma$  among the eigenvalues $ \mu_i'  $, \textcolor{black}{such that $\mu_i'\leq\gamma$}. Due to the interlacing theorem, we have that $ 0 \leq {\gamma} - \lambda_i \leq {\gamma} - \mu_i' $ for $i > i^{*}$ and $ {\gamma} - \lambda_i \leq {\gamma} - \mu_i' \leq 0 $ for $i < i^{*}$.

Then,
\[
    \sum_{i = 1}^{N} ({\gamma} - \lambda_i)^{-l} \geq \sum_{i = 1}^{N} ({\gamma} - \mu_i')^{-l} - ({\gamma} - \mu_{i^*}')^{-l} + ({\gamma} - \lambda_{i^*})^{-l}
\]
It follows from rigidity \textcolor{black}{(see theorem 2.9 of \cite{BenaychG2018})}, that for any $ \epsilon > 0 $ w.o.p.
\[
    O(N^{-2/3} \log N) = \left(\frac{i^*}{N}\right)^{2/3}  + O\left(N^{-2/3 + \epsilon} (i^{*})^{-1/3}\right),
\]
which implies $ i^{*} = O(N^{\epsilon}) $. Therefore, using proposition~\ref{prop:sticky} we have that w.o.p. $ |\lambda_{i^*} - \mu_{i^{*}}'| \lesssim  N^{-1 + 3\epsilon}$. Furthermore, by the non-concentration result from proposition 25 of \cite{johnstone2020logarithmic} %\ref{prop:wign-non-conc-2},
we have, with high probability, $ |{\gamma} - \mu_{i^*}'|^{-l} \leq N^{2l/3 + l\epsilon} $. Hence we obtain that  with high probability
\begin{align}
\label{eq:mu and lambda for i*}
    ({\gamma} - \mu_{i^{*}}')^{-l} - ({\gamma} - \lambda_{i^*})^{-l} &= ({\gamma} - \mu_{i^{*}}')^{-l} \left[ 1 - \left( 1 + \frac{\mu_{i^*}' - \lambda_{i^{*}}}{{\gamma} - \mu_{i^*}'}\right)^{-l}\right] \\
    &= O(N^{2l/3 + l\epsilon}) \left[1 - \left(1 + O(N^{-1/3 + C\epsilon})\right)^{-l}\right] \notag\\
    &= O(N^{(2l-1)/3 + C\epsilon})\notag
\end{align}
Taking $\epsilon$ sufficiently small, we obtain that for any $ L > 0$,
\begin{equation}
\label{eq:lower bound for inverse moment}
    \sum_{i = 1}^{N} ({\gamma} - \lambda_i)^{-l} \geq \sum_{i = 1}^{N} ({\gamma} - \mu_i')^{-l} + o_{\Pr}(N^{2l/3} \log^{-L} N) \, .
\end{equation}

To obtain the inequality in the opposite direction, note that with high probability, there exists $C>0$ such that
for all $i\leq i^{**}:=[N^\epsilon]$
\begin{equation}
\label{eq:many comparisons}
(\gamma-\mu_i')^{-l}-(\gamma-\lambda_i)^{-l}=O(N^{(2l-1)/3+C\epsilon}).
\end{equation}
This fact can be established similarly to \eqref{eq:mu and lambda for i*}. Furthermore, by eigenvalue rigidity w.o.p.
\[
\gamma-\mu_{i^{**}}'= \gamma-\gamma_{i^{**}}+\gamma_{i^{**}}-\mu_{i^{**}}'\geq \gamma-\gamma_{i^{**}}- O(N^{-2/3-\epsilon/3}), 
\]
where $\gamma_{i^{**}}$ is the typical location of the $i^{**}$-th eigenvalue, satisfying
\[
2-\gamma_{i^{**}}\asymp \left(\frac{i^{**}}{N}\right)^{2/3}\asymp N^{-2/3+2\epsilon/3}.
\]
Since $\gamma=2+O(N^{-2/3}\log N)$, we obtain w.o.p
\begin{equation}
\label{eq:mu_i** is small}
\gamma-\mu_{i^{**}}'\geq N^{-2/3+\epsilon/2}.
\end{equation}
Therefore, by the interlacing theorem, w.o.p. 
\begin{equation}
\label{eq:opposite}
(\gamma-\lambda_{i+1})^{-l}\leq (\gamma-\mu_{i}')^{-l}
\end{equation}
for all $i\geq i^{**}$.

Using \eqref{eq:many comparisons}-\eqref{eq:opposite}, we have with high probability
\begin{align*}
    \sum_{i=1}^N(\gamma-\lambda_i)^{-l}\leq &\sum_{i=1}^N(\gamma-\mu_i')^{-l}-\sum_{i=1}^{i^{**}}\left[(\gamma-\mu_i')^{-l}-(\gamma-\lambda_i)^{-l}\right]+(\gamma-\mu_{i^{**}})^{-l}-(\gamma-\mu_{N})^{-l}\\
    =&\sum_{i=1}^N(\gamma-\mu_i')^{-l}+O(N^{(2l-1)/3+(C+1)\epsilon})+O(N^{2l/3-\epsilon l/2})+O(1)
\end{align*}
Choosing $\epsilon$ sufficiently small, we obtain that for any $L>0$,
\[
\sum_{i=1}^N(\gamma-\lambda_i)^{-l}\leq \sum_{i=1}^N(\gamma-\mu_i')^{-l}+o_{\Pr}(N^{2l/3}\log^{-L} N).
\]
Combining this with \eqref{eq:lower bound for inverse moment}, we conclude that
\[
\sum_{i=1}^N(\gamma-\lambda_i)^{-l}= \sum_{i=1}^N(\gamma-\mu_i')^{-l}+o_{\Pr}(N^{2l/3}\log^{-L} N).
\]

\textcolor{black}{Taking $L>\max\left\{\frac{1}{4}, l-\frac{3}{2}\right\}$}, we see that the difference of $ o_{\Pr}(N^{2l/3} \log^{-L} N) $  between the spiked statistics and non-spiked one is sufficient for  \eqref{eq:DG-bound1} and \eqref{eq:p4-bd} to hold in the spiked case as well. 

\textcolor{black}{Finally, equation \ref{inv mom 1} follows from \eqref{eq:p4-bd} in exactly the same way as in the Gaussian case, considered in section \ref{sec:lemma2.9 Gaussian}.}\qed

\subsection{Proof of \cref{independence} for Wigner case}
\label{sec:proof-indep-wigner}

First we consider the case with no spike, i.e. $J = 0$ and eigenvalues $\mu_j'$ of $W_N'$ in place of eigenvalues $\lambda_j$ of $W_{J,N}'$. In this case, this is an immediate consequence of proposition
\ref{prop:joint-convergence} and previous arguments for the
log-determinant in \cite{johnstone2020logarithmic} and the
largest eigenvalue in \cref{lemma three points} (i).
Indeed, in the proof of Proposition~{27} of \cite{johnstone2020logarithmic} we show that
\[
    \xi_{1N}(W_N') = \tau_N^{-1} (\bar{\mu}_N-g_0(W_N')) + o_{\Pr}(1),
\]
where for $\gamma_N = N^{-2/3 - 2\epsilon}$ with $\epsilon > 0$ small
\[ 
    g_0(W) = \int_{\gamma_N}^{N^{100}} \Im s_W(2 + i \eta) \diff\eta,
    \qquad \bar{\mu}_N = N/2-\frac{\alpha-1}{6}\log N + N \log(N^{100}).
\]
It is enough to consider joint convergence of
\begin{equation*}
  \tilde{\xi}_{1N}(W_N') = \tau_N^{-1}(\bar{\mu}_N-g_0(W_N'))
  \qquad\text{and}\qquad \xi_{2N}(W_N'),
\end{equation*}
since $(\xi_{1N}(W_N'),\xi_{2N}(W_N')) = (\tilde{\xi}_{1N}(W_N'), \xi_{2N}(W_N')) +
o_\Pr(1)$. 

In Proposition \ref{prop:joint-convergence}, \textcolor{black}{let us set $Q_1(W,s)$ and $Q_2(W,s)$ as follows. We start from $Q_2(W,s)$. Using \eqref{eq:lgstsw}, take $Q_2(W,s):=G_1(g(W,E(s) \pm \ell)$. Then, as explained immediately after \eqref{eq:lgstsw}, such $Q_2(W,s)$ satisfies condition $F(\delta_{2,N})$ with $\delta_{2,N}=N^{-1/3+O(\epsilon)}$. Furthermore, \eqref{eq:lgstsw} says that
\begin{equation}
\label{eq:Q1}
\bo\{\xi_{N2}(W)\leq s-N^{-\epsilon}\}\leq Q_2(W,s) \leq \bo\{\xi_{N2}(W)\leq s+N^{-\epsilon}\}.
\end{equation}}

\textcolor{black}{
Turning to $Q_1(W,s)$, let $H: \mathbb{R} \to [0,1]$ be a smooth decreasing function such that 
\begin{equation*}
  H(x) =
  \begin{cases}
    1 & \text{if } \ x \leq -\eta_N \\
    0 & \text{if } \ x \geq \eta_N.
  \end{cases}
\end{equation*}
Define $Q_1(W,s)=G_s(g_0(W)):=H(\tau_N^{-1}(\bar{\mu}-g_0(W))-s)$. Then clearly
\begin{equation}
\label{eq:Q2}
   \bo\{\xi_{N1}(W)\leq s-\eta_N\}\leq Q_1(W,s) \leq \bo\{\xi_{N1}(W)\leq s+\eta_N\}. 
\end{equation}
Observe that if we choose $\eta_N=(\log N)^{-1/4}$, then  
$\norm{G_s^{(j)}}_\infty\lesssim(\tau_N\eta_N)^{-j}\lesssim(\log N)^{-j/4}$. Then \cref{prop:system} (1) implies that $Q_1(\cdot,s)$ satisfy condition F with $\delta_{1N}=(\log N)^{-1/4}$.}

\textcolor{black}{Convergence of the Gaussian versions
$(\xi_{1N},\xi_{2N}) \stackrel{\rm d}{\to} \mathcal{N}(0,1) \times
\TW_{2/\alpha}$ has been established in \cref{sec:independence for Gaussian case}. Hence, the conclusion for $(\xi_{1N}(W_N'),\xi_{2N}(W_N'))$ follows now from equations \eqref{eq:Q1}, \eqref{eq:Q2} and \cref{prop:joint-convergence}.}

As for the spiked case, eq. (95) of \cite{johnstone2020logarithmic} shows that $ \xi_{1N}(W_{J,N}') = \xi_{1N}(W_N') + o_{\Pr}(1) $ for a fixed $J \in (0, 1)$. Moreover, thanks to the stickiness property of Proposition~\ref{prop:sticky}, we also have $ \xi_{2N}(W_{J,N}') = \xi_{2N}(W_N') + o_{\Pr}(1) $. Therefore, the limiting distribution of $ (\xi_{1N}(W_{J,N}'), \xi_{2N}(W_{J,N}')) $ does not change as long as the spike is sub-critical. \qed

% \appendix
\section{Technical appendix}
\label{sec: appendix}

\def\gt{\tilde{\gamma}}
\def\Rtone{\tilde{R}^{(1)}}

\subsection{Remarks on contour representation \eqref{eq for z}. } \ 
\label{sec:remarks-cont-repr}
%We refer to the proof given in \cite[][Lemma 1.3]{Baik2016}.
The two cases $\alpha = 1, 2$ may be obtained together % by a suitable rephrasing. 
\textcolor{black}{by following the proof given in \cite[][Lemma 1.3]{Baik2016}.}
%Indeed, 
Write normalized measure on $S^{n-1} = \{ x \in \R^n : \| x \| = 1
\}$ as $d \Omega/|S^{n-1}|$. By diagonalizing \textcolor{black}{$W_{J,N}$} and changing
variables, the real and complex partition functions are
respectively 
\begin{align*}
  \textcolor{black}{\mathcal{I}_{2,J,N}}
  & = \frac{1}{|S^{N-1}|} \int_{S^{N-1}} \exp \Big\{ \tfrac{1}{2} N\beta
    \sum_{i=1}^N \lambda_i x_i^2 \Big\} \, d \Omega, \\
    \textcolor{black}{\mathcal{I}_{1,J,N}}
  & = \frac{1}{|S^{2N-1}|} \int_{S^{2N-1}} \exp \Big\{ N\beta
    \sum_{i=1}^N \lambda_i (x_{2i-1}^2 + x_{2i}^2) \Big\} \, d \Omega.
\end{align*}
\textcolor{black}{Both partition functions have form
\begin{equation}
\label{eq: the same form}
\frac{1}{|S^{n-1}|}\int_{S^{n-1}}\exp\Big\{\frac{1}{2}n\beta\sum_{j=1}^n\mu_j x_j^2\Big\}\mathrm{d}\Omega,
\end{equation}
where $n=N$, $\mu_j=\lambda_j$ for the real case; and $n=2N$, $\mu_{2i-1}=\mu_{2i}=\lambda_i$ for the complex case. \cite{Baik2016} derive a contour integral representation for such a spherical integral by evaluating
\[
J(z)\equiv \int_{\mathbb{R}^n} \exp\Big\{\frac{1}{2}n\beta\sum_{j=1}^n\mu_jy_j^2\Big\}\exp\Big\{-\frac{1}{2}n\beta z\sum_{j=1}^n y_j^2\Big\}\mathrm{d}y,\quad \operatorname{Re} z>\mu_1
\]
in two different ways. First, directly as a Gaussian integral, and then, using polar coordinates. For the reader's convenience, we reproduce these two steps below.}

\textcolor{black}{Evaluating $J(z)$ directly, we obtain
\begin{equation}
\label{eq: direct J}
J(z)=\left(\frac{2\pi}{\beta n}\right)^{n/2}\prod_{j=1}^n\frac{1}{\sqrt{z-\mu_j}},\quad \operatorname{Re} z>\mu_1.
\end{equation}
On the other hand, using polar coordinates $y=r x $, and further letting $\frac{1}{2}\beta n r^2 =t$, yields
\begin{eqnarray*}
    J(z)&=&\frac{1}{2(\beta n/2)^{n/2}}\int_0^\infty e^{-z t}t^{(n/2)-1}I(t)\mathrm{d}t,\\
    I(t)&\equiv&\int_{S^{n-1}}e^{t\sum_{j=1}^n\mu_j x_j^2}\mathrm{d}\Omega. 
\end{eqnarray*}
Next, \cite{Baik2016} note that $J(z)$ is proportional to the Laplace
transform of $t^{(n/2)-1}I(t)$. They use the inverse Laplace transform together with \eqref{eq: direct J} to arrive at
\[
I(t)=\frac{2\pi^{n/2}}{t^{n/2-1}}\frac{1}{2\pi\mathrm{i}}\int_{\gamma-\im\infty}^{\gamma+\im\infty}e^{zt}\prod_{j=1}^n\frac{1}{\sqrt{z-\mu_j}}\mathrm{d}z,\quad \gamma>\mu_1.
\]
Using this with $t=n\beta/2$ and recalling that $|S^{n-1}|=2\pi^{n/2}/\Gamma(n/2)$, we see that \eqref{eq: the same form} equals
}
%
%Replacing $\beta, N, \lambda$ in \cite{Baik2016} Lemma 1.3 by
%placeholders $\check{\beta}, n, \mu$, and noting their equations
%(4.2), (4.7), we have
\begin{align*}
  & \frac{1}{|S^{n-1}|} \int_{S^{n-1}} e^{ \textcolor{black}{\frac{1}{2}n\beta} %\check{\beta} 
    \sum_{1}^n \mu_i x_i^2 } \, \mathrm{d} \Omega
     = C_n \int_{\gamma - \im \infty}^{\gamma + \im \infty}
        e^{\frac{n}{2} G(z)} \ dz, \qquad \gamma > \mu_1 \\
  &G(z)  = %2 \check{\beta} z \textcolor{black}{\beta z}
  \textcolor{black}{\beta z} - \frac{1}{n} \sum_1^n \log(z - \mu_i),
  \qquad
  C_n   = \frac{1}{2 \pi \im} \frac{\Gamma(n/2)}%{(n \check{\beta})^{n/2 -1}}.
  {\textcolor{black}{(n\beta/2)^{n/2-1}}}.
\end{align*}
To recover our \eqref{eq for z}, set \textcolor{black}{$n=N$, $\mu_j=\lambda_j$ for the real case, and $n=2N$, $\mu_{2i-1}=\mu_{2i}=\lambda_i$ for the complex case}.

%$\check{\beta} = \beta/2$.
%In the real case, put $n = N$ so that $n \check{\beta} = N\beta/2$. 
%In the complex case, set
%$n = 2N$ (so that $n \check{\beta} = N\beta$)
%and $\mu_{2i-1} = \mu_{2i} = \lambda_i$ for $i = 1, \ldots, N$.

%\bigskip
\textit{Remark:} from this it seems that the exponent in the right side of
the display before \cite[][(4.11)]{Baik2016} should read $N G_H(z)$ and not
$\frac{N}{2} G_H(z)$.

\subsection{\textcolor{black}{Lemma used in the derivation of \eqref{eq:p-decomp}}}\label{lem:conditional-unconditional}
\textcolor{black}{
The following lemma is used in the derivation of \eqref{eq:p-decomp} to show that $\Delta_N=O_{\Pr}(N^{-1/3})$.}
\begin{lemma}
 \textcolor{black}{Let $\{X_N\}$ and $\{ \mathcal{F}_N \}$ be a sequence of random variables and
  sigma-fields respectively. Each of the following are sufficient
  conditions for $X_N = O_\Pr(1)$: \\
  (i) $\E ( |X_N| \big| \mathcal{F}_N) = O_\Pr(1)$, or
  (ii) $\E ( \left.X_N \right| \mathcal{F}_N) = 0$ and $Var (\left.X_N \right| \mathcal{F}_N) = O_\Pr(1)$.}
\end{lemma}
\begin{proof}
  \textcolor{black}{Let $\epsilon >0$ be fixed. For case (i), choose $M_Y$ so that
  $\Omega = \{ \E ( |X_N| \big| \mathcal{F}_N) \leq M_Y \}$ has probability at
  least $1 - \epsilon$ for $N > N_0(\epsilon)$. Then, since $\Omega
  \in \mathcal{F}_N$,
  \begin{align}
    \Pr (|X_N| > M)
      & \leq \E \{ \mathbb{I}(\Omega), \Pr[ |X_N|>M \big| \mathcal{F}_N ]\} +
        \epsilon \label{eq:decomp} \\
      & \leq M_Y/M + \epsilon \leq 2 \epsilon  \notag
  \end{align}
  for $N > N_0(\epsilon)$ using Markov's inequality and setting $M =
  M_Y/\epsilon$.}

  \textcolor{black}{For case (ii), again use (\ref{eq:decomp}), now with
  $\Omega = \{ \Var ( X_N \big| \mathcal{F}_N) \leq M_Y \}$ and appropriate
  $N_0(\epsilon)$. Now from Chebychev's inequality, and with
  $M = \sqrt{M_Y/\epsilon}$,
  \begin{equation*}
    \Pr (|X_N| > M)
    \leq \E \{ \mathbb{I}(\Omega) \Var(X_N \big| \mathcal{F}) \} /M^2+\epsilon
    \leq M_Y/M^2 + \epsilon \leq 2 \epsilon. \qedhere
  \end{equation*}}
\end{proof}

\subsection{Proof of Lemma \ref{lem:OP1}}
\label{sec:OP1proof}

\begin{proof}
  Fix $\epsilon > 0$, and note the decompositions, for $M, \delta > 0$
  to be chosen,
  \begin{align*}
    \Pr\Big( \Big|\frac{Z_N}{X_N-Y_N} \Big| > M \Big)
    & \leq \Pr ( |Z_N| > M \delta) + \Pr(|X_N-Y_N| < \delta) \\
    \Pr (|X_N-Y_N| < \delta)
    & = \E [ \Pr( |X_N-y| < \delta| Y_N = y)]  \leq 2B\sqrt{2\delta}
  \end{align*}
  \textcolor{black}{for large $N$}, 
  with the final inequality using assumptions (ii) and (iii).
  Choose $\delta = \frac{1}{2}\left(\frac{\epsilon}{4B}\right)^2$ and use assumption (i) to yield
  $B_Z = B_Z(\epsilon)$ for which $\Pr( |Z_N| > B_Z) < \epsilon/2$ for
  $N > N(\epsilon)$. Choosing $M = B_Z/\delta$, we obtain
  $\Pr ( |Z_N/(X_N-Y_N)| > M) < \epsilon$ for large $N$, as required.
\end{proof}

\subsection{Check of \texorpdfstring{\eqref{johnstone_ma1}}{}}
\label{Johnstone Ma derivation} \label{one-point-tail}

This follows from results of Tracy and Widom
on representation and scaling of the GUE kernel.
In the notation of \cite{trwi96}, $S_{N}(x, y)$ denotes the kernel of
the GUE scaled to have bulk $ [-\sqrt{2N}, \sqrt{2N}] $.
The kernel is expressed in terms of Hermite functions 
$ \phi_{k}(x) = (\sqrt{\sqrt{\pi}2^k k!})^{-1} e^{-x^2/2} H_{k}(x)
$, with $H_{k}(x)$ being the Hermite polynomials w.r.t. to the weight
$ e^{-x^2}$.
Let $\tau_N = 2^{-1/2} N^{-1/6}$. \cite{trwi96} note that in the
scaling
%\fix{(it seems the additional scaling $(2N)^{1/4}$ is missing)}
\begin{equation*}
  \varphi_\tau(s)
  = (N/2)^{1/4} \phi_N(\sqrt{2N} + \tau_N s)  \qquad
  \psi_\tau(s)
  = (N/2)^{1/4} \phi_{N-1}(\sqrt{2N} + \tau_N s),
\end{equation*}
 the classical Plancherel-Rotach asymptotics
for Hermite polynomials, e.g. \cite[][eq. (8.22.14)]{szeg67}, yields
convergence of $N^{-1/6} \varphi_\tau(s)$ and $ N^{-1/6} \psi_\tau(s)$
to the Airy function $\Ai(s)$, and of most relevance here, with estimates
\begin{equation}
  \label{eq:expbds}
  N^{-1/6} \varphi_\tau(s), \quad N^{-1/6} \psi_\tau(s)  = O(e^{-s}),
\end{equation}
uniformly in $N$ and for $s$ bounded below, see also 
\cite[][p. 403]{olve74}.
\cite{trwi96} also give an integral representation of $S_N$, which in the
scaling $S_\tau(s,t) = \tau_N S_N(\sqrt{2N} + \tau_Ns,\sqrt{2N} + \tau_Nt)$
becomes
\begin{equation*}
  S_\tau(s,s)
  = N^{-1/3} \int_0^\infty \varphi_\tau(s+u) \psi_\tau(s+u) \diff u
  \leq C(\gamma) e^{-2s}
\end{equation*}
for $s \geq \gamma$ in view of \eqref{eq:expbds}.
In our notation, with $x = 2 + N^{-2/3} s$, we have
$\sqrt{N/2} \, x = \sqrt{2N} + \tau_N s$ and so
\begin{equation*}
  \rho_N(x)
  = \frac{1}{\sqrt{2N}} S_N(\sqrt{N/2} \, x, \sqrt{N/2} \, x)
      = \frac{1}{\tau_N \sqrt{2N}} S_\tau(s,s).
  % = \frac{1}{\sqrt{2N}} S_N(\sqrt{\frac{N}{2}} x, \sqrt{\frac{N}{2}} x)
  %     = \frac{1}{\sqrt{2N} \tau_N} S_\tau(s,s).
\end{equation*}
Noting that $\tau_N \sqrt{2N} = N^{1/3}$, we recover \eqref{johnstone_ma1}.

\subsection{Proof of lemma~\ref{lemma_main_eigenvector}}
\label{section_main_eigenvector}

Recall the tridiagonal form of the matrix $M_N$ \eqref{tridiag_M_N}. \textcolor{black}{Denote $a_i/\sqrt{N}$ and $b_i/\sqrt{N}$ as $\tilde{a}_i$ and $\tilde{b}_i$, respectively.}
Since the main eigenvector \( v \) satisfies $ M_N v = \lambda_1 v $, we have 
\[
    v_{1} (\tilde{a}_{1} - \lambda_1) + \tilde{b}_{1} v_{2} = 0,
    \qquad
    v_{i-1} \tilde{b}_{i-1} + v_{i} (\tilde{a}_{i} - \lambda_1) + v_{i + 1} \tilde{b}_{i} = 0 ,
\]
for all $i = 2, \dots, N-1$.
On the event \( \{\tilde{b}_i > 0, i=1,...,N-1\}\), which happens with probability \(1\), we have \( v_1 \neq 0 \), since otherwise all \( v_{i} = 0 \). Hence, we can consider a re-normalization with \( v_{1} = 1\) \textcolor{black}{(so that the statement of \cref{lemma_main_eigenvector} should be re-formulated with $v_i/\norm{v}$ replacing $v_i$)}. We have then
\[
    v_{2} = \frac{\lambda_1 - \tilde{a}_1}{\tilde{b}_1},
    \qquad
    v_{i + 1} = \frac{\lambda_1 - \tilde{a}_i}{\tilde{b}_i} v_{i} - \frac{\tilde{b}_{i-1}}{\tilde{b}_{i}} v_{i-1},
    \qquad
    i \geq 2 \, .
\]

\def\rt{\tilde{r}}
It will be convenient to reformulate the recursion in terms of new variables. Set, 
\[
    u_{i} = \frac{v_{i}}{\prod_{j = 1}^{i - 1} \rt_j \tilde{b}_{j}^{-1}} \, ,
    \qquad
    \rt_j = 1 + \sqrt{1 - \frac{j - 1}{N}} \, .
\]
Notice that
$\tilde{r}_j$ can be thought of as a special case of $r_j$ (see definition \eqref{r_definition}) with \( \theta_N = 1 \).
Since \(\tilde{r}_1 = 2\), we have
\[
    u_{1} = 1,
    \qquad
    u_{2} = \frac{\frac{\lambda_1 - \tilde{a}_1}{\tilde{b}_1}}{ \tilde{r}_1 \tilde{b}_{1}^{-1}} = \frac{\lambda_1}{2} - \frac{\tilde{a}_1}{2} = 1 + O_{\Pr}(N^{-1/2}) \, ,
\]
and for \( i = 2, \dots, N-1\),
\begin{equation}\label{u recursion}
    u_{i + 1} = \frac{\lambda_1 - \tilde{a}_i}{\rt_{i}} u_{i} - \frac{\tilde{b}_{i-1}^2}{\rt_{i-1} \rt_{i}} u_{i-1} \, .
\end{equation}
In the next lemma we control the fluctuations of ratios \( u_{i+1} / u_{i} \). Denote, \( R_i = u_{i + 1} / u_{i} - 1 \).

\begin{lemma}\label{R_i_bound}
We have
\[
    \max_{i \leq N - N^{1/3} \log^3 N} |R_{i}| = o_{\Pr}(N^{-1/3}) \, .
\]
\end{lemma} 

The proof of this bound repeats one step in the proof of the
log-determinant CLT from \cite{johnstone2020logarithmic}. For the sake
of completeness, we reproduce it in section \ref{proof_R_i_bound} below. We are now
ready to finish the proof of \cref{lemma_main_eigenvector}.

%\begin{proof}[Proof of \cref{lemma_main_eigenvector}]
We will show our bound on the following event,
\[
    \mathcal{E}=\left\{\max_{i \leq N - N^{1/3} \log^{3} N} |R_{i}| = o(N^{-1/3}),
    \qquad
    \max_{i\leq N} \tilde{b}_{i} \leq 1 + 2\sqrt{\frac{\log N}{N}}\right\}.
\]
where the bound on $\max \tilde{b}_i$ holds \textcolor{black}{a.a.s.}~due to the concentration of chi-squared variables \cite[theorem~2.3]{boucheron2013concentration}.
%\[
%    \Pr(\chi(d) > \sqrt{d} + \sqrt{2t}) \leq e^{-t}, \qquad \forall t > 0 .
%\]
Namely, we show that on the event $\mathcal{E}$ for any \( D > 0\),
\[
    \max_{i \leq N - 2 N^{1/3} \log^3 N} \left| \frac{v_i}{v_{i + k}} \right| = O(N^{-D})  ,
    \qquad k = \lceil N^{1/3} \rceil \, ,
\]
which immediately yields the statement of \cref{lemma_main_eigenvector} \textcolor{black}{by noting that $ |v_i| / \| v\| \leq |v_i / v_{i + k}|$.}

First, we have that
\(
    \frac{u_{i}}{u_{i + k}} = \prod_{j = i }^{i + k-1} (1 + R_{j})^{-1} = 1 + o(1) \, 
\)
for all \( i \leq N - 2N^{1/3} \log^3 N\). Then, we have a bound
\begin{equation}\label{vi ratio}
    \left|\frac{v_{i}}{v_{i + k}}\right| = \left|\frac{u_i}{u_{i + k}}\right| \prod_{j = i}^{i + k - 1}  \frac{b_{j}}{ \rt_{j}} \leq (1+o(1))\prod_{j = i}^{i + k - 1}  \frac{1 + 2 \sqrt{\frac{\log N}{N}}}{\rt_{j}} \,.
\end{equation}
Each \( j \) in the range $[i, i + k-1]$ satisfies \( j \leq N - N^{1/3} \log^{3} N\), where we have a lower-bound
\[
    \rt_{j} = 1 + \sqrt{1 - \frac{j-1}{N}} \geq 1 + \sqrt{N^{-2/3} \log^{3} N} = 1 + N^{-1/3} \log^{3/2} N \, .
\]
It remains to plug this bound into \cref{vi ratio}, so we obtain
\begin{align*}
    \left|\frac{v_{i}}{v_{i + k}}\right| &\leq (1+o(1))\left( \frac{1 + 2\sqrt{N^{-1} \log N}}{1 + N^{-1/3} \log^{3/2} N } \right)^{k}%\\ \leq \left( \frac{1 + n^{-1/2} \log n}{1 + n^{-1/3} \log^{3/2} } \right)^{n^{1/3}}
    = (1+o(1))\left( 1 - N^{-1/3} \log^{3/2} N  + o(N^{-1/3}) \right)^{k} \\
    &= e^{-\log^{3/2} N + o(1)} = N^{-\log^{1/2} N + o(1)} \, ,
\end{align*}
which is smaller than any \( N^{-D} \) for large enough \(N\). This completes the proof of \cref{lemma_main_eigenvector}. \qed
%\end{proof}

\subsubsection{Proof of Lemma~\ref{R_i_bound}}\label{proof_R_i_bound}
Let \( \tilde{c}_{i}=c_i/\sqrt{N} = (b_{i}^2 - i)/{\sqrt{Ni}} \). We rewrite the recurrence \eqref{u recursion} as follows,
\begin{align*}
    R_{i} &= -1 + \frac{\lambda_1}{\rt_i} - \frac{\tilde{a}_i}{ \rt_i} - \frac{i-1}{N \rt_{i-1} \rt_{i}} \frac{1}{1 + R_{i-1}} - \frac{\sqrt{\frac{i-1}{N}}}{\rt_{i-1} \rt_{i}} \frac{\tilde{c}_{i-1}}{1 + R_{i-1}} \, .
\end{align*}
Denoting \( \mt_i = 1 - \sqrt{1 - (i-1)/N} \) we have \( \mt_i = 2 - \rt_i \) and \( \mt_i \rt_i = \frac{i-1}{N} \). Using \( \frac{1}{1 + R_i} = 1 - R_{i} + \frac{R_i^2}{1 + R_{i}} \), we have a linear expansion
\begin{align*}
    R_{i} =&\; -1 + \frac{\lambda_1}{\rt_i} - \frac{\tilde{a}_i}{\rt_i} - \frac{\mt_{i}}{\rt_{i-1}} (1 - R_{i-1}) - \frac{\mt_{i}}{\rt_{i-1}} \frac{R_{i-1}^2}{1 + R_{i-1}} - \frac{\sqrt{\frac{i-1}{N}}}{\rt_{i-1} \rt_{i}} \frac{\tilde{c}_{i-1}}{1 + R_{i-1}} \\
    =& \; \frac{-\rt_i + \lambda_1 - \frac{\rt_i}{\rt_{i-1}} \mt_{i}}{\rt_{i}} + \frac{\mt_i}{\rt_{i-1}} R_{i-1} + \left( -\frac{\tilde{a}_i}{\rt_{i}} - \frac{\sqrt{\frac{i-1}{N}} \tilde{c}_{i-1}}{ \rt_{i-1} \rt_{i}} \right)\\
    +&\left[ -\frac{\mt_i}{\rt_{i-1}} \frac{R_{i-1}^2}{1 + R_{i-1}} +\frac{\sqrt{\frac{i-1}{N}} \tilde{c}_{i-1}}{\rt_{i-1}\rt_{i}} \frac{R_{i-1}}{1 + R_{i-1}} \right] .
\end{align*}
We can simplify the above expression as follows
\begin{align}\label{R recurrence}
    R_{i} &= \deltat_i + \gt_{i} R_{i-1} + \xit_{i} + \tilde{R}_{i-1}^{(1)},
\end{align}
where we introduce the notation
\begin{align*}
    \tilde{\gamma_i} & = \frac{\mt_i}{\rt_{i-1}}, \qquad
    \tilde{\delta}_i = \frac{\lambda_1 - 2}{\rt_i} - \frac{\rt_i - \rt_{i-1}}{\rt_i \rt_{i-1}} \mt_{i}, \qquad
    \xit_{i} = - \frac{\tilde{a}_i}{ \rt_{i}}  - \sqrt{\frac{i-1}{N}} \frac{\tilde{c}_{i-1}}{\rt_{i-1} \rt_{i}}, \\
    \tilde{R}_{i-1}^{(1)} &= -\frac{\mt_i}{\rt_{i-1}} \frac{R_{i-1}^2}{1 + R_{i-1}} + \sqrt{\frac{i-1}{N}} \frac{ \tilde{c}_{i-1}}{\rt_{i-1}\rt_{i}} \frac{R_{i-1}}{1 + R_{i-1}} \, .
\end{align*}
We iteratively unpack the recurrence equation \eqref{R recurrence} to get
\begin{equation}\label{ev_R_i_decomposition}
\begin{aligned}
    R_{i} = & \deltat_{i} + \gt_{i} R_{i-1} + \xit_{i} + \Rtone_{i-1} \\
    = & \deltat_{i} + \gt_{i} \deltat_{i-1}  + \xit_{i} + \gt_{i} \xit_{i-1} + \Rtone_{i-1} + \gt_{i} \Rtone_{i-2} \\
    = & \dots \\
    = & \tilde{\delta}_{i} + \gt_{i} \tilde{\delta}_{i-1} + \dots + \gt_{i} \dots \gt_{3} \deltat_{2} \\
    &+ \xit_{i} + \gt_{i} \xit_{i-1} + \dots + \gt_{i} \dots \gt_{3} \xit_{2} \\
    & + \Rtone_{i-1} + \gt_{i} \Rtone_{i-2} + \dots + \gt_{i} \dots \gt_{3} \Rtone_{1} \\
    & + \gt_{i} \dots \gt_{2} R_{1} \, .
\end{aligned}
\end{equation}
Since \( \gt_{i} < 1 \) is an increasing sequence, we have
\[
    1 + \gt_{i} + \dots + \gt_{i} \dots \gt_{3} \leq \frac{1}{1 - \gt_{i}} = \frac{\rt_{i-1}}{\rt_{i-1} - \mt_{i}} < \frac{2}{1 - \rt_{i}} \, .
\]
Similarly to \textcolor{black}{lemma~{10} in \cite{johnstone2020logarithmic}}, %\label{Lemma L0}
we see that $\tilde{\xi}_{i} + \gt_{i} \xit_{i-1} + \dots + \gt_{i} \dots \gt_{3} \xit_{2}$ are sub-gamma random variables, which implies that for some constant \( K_1 > 0 \), with probability at least \( 1 - 1/N\), uniformly over \( i = 2, \dots, N\),
\begin{equation}\label{ev_event_1}
    |\xit_{i} + \gt_{i} \xit_{i-1} + \dots + \gt_{i} \dots \gt_{3} \xit_{2}| \leq K_1 \left( \sqrt{\frac{\log N}{N(\tilde{r}_i -1)}} + \frac{\log N}{N} \right) \, .
\end{equation}
We also have
\begin{align*}
    |\deltat_{i}| &\leq |\lambda_1 - 2| + \left(\sqrt{1 - \frac{i-2}{N}} - \sqrt{1 - \frac{i-1}{N}}\right)\left(1 - \sqrt{1 - \frac{i-1}{N}}\right) \\
    &\leq |\lambda_1 - 2| + \frac{1}{N} \frac{ \frac{i-1}{N}}{\sqrt{1 - \frac{i-1}{N}}} \, .
\end{align*}
The function \( \frac{x}{\sqrt{1- x}} \) has the derivative \( \frac{2 - x}{2(1 - x)^{3/2}}\) and therefore is increasing on \( (0,1) \). Thus, we have for all \( i \leq N - N^{1/3} \),
\begin{equation}\label{ev_delta_bound}
    |\deltat_{i}| \leq |2 - \lambda_1| + N^{-2/3} \, .
\end{equation}
In addition, since $\tilde{c}_i$ is a centered and scaled $\chi^2$ random variable, we have for some constant \( K_2 > 0\) with probability at least \(1 - 1/N\)
\begin{equation}\label{ev_event_2}
    \max_{i} |\tilde{c}_{i}| \leq K_2 \frac{\log N}{\fix{\sqrt{N}}} \,.
\end{equation}
Further, \( R_1 = \frac{\lambda_1 - 2}{2} - \frac{\tilde{a}_1}{2} \), hence for some $K_3>0$ with probability at least \(1 - 1/N\), we have for all \(i \),
\begin{equation}\label{ev_event_3}
    |\gamma_{i} \dots \gamma_{2} R_{1}| \leq K_3 \frac{\log N}{\sqrt{N}}\, .
\end{equation}
Finally, by the Tracy-Widom law (e.g., theorem~{4.5.42} in \cite{anderson2010introduction}) we have that for any \( \epsilon > 0 \) there is \( C_{\epsilon}\) such that for large enough \( N \),
\[
    \Pr( |\lambda_1 - 2| > C_{\epsilon} N^{-2/3}) \leq \epsilon \, .
\]

Now, consider the event 
\[
    \mathcal{E}(\epsilon) = \left\{ |\lambda_1 - 2 | \leq C_{\epsilon / 2}N^{-2/3} \text{ and \eqref{ev_event_1}, \eqref{ev_event_2},  \eqref{ev_event_3} hold} \right\},
\]
so that for large enough \(N\), \( \Pr(\mathcal{E}(\epsilon)) \geq 1 - \epsilon \). We will show by induction that on this event, for all \( i \leq N - N^{1/3} \log^{3} N\),
\[
    |R_{i}| \leq \frac{N^{-1/3}}{\log^{1/5} N} \, .
\]

The base holds due to \eqref{ev_event_3}. Suppose that \( \max_{j \leq i-1} |R_{i}| \leq \frac{N^{-1/3}}{\log^{1/5} N} \), which is at most \( 1/ 2\) for large enough \(N\). Then, by \eqref{ev_event_2} we have for \( j \leq i-1\)
\begin{align*}
    |\Rtone_{j}| &\leq 2 \frac{N^{-2/3}}{\log^{2/5} N} + \frac{2 K_2 \log N}{\sqrt{N}} \times \frac{N^{-1/3}}{\log^{1/5} N} = \frac{N^{-2/3}}{\log^{2/5} N} \left( 2 + 2 K_2 \frac{\log^{6/5} N}{N^{1/6}} \right) \\
    &\leq 3 \frac{N^{-2/3}}{\log^{2/5} N},
\end{align*}
for large enough \(N\). Further, for all \( i \leq N - N^{1/3} \log^{3} N\), we have
\[
    \frac{1}{\rt_{i}-1} \leq \frac{N^{1/3}}{\log^{3/2} N} \, .
\]
Therefore,
\begin{align*}
    |\Rtone_{i-1} + \gt_{i} \Rtone_{i-2} + \dots + \gt_{i} \dots \gt_{3} \Rtone_{1}| &\leq \frac{2}{\rt_i-1} 
    \times 3 N^{-2/3} \log^{-2/5} N \\
    &\leq \frac{6}{\log^{3/2} N} \frac{N^{-1/3}}{\log^{1/5} N} \, . 
\end{align*}
By \eqref{ev_delta_bound} and \( |\lambda_1 - 2| \leq  N^{-2/3} \log^{2/3} N\) we have for all \( i = 3, \dots, N\),
\begin{align*}
    \tilde{\delta}_{i} + \gt_{i} \tilde{\delta}_{i-1} + \dots + \gt_{i} \dots \gt_{3} \deltat_{2} \leq \frac{4}{\rt_i-1} N^{-2/3} \log^{2/3} N &\leq \frac{4}{\log^{19/30} N} \frac{N^{-1/3}}{\log^{1/5} N} \, .
\end{align*}
Finally, from \eqref{ev_event_1} we get for all \( i = 1, \dots, N - N^{1/3} \log^3 N\),
\[
    |\xit_{i} + \gt_{i} \xit_{i-1} + \dots + \gt_{i} \dots \gt_{3} \xit_{2}| \leq K_1 \left( \sqrt{\frac{\log N}{N(\rt_i-1)}} + \frac{\log N}{N} \right) \leq 2K_1 \frac{N^{-1/3}}{\log^{1/4} N} \, .
\]

From decomposition \eqref{ev_R_i_decomposition}, we therefore obtain for large enough \(N\),
\begin{align*}
    |R_i| &\leq \frac{N^{-1/3}}{\log^{1/5} N} \left( \frac{4}{\log^{19/30} N} + \frac{2K_1}{\log^{1/20} N} + \frac{6}{\log^{3/2} N} + K_{3} \frac{\log^{6/5} N}{N^{1/6}}   \right) \\
    &\leq \frac{N^{-1/3}}{\log^{1/5} N} \, ,
\end{align*}
which proves the induction step, and the lemma follows. \qed

\subsection{Proof of \cref{prop:sticky}}\label{proof:sticky}   

The proof is a direct consequence of the following \emph{isotropic local law} and \emph{isotropic delocalization} results due to Knowles and Yin \cite{Knowles2013b}:

%\textit{Isotropic local law.} 
\begin{proposition}
\label{thm: isotropic}
Let $W_N'$ be a Wigner matrix satisfying \textcolor{black}{assumption W. Let $R(z)=(W_N'-zI)^{-1}$ and, for any $\tau>0$, let
\[
\mathbf{S}(\tau)=\left\{z:=E+\im \eta:|E|<\tau^{-1},N^{-1+\tau}\leq \eta\leq \tau^{-1}\right\}.
\]}
% satisfying conditions \textbf{W1}, \textbf{W2}, and \=textbf{W3}. In
% addition, suppose that the off-diagonal entries of $W_N$ have zero
% third moment in the sense that
% $\E\xi_{ij}^3=\E\xi_{ij}^2\overline{\xi_{ij}}=0$ for any $i\neq j$.
We have:
\newline
(i) (isotropic local law) Fix $\tau>0$. Then for each $\epsilon > 0$, we have
    w.o.p.
\begin{equation}
    \label{isotropic law}
    \mathbf{v}^\ast R(z)\mathbf{w}=m_{SC}(z)\mathbf{v}^\ast\mathbf{w}+O (N^{\varepsilon}\Psi(z)),
    \qquad
    \Psi(z) := \sqrt{\frac{\Im m_{SC}(z)}{N \eta }} + \frac{1}{N\eta}
\end{equation}
uniformly for $z \in \mathbf{S}(\tau)$ and
% normalized (to have unit Euclidean norm)
for any two deterministic vectors $\mathbf{v,w}$ of unit Euclidean
length in $\mathbb{C}^N$.
% Outside the spectrum, we have w.o.p. 
% \begin{equation}
%   \label{eq:outside-spec}
%   \mathbf{v}^\ast R(z)\mathbf{w}=s_{sc}(z)\mathbf{v}^\ast\mathbf{w}
%   +O(N^{\varepsilon-1/2}(\kappa + \eta)^{-1/4})
% \end{equation}
% uniformly for $z \in \mathbf{S}^\circ(\tau)\equiv \{ z \in \mathbb{C}: 2 \leq |E| \leq
%   \tau^{-1}, \   \eta > 0 \}$ and
%   normalized $\bv, \mathbf{w}$, where $\kappa=\big||E|-2\big|$.
\newline
(ii) (isotropic delocalization) Let $\mathbf{u}^{(j)}$ be the $j$-th principal normalized eigenvector of $W_N'$. Then, for each $\varepsilon>0$, we have w.o.p.
\begin{equation}
    \label{isotropic delocalization}
    \max_{j}|\mathbf{v}^\ast \mathbf{u}^{(j)}|^2=O(N^{\varepsilon-1})
\end{equation}
uniformly for normalized deterministic vectors $\mathbf{v}\in\mathbb{C}^N$.
\end{proposition}

\textcolor{black}{Knowles and Yin's isotropic local law and isotropic delocalization results require matching second moments on the diagonal of $W_N'$. They also use a slightly different definition of Wigner matrices from ours. A modification of their proof, accommodating our setting, is given in section~B.3 of \cite{johnstone2020logarithmic}.}

\textcolor{black}{As explained in \cite{Knowles2013b}, the isotropic local law can be strengthened ``outside the spectrum'' as follows (a proof is almost identical to the proof of theorem 2.3 of \cite{Knowles2013b}, so we omit it):
\begin{proposition}
\label{isotropic law outside}
Fix $\Sigma\geq 3$ and let $z=E+\im \eta$. Then for any $\epsilon>0$, any
\[
E\in[-\Sigma,-2-N^{-2/3+\epsilon}]\cup [2+N^{-2/3+\epsilon},\Sigma],
\]
any $\eta\in(0,\Sigma]$, and any deterministic vectors $\mathbf{v,w}$ of unit Euclidean
length in $\mathbb{C}^N$ we have w.o.p.
\begin{equation}
    \label{eq:outside}
    \mathbf{v}^\ast R(z)\mathbf{w}=m_{SC}(z)\mathbf{v}^\ast\mathbf{w}+O \left(N^{\varepsilon}\sqrt{\frac{\Im m_{SC}(z)}{N \eta }}\right).
\end{equation}
\end{proposition}
}
%For completeness, we derive \cref{isotropic law outside} from \cref{thm: isotropic} in \cref{sec:proof of outside} below, closely following Knowles and Yin's derivation of their theorem 2.3.}

To prove \cref{prop:sticky}, we copy part of the argument of
  \cite[][theorem 6.3]{Knowles2013b}, with tweaks in order to replace
  $\zeta$-high 
  probability statements by stochastic domination, and to extend the
  method to the top $k$ eigenvalues $\{ \lambda_j \}_1^k$ in place of
  $\lambda_1$.
Recall that $ \lambda $ is the eigenvalue of $W_{J,N}' $ iff $ \det(W_{J,N}' - \lambda) = 0 $. For $ \lambda \notin \{ \mu_{1}', \dots, \mu_{N}'\} $,
\[
    \det(W_{J,N}' - \lambda) = \det(W_{N}' - \lambda + J \bv \bv^{*}) = \det(W_{N}' - \lambda) (1 + J \bv^* (W_{N}' - \lambda)^{-1} \bv) \, .
\]
Hence $ \lambda \in \{ \lambda_1, \dots, \lambda_N\} $ is equivalent to $ \bv^{*}
R(\lambda) \bv = -{1}/{J} $. 

Fix $\epsilon > 0$ and let $x = 2 + N^{-2/3 + \epsilon}$. \textcolor{black}{By \cite[][lemma 3.2]{Knowles2013b}, 
\begin{equation*}
%\label{eq:imm}
\frac{\Im m_{SC}(x+\im \eta)}{\eta}\asymp (N^{-2/3+\epsilon}+\eta)^{-1/2}.
\end{equation*}
Therefore,} from the isotropic law \textcolor{black}{outside the spectrum}
\eqref{eq:outside}, we have w.o.p.~that
\begin{equation}
  \label{eq:gvxx}
  \bv^{*} R(x) \bv = m_{SC}(x) + O(N^{-1/3 +3 \epsilon/4}),
\end{equation}
while from semicircle law estimates \cite[][(3.3)]{Knowles2013b} we find
\begin{equation}
  \label{eq:1mx}
  1 + m_{SC}(x) \asymp N^{-1/3 + \epsilon/2},
\end{equation}
which together yield $1 + \bv^{*}R(x)\bv \geq 1 - 1/J$ {for sufficiently large $N$}.
% \begin{equation*}
%   1 + G_{\bv \bv}(x) \geq 0 \geq 1 - 1/J.
% \end{equation*}
Since $y \to \bv^{*} R(y) \bv$ is increasing in $(\mu_1',\infty)$ and
$\mu_1' \leq x$ w.o.p., it follows that $\mu_1' \leq \lambda_1 \leq
x$.

Suppose that $ j \leq N^{\epsilon} $ and
 let $q = 2N^\epsilon$, and let us
split the projected resolvent into `edge' and `bulk'
components:
\begin{equation*}
  \bv^{*} R(\lambda) \bv = \bigg( \sum_{\alpha \leq q} + \sum_{\alpha > q}
  \bigg) \frac{|v_\alpha|^2}{\mu_\alpha' - \lambda}
  = R^{\rm e}_{\bv \bv}(\lambda) + R^{\rm b}_{\bv \bv}(\lambda).
\end{equation*}
We first show that the bulk part satisfies $ R^{\mathrm{b}}_{\bv \bv}(\lambda_j) \approx -1$. To do so, we compare it to $\bv^{*} R(x) \bv \approx
-1$ and show that the bulk components
$R^{\rm b}_{\bv \bv}(\lambda_j)$ and $R^{\rm b}_{\bv \bv}(x)$ are close.
Indeed, w.o.p.
\begin{align}
 |R^{\rm b}_{\bv \bv}(\lambda_j) - R^{\rm b}_{\bv \bv}(x)|
   & \leq C N^{-2/3 + \epsilon} \sum_{\alpha > q}
     \frac{|v_\alpha|^2}{(\mu_\alpha' - \lambda_j)^2} \notag \\
   & \leq C N^{-2/3 + 2\epsilon} \bigg[ \sum_{k \geq 1}
     \frac{2^k N^{-1}}{(2^{2k/3} N^{-2/3})^2} + {1} \bigg] \notag
  \\
   & \leq CN^{-1/3 + 2\epsilon}.  \label{eq:bulk-diff}
\end{align}
In the first inequality, we used $ \lambda_j - \mu_{\alpha}' \leq x - \mu_{\alpha}' $ and $|\lambda_j - x| \leq 
N^{-2/3 + \epsilon}$ w.o.p., \textcolor{black}{where the latter fact follows from the eigenvalue rigidity (theorem 2.9 of \cite{BenaychG2018}) and the interlacing theorem.} In the second we 
estimated the contribution of the eigenvalues $\alpha \leq N/2$ using the dyadic
decomposition into the sets
\begin{equation*}
   U_k:= \{ \alpha \in [q, N/2] ~:~ 2^k \leq \alpha \leq 2^{k+1} \},
 \end{equation*}
combined with eigenvalue rigidity 
%\ref{prop:P14analog}(iii)) 
%and the estimate 
%\begin{equation}
%  \label{eq:growth}
%  2 - \gamma_\alpha \asymp (\alpha/N)^{2/3} \qquad \text{for } \qquad
%  \alpha \leq N/2, 
%\end{equation}
 and the delocalization estimate \eqref{isotropic delocalization}. %{In particular, using  $ \mu_j \geq \lambda_j $ we find that  $ \lambda_j - \lambda_{\alpha} \gtrsim (\alpha / N)^{2/3} $ for all $ \alpha \geq q $, w.o.p.}
%A similar (in fact easier) dyadic decomposition works for the
%remaining eigenvalues $\alpha \geq N/2$ and yields the last term of
%the second line.

From eigenvalue rigidity and \textcolor{black}{the estimate of the typical eigenvalue location
\[
 2 - \gamma_\alpha \asymp (\alpha/N)^{2/3} \qquad \text{for } \qquad
 \alpha \leq N/2, 
\]}
we have for $\alpha
\leq q$ that $x - \mu_\alpha' \asymp N^{-2/3}(N^\epsilon +
\alpha^{2/3} + \alpha^{-1/3}) \asymp N^{-2/3 + \epsilon}$, and so
using also delocalization, we have w.o.p.
\begin{equation*}
  |R_{\bv \bv}^{\rm e}(x)|
  \leq \sum_{\alpha \leq q} \frac{|v_\alpha|^2}{|x-\mu_\alpha'|}
  \lesssim  \frac{N^\epsilon}{N}\frac{q}{N^{-2/3 + \epsilon}}
   =2 N^{-1/3 + \epsilon}.
\end{equation*}
Combining \eqref{eq:bulk-diff} with the previous display and then with
\eqref{eq:gvxx} and \eqref{eq:1mx}, we get
\begin{equation*}
  R_{\bv \bv}^{\rm b}(\lambda_j)
  = R_{\bv \bv}^{\rm b}(x) + O(N^{-1/3 + 2 \epsilon})
  = \bv^{*}R(x)\bv + O(N^{-1/3 + 2 \epsilon})
  = -1 + O(N^{-1/3 + 2 \epsilon}) .
\end{equation*}
Consequently,
\begin{equation*}
  R_{\bv \bv}^{\rm e}(\lambda_j)
  = - 1/J - R_{\bv \bv}^{\rm b}(\lambda_j)
  = 1 - 1/J + O(N^{-1/3 + 2\epsilon}).
\end{equation*}
Since $\mu_j' \leq \lambda_j < \mu_{j-1}'$ (with $\lambda_0' = +\infty)$,
\begin{equation*}
  \frac{1}{\lambda_j - \mu_j'} \sum_{\alpha = j}^q |v_\alpha|^2
  \geq \sum_{\alpha = j}^q \frac{|v_\alpha|^2}{\lambda_j-\mu_\alpha'}
  \geq - R_{\bv \bv}^{\rm e}(\lambda_j)
  = \frac{1}{J} - 1 + O(N^{-1/3+2 \epsilon}).
\end{equation*}
Since delocalization implies $\sum_j^q |v_\alpha|^2 \leq q
N^{-1+\epsilon} \leq 2N^{-1+2 \epsilon}$, we find that w.o.p.
\begin{equation*}
  \lambda_j - \mu_j' \lesssim \frac{J}{\textcolor{black}{1-J}} N^{-1+2\epsilon},
\end{equation*}
from which the result follows.
\qed

\color{black}

\subsection{Proof of equation  \eqref{eq:ey17.3}}\label{sec:proof of eq 6.6}
Equation \eqref{eq:ey17.3} follows from lemma 17.3 of \cite{ey17} under a somewhat different definition of Wigner matrices used in that book. Here we prove \eqref{eq:ey17.3}, along the lines of the proof of lemma 17.3 (that, in turn, is based on the proof of lemma 6.1 in \cite{ErdosYY2012}), accommodating our definition \ref{Wigner definition}.

Recall that $E_\infty = 2 + 2N^{-2/3 + \epsilon}$. Define $\chi_E(x):=\mathbf{1}_{[E,E_\infty]}(x)$ and note that $\mathcal{N}_W(E,E_\infty)=\tr\chi_E(W)$, where the left hand side denotes the number of eigenvalues of $W$ in $[E,E_\infty]$.
Now, our definition of Wigner matrices coincides with that of \cite{BenaychG2018}. By theorem 2.9 (rigidity of eigenvalues) of that paper, we have w.o.p.
\begin{equation}
    \label{eq:compact vs infty}
    \mathcal{N}_W(E,\infty)=\mathcal{N}_W(E,E_\infty)=\tr\chi_E(W)
\end{equation}
for a Wigner matrix $W$.

Next, let us approximate $\tr\chi_E(W)$ by its smoothed version $\tr[\chi_E\star\theta_\eta](W)$, where $\theta_\eta(x)=\frac{1}{\pi}\frac{\eta}{x^2+\eta^2}$. Notice
\begin{equation}
    \label{eq:smoothed}
\tr[\chi_E\star\theta_\eta](W)=\frac{N}{\pi}\int_E^{E_\infty} \Im s_W(y+\im \eta)\diff y =g(W,E).
\end{equation}
Let $d=d(x):=|x-E|+\eta$,  $d_\infty=d_\infty(x):=|x-E_\infty|+\eta$, 
% Then, an elementary calculation yields (see (6.10) in \cite{ErdosYY2012})
%\[
%|\chi_E(x)-\chi_E\star\theta_\eta(x)|\leq C\eta\left[\frac{E_\infty-E}{d_\infty(x)d(x)}+\frac{\chi_E(x)}{d_\infty(x)+d(x)}\right]
%\]
%for some constant $C>0$. 
$\ell_1:=N^{-2/3-3\epsilon}$, and $\eta=N^{-2/3-9\epsilon}$.
% Then, for any $E$ such that $|E-2|\leq N^{-2/3+\epsilon}$, $E_\infty-E\gg l_1$ and it is easy to see that
% \[
% |\chi_E(x)-\chi_E\star\theta_\eta(x)|\leq \begin{cases}
%   C\eta/l_1=O(N^{-6\epsilon})& \text{ if }\min\{d,d_\infty\}\geq l_1\\
%   O(1)&\text{ if } \min\{d,d_\infty\}<l_1
%\end{cases}.
% \]
Then elementary calculations yield (see discussion in \cite{ErdosYY2012} between equations (6.9) and (6.11)), for any $E$ such that $|E-2|\leq \frac{3}{2}N^{-2/3+\epsilon}$,
\begin{equation}
\label{eq:ey12 6.11}
|\tr\chi_E(W)-\tr[\chi_E\star\theta_\eta](W)|\leq C\left(\tr f(W)+ \frac{\eta}{\ell_1}\mathcal{N}(E,E_\infty)+\mathcal{N}(E-\ell_1,E+\ell_1)+\mathcal{N}(E_\infty-\ell_1,\infty)\right)
\end{equation}
for some constant $C>0$, where
\[
f(x):=\frac{\eta(E_\infty-E)}{d_\infty(x)d(x)}\mathbf{1}\{x\leq E-\ell_1\}.
\]

Now note that if $y$ is such that
\[
N\int_y^2 p_{SC}(x)\diff x=N^{2\epsilon},
\]
then $2-y\gtrsim N^{-2/3+4\epsilon/3}>2-E$. Therefore, the eigenvalue rigidity (theorem 2.9 of \cite{BenaychG2018}) yields 
\[
\mathcal{N}(E,E_\infty)\leq N^{2\epsilon}\text{ w.o.p.}
\]
The same rigidity result implies that 
\[
\mathcal{N}(E_\infty-\ell_1,\infty)=0\text{ w.o.p.}
\]
Using the latter two displays in \eqref{eq:ey12 6.11}, we obtain w.o.p.
\begin{equation}
\label{eq:ey12 6.13}
|\tr\chi_E(W)-\tr[\chi_E\star\theta_\eta](W)|\leq C\left(\tr f(W)+\mathcal{N}(E-\ell_1,E+\ell_1)+N^{-4\epsilon}\right),
\end{equation}
when $W$ is a Wigner matrix.

Next, the arguments of \cite{ErdosYY2012} which lead them to their equation (6.17) yield in our case
\[
\tr f(W)\leq CN\eta(E_\infty-E)\int\frac{1}{y^2+\ell_1^2}\Im s_W(E-y+\im \ell_1)\diff y.
\]
On the other hand, by the local law for Wigner matrices (theorem 2.6 of \cite{BenaychG2018}),
\[
\Im s_W(E-y+\im \ell_1)\leq \Im m_{SC}(E-y+\im \ell_1)+\frac{N^\epsilon/2}{N\ell_1},
\]
w.o.p. uniformly for $|E-y|$ bounded by any large constant. Since w.o.p.
\[
\sup_{|y|>10}|s_W(E-y+\im \ell_1)|=O(1)\qquad\text{and}\qquad \sup_{|y|>10}|m_{SC}(E-y+\im \ell_1)|=O(1),
\]
while $N\eta(E_\infty-E)<N^{-1/3}$ for sufficiently large $N$, we have
\begin{equation}
\label{eq:ey12 6.17}
\tr f(W)\leq CN\eta(E_\infty-E)\int\frac{1}{y^2+\ell_1^2}\left[\Im m_{SC}(E-y+\im \ell_1)+\frac{N^{\epsilon/2}}{N\ell_1}\right]\diff y+CN^{-1/3},
\end{equation}
w.o.p.

Clearly,
\[
\int\frac{1}{y^2+\ell_1^2}\frac{1}{\ell_1}\diff y\leq \frac{C}{\ell_1^2}.
\]
Using this in \eqref{eq:ey12 6.17}, we obtain
\[
\tr f(W)\leq CN\eta(E_\infty-E)\int\frac{1}{y^2+\ell_1^2}\Im m_{SC}(E-y+\im \ell_1)\diff y+CN^{-3\epsilon/2}.
\]
The arguments of the proof of lemma 6.1 in \cite{ErdosYY2012} imply that the first term on the right hand side of the latter inequality is bounded by $CN^{-2\epsilon}$. Hence overall,
\[
\tr f(W)\leq C N^{-3\epsilon/2}
\]
w.o.p.  Recalling \eqref{eq:ey12 6.13}, we obtain, w.o.p.
\[
|\tr\chi_E(W)-\tr[\chi_E\star\theta_\eta](W)|\leq C\left(N^{-3\epsilon/2}+\mathcal{N}(E-\ell_1,E+\ell_1)\right).
\]

Next, let $\ell=\frac{1}{2}N^{-2/3-\epsilon}$. Similarly to the proof of corollary 6.2 in \cite{ErdosYY2012}, we obtain, for any $E$ such that $|E-2|\leq N^{-2/3+\epsilon}$ w.o.p.
\[
\tr \chi_E(W)\leq \tr\chi_{E-\ell}\star\theta_\eta(W)+CN^{-3\epsilon/2}+C\frac{\ell_1}{\ell}\mathcal{N}(E-2\ell,E+\ell).
\]
From the semicircle law on small scales (theorem 2.8 in \cite{BenaychG2018}), we obtain, w.o.p.
\[
\mathcal{N}(E-2\ell,E+\ell)\leq N\int_{E-2\ell}^{E+\ell}p_{SC}(x)\diff x+N^{\epsilon/2}.
\]
Directly evaluating the integral (recalling that $E$ is in the vicinity of $2$), we obtain
\[
\int_{E-2\ell}^{E+\ell}p_{SC}(x)\diff x\leq C \ell \sqrt{N^{-2/3+\epsilon}}\leq C N^{-1-\epsilon/2}.
\]
This taken with $\ell_1/\ell=2N^{-2\epsilon}$ yield
\[
\frac{\ell_1}{\ell}\mathcal{N}(E-2\ell,E+\ell)\leq CN^{-3\epsilon/2},
\]
and hence, w.o.p.
\[
\tr\chi_E(W)\leq \tr\chi_{E-\ell}\star\theta_\eta(W)+CN^{-3\epsilon/2},
\]
which, for sufficiently large $N$, implies a cruder inequality
\[
\tr\chi_E(W)\leq \tr\chi_{E-\ell}\star\theta_\eta(W)+N^{-\epsilon}.
\]
A lower bound can be established similarly. The bounds and equations \eqref{eq:compact vs infty}, \eqref{eq:smoothed} yield equation \eqref{eq:ey17.3}. $\qed$
\color{black}

%\fix{We collect for later use some elementary criteria for convergence in
%probability of a sequence of random variables $\{X_N\}$.

%\textbf{C.1} If for each $c$ large, $X_N = Y_{N1}(c) + Y_{N2}(c)$ with
%$Y_{N1}(c) = o_{\Pr }(1)$ and $\Var Y_{N2}(c) \leq 1/c^2$ for $N >
%N(c)$, then $X_N = o_{\Pr }(1)$.

%\textbf{C.2} If for each $\epsilon$ small there exist events
%$\mathcal{E}_{N,\epsilon}$ of probability at least $1 - \epsilon$ for
%$N > N(\epsilon)$ such that on $\mathcal{E}_{N,\epsilon}$ we have
%$X_N = Y_{N1}(\epsilon) + Y_{N2}(\epsilon)$ with
%$Y_{Nk}(\epsilon) = O_{\Pr }(1)$, then $X_N = O_{\Pr }(1)$.}

%\printbibliography

\end{document}